\documentclass[12pt,reqno,a4paper]{article}
\usepackage[english]{babel}
\usepackage{amsmath}
\usepackage{amsfonts}
\usepackage{amsthm}
\usepackage{amssymb}

\usepackage{xspace}
\usepackage{euscript}
\usepackage{graphicx}
\usepackage{amscd}
\usepackage{tabularx}

\usepackage{enumerate}
 \usepackage{epsfig} 
 \usepackage{graphics} 
\theoremstyle{plain}
\newtheorem{theo}{Theorem}[section]
\newtheorem{lem}[theo]{Lemma}
\newtheorem{prop}[theo]{Proposition}
\theoremstyle{definition}
\newtheorem{definition}[theo]{Definition}
\theoremstyle{remark}
\newtheorem{rem}[theo]{Remark}
\numberwithin{equation}{section}

\newcommand{\C}{\mathbb{C}}

\newcommand{\R}{\mathbb{R}}
\newcommand{\N}{\mathbb{N}}
\newcommand{\M}{\mathbb{M}}
\newcommand{\divrg}{\textrm{div}\,}

\title{Stable determination of an inclusion in an elastic body by boundary measurements
\thanks{The first and the fourth authors are supported by Universit\`a degli Studi di Trieste FRA 2012 `Problemi Inversi',
the second author is supported by MIUR, PRIN grant no. 2008ZPC95C\_002,
the third author is partially supported by the Carlos III
University of Madrid-Banco de Santander Chairs of Excellence
Programme for the 2013-2014 Academic Year.}
}
\author{Giovanni Alessandrini\thanks{Dipartimento di Matematica e Geoscienze,
Universit\`a degli Studi di Trieste, via Valerio 12/1, 34127
Trieste, Italy. E-mail: \textsf{alessang@units.it}}, Michele Di
Cristo\thanks{Dipartimento di Matematica Francesco Brioschi,
Politecnico di Milano, Piazza Leonardo da Vinci 32, 20133 Milano,
Italy. E-mail: \textsf{michele.dicristo@polimi.it}}, \\ Antonino
Morassi\thanks{Dipartimento di Ingegneria Civile e Architettura,
Universit\`a degli Studi di Udine, via Cotonificio 114, 33100
Udine, Italy. E-mail: \textsf{antonino.morassi@uniud.it}} \  and
Edi Rosset\thanks{Dipartimento di Matematica e Geoscienze,
Universit\`a degli Studi di Trieste, via Valerio 12/1, 34127
Trieste, Italy. E-mail: \textsf{rossedi@univ.trieste.it}}}

\date{}

\begin{document}

\maketitle

\begin{abstract}

We consider the inverse problem of identifying an unknown
inclusion contained in an elastic body by the
Dirichlet-to-Neumann map. The body is made by linearly elastic,
homogeneous and isotropic material. The Lam\'{e} moduli of the
inclusion are constant and different {}from those of the
surrounding material. Under mild a-priori regularity assumptions
on the unknown defect, we establish a logarithmic stability
estimate.
Main tools are
propagation of smallness arguments based on three-spheres
inequality for solutions to the Lam\'{e} system and a refined asymptotic analysis
 of the fundamental solution of the Lam\'{e} system
in presence of an inclusion which shows surprising features.

\end{abstract}

\centerline{}

\section{Introduction}
\label{SecIntroduction}

This paper deals with the inverse problem of determining an
elastic inclusion $D$ contained in an elastic body $\Omega$ by
measuring displacements and tractions at the boundary $\partial
\Omega$. More precisely, let $\Omega$ be a bounded domain in $
\R^3$ and let $D$ be an open set contained in $\Omega$.
Assume that both the body $\Omega$ and the inclusion $D$ are made
by different homogeneous, isotropic, elastic materials, with
Lam\'{e} moduli $\mu$, $\lambda$ and $\mu^I$, $\lambda^I$,
respectively, satisfying the strong convexity conditions $\mu>0$,
$2\mu+3\lambda>0$, $\mu^I>0$, $2\mu^I+3\lambda^I>0$. For a given
$g \in H^{ \frac{1}{2}}(\partial \Omega)$, consider the weak
solution $u\in H^1(\Omega)$ to the Dirichlet problem
\begin{center}
\( {\displaystyle \left\{
\begin{array}{lr}
  \divrg((\C + (\C^I - \C)\chi_D)\nabla u)=0,
  & \hbox{in}\ \Omega,
    \vspace{0.25em}\\
  u=g, & \hbox{on}\ \partial \Omega,
\end{array}
\right. } \) \vskip -4.4em
\begin{eqnarray}
& & \label{eq:Intro.1}\\
& & \label{eq:Intro.2}
\end{eqnarray}
\end{center}
where $\C$, $\C^I$ are the elastic tensors of the body and of the
inclusion, respectively, and $\chi_D$ is the characteristic
function of $D$. We denote by $\Lambda_D : H^{ \frac{1}{2}}
\rightarrow H^{ - \frac{1}{2}}$ the Dirichlet-to-Neumann map
associated to the problem \eqref{eq:Intro.1}--\eqref{eq:Intro.2},
that is the operator which maps the Dirichlet data $u|_{\partial
\Omega}$ onto the corresponding Neumann data $((\C + (\C^I - \C)\chi_D)\nabla u)\nu$, taken in the weak sense (see \eqref{eq:3.4}, \eqref{eq:3.4bis} below),
where $\nu$ is the outer unit normal to
$\partial\Omega$. The inverse problem we are considering here is to
determine $D$ when $\Lambda_D$ is given.

This problem is one of the fundamental issues of inverse problems
in linear elasticity. In fact, the physical problem described by
equations \eqref{eq:Intro.1}--\eqref{eq:Intro.2} corresponds to a
class of diagnostic problems very common in practical
applications, in which the inclusion is constituted by a faulty or
damaged portion of the elastic body and only the exterior boundary
of the experimental sample is accessible to measurements. The
hypothesis of piecewise constant coefficients is also realistic
and describes practical situations in which there is a jump of the
elastic coefficients at the interface of the inclusion. However,
despite the simplicity of its formulation and the relevant
implications in practical applications, few general results on
this inverse problem are known.

The inverse problem of determining an elastic inclusion could be
framed as a special case of determination of the Lam\'{e} moduli
{}from the Dirichlet-to-Neumann map. In this case, however,  most
of the results currently available concern only regular elastic
coefficients. In \cite{NU93}, Nakamura and Uhlmann
established that in two dimensions the Lam\'{e} moduli are
uniquely determined by the Dirichlet-to-Neumann map, assuming that
they are smooth (e.g., $C^\infty(\overline{\Omega})$) and
sufficiently close to a pair of positive constants. For the
three-dimensional case, the uniqueness for both Lam\'{e} moduli
was proved in \cite{NU94}, \cite{ER02}, \cite{NU03}, provided that
they are smooth and that the shear modulus is close to a positive
constant. Some of the above uniqueness results have been
proved in the case of partial Cauchy data, see \cite{IUY12} for
more details.
Concerning results for less regular coefficients, quite recently, the uniqueness and Lipschitz stability in the case of discontinuous piecewise constant Lam\'{e} tensors, with unknown constants, but with a known decomposition of the domain, has been achieved in \cite{BFV13}.

An alternative approach is the one based on
identification of an unknown boundary, namely the interface
$\partial D$ of the inclusion, by measurements taken on $\partial
\Omega$. The extreme cases of a cavity or of a rigid inclusion in an
isotropic elastic body were considered in \cite{MR04} and in
\cite{MR09}, respectively. For this class of inverse problems,
under mild regularity assumptions on the unknown interface,
Morassi and Rosset established a stability estimate of log-log
type {}from a single pair of Cauchy data. For the elastic
inclusion even the uniqueness question {}from a finite number of
boundary measurements, not to mention stability, remains a largely
open issue.

In connection with the problems discussed above, we wish to
mention the reconstruction issue that has drawn a lot of attention
in recent years.

Ikehata developed in \cite{Ike02} the so-called probe method for
reconstructing inclusions in elastic bodies by means of singular
or fundamental solutions. A key ingredient of the method is a
Runge type approximation theorem, which is useful to guarantee the
existence of an approximating sequence to the singular solution.
The basic idea of this method comes {}from Isakov's fundamental
paper \cite{Is88}, in which the uniqueness of the determination of
an inclusion in an electrical conductor {}from the
Dirichlet-to-Neumann map was proved.
See also a corresponding result of uniqueness for elastic inclusions \cite{INT99}.
Unfortunately, Runge type
approximation theorems are  typically based on nonconstructive
arguments and, therefore, they are not suitable for stability estimates.
Still along this line of research, interesting results for the
reconstruction of an unknown inclusion in two dimensions were
obtained by Ikehata in \cite{Ike04}.

Uhlmann and Wang proposed in \cite{UW08} a method for constructing
complex geometrical optics (CGO) solutions with general phases for
various systems with Laplacian principal part, which include the
inhomogeneous Lam\'{e} system in the plane. In particular, in
\cite{UWW09}, the authors provided a reconstruction algorithm to
the inverse problem of determining $D$ {}from $\Lambda_D$. The
idea is to probe the medium with CGO solutions having
polynomial-type phase functions. The method works for bounded or
unbounded planar regions, made by inhomogeneous Lam\'{e} material,
and does not need a Runge type theorem. Using the CGO solutions,
the authors develop an algorithm to reconstruct the exact shape of
a large class of inclusions, including star-shaped domains.
Numerical implementation of the method gave encouraging results.
Extension to three dimensions, however, does not seem to be  easy as
this method heavily relies on the use of conformal mappings.

In this paper we prove, under suitable mild a-priori assumptions
on the regularity and on the topology of $D$, a continuous
dependence of $D$ {}from $\Lambda_D$ with a modulus of continuity
of logarithmic type. Our proof is inspired by the paper by
Alessandrini and Di Cristo \cite{ADiC05}, in which a logarithmic
stability estimate for the corresponding problem in impedance
tomography, which involves a single scalar elliptic equation with
piecewise constant coefficient, was obtained. In this direction,
we would like also to mention the recent papers by Di Cristo and
Vessella \cite{DiCV10}, \cite{DiCV11}, for  analogous results for the stable
determination of a time varying inclusion within a thermal
conductor. The aforementioned papers are based on quantitative
estimates of unique continuation and on accurate study of the
asymptotic behavior of fundamental solutions when the singularity
gets close to the unknown interface.

The approach we follow here goes along the same line of reasoning but there are several
steps which present new difficulties and in which we have been forced
to introduce novel arguments. Let us outline the main steps of the proof and the new challenges that we have encountered.

Consider two possible inclusions $D_1$, $D_2$ and their corresponding Dirichlet-to-Neumann
maps $\Lambda_{D_1}$, $\Lambda_{D_2}$. The main steps are the following.

i) We introduce the fundamental solutions $\Gamma^{D_1}$, $\Gamma^{D_2}$ for the Lam\'{e} system \eqref{eq:Intro.1} in the full space when $D=D_1,D_2$ respectively.

ii) We show that $(\Gamma^{D_1} - \Gamma^{D_2})(y,w)$ can be dominated linearly  by
$\Lambda_{D_1} - \Lambda_{D_2}$ when $y,w$ are outside of $\Omega$ (see \eqref{eq:21.0}, \eqref{eq:24.1}).

iii) We propagate the smallness of $(\Gamma^{D_1} - \Gamma^{D_2})(y,w)$ as $y,w$ are moved inside of $\Omega$ in the connected component $\mathcal{G}$ of $\R^3 \setminus
(\overline{D_1 \cup D_2})$ which contains $\R^3 \setminus \overline{\Omega}$.

iv) We examine the asymptotics of $(\Gamma^{D_1} - \Gamma^{D_2})(y,w)$ as $y,w$ approach
to a point $P$ of $\partial D_1\setminus \overline{D_2}$ (or $\partial D_2\setminus \overline{D_1}$).

v)  We evaluate the distance between $D_1$ and $D_2$ by matching the smallness estimates of Step
iii) with the blowup asymptotics of Step iv).

Let us now illustrate with some more details the character of such steps.

Step i) is based on a-priori regularity estimates of solutions of the Lam\'{e} system
with piecewise constant Lam\'{e} moduli. For this purpose we appeal to the theory
of existence developed by Hofmann and Kim \cite{HK07} and to the a-priori bounds due to Li and
Nirenberg \cite{LN03}. Details can be found in Section \ref{funda}.

Step ii) is based on a version of the so-called Alessandrini's identity, Lemma \ref{lem:13.1}.

Step iii) contains a complication of geometrical/topological character due to the fact that quantitative estimates of unique continuation can be obtained only in sets which are not only
topologically connected, but also whose connectedness is expressed in tight quantitative terms.
Namely, pairs of points need to be connected by chains of balls of controlled size and number, in such a way that the iteration of the three-spheres-inequality gives rise to controlled constants
and moduli of continuity  in the estimates of propagation of smallness.

Note that $D$ is contained in $\Omega$ with no constraint on the distance between
$\overline{D}$ and $\partial\Omega$, actually we may even admit that
$\partial D\cap \partial\Omega\neq\emptyset$.

Indeed, in previous studies, it was assumed $\hbox{dist}(\overline{D},\partial\Omega)\geq const>0$,
here we show that this requirement is unnecessary since for our purposes the D--N map is estimated only for solutions to \eqref{eq:Intro.1} defined on a domain $\widetilde{\Omega}$
strictly larger than $\Omega$. See Section \ref{proof_upperbound}, Step 1.

Step iv) is the one in which the difference between the scalar conductivity equation and the vector Lam\'{e} system becomes more evident and (in our view) presents the most challenging and interesting features.

In fact, in the scalar case it was possible to pick $y=w$ and prove that
$(\Gamma^{D_1} - \Gamma^{D_2})(y,y)$ blows up as $y$ tends nontangentially
to  $P\in\partial D_1\setminus \overline{D_2}$, and to evaluate quantitatively the
blowup rate. In the present case the situation is more complicated for a number of reasons. First of all the fundamental solutions are matrix valued (not scalar) functions and, therefore, it is crucial to understand which of the entries of $\Gamma^{D_1} - \Gamma^{D_2}$ has the desired
blowup behavior. Second, we are assuming that either $\mu^I\neq\mu$ or $\lambda^I\neq\lambda$
with \emph{no order condition} between such parameters. Hence, we cannot expect, in general, that the difference matrix $\Gamma^{D_1} - \Gamma^{D_2}$ may satisfy any positivity condition. For these reasons we have chosen to examine each diagonal entry of $\Gamma^{D_1} - \Gamma^{D_2}$
separately.
Similarly to the scalar case, we can show that, as $y, w$ tend to $P\in\partial D_1\setminus \overline{D_2}$, $(\Gamma^{D_1} - \Gamma^{D_2})(y,w)$ has, in a suitable reference frame, the same asymptotic behavior of $(\Gamma^{+} - \Gamma)(y,w)$. Here $\Gamma$ is the standard Kelvin fundamental solution with Lam\'{e} moduli $\mu$, $\lambda$ and $\Gamma^+$ is the fundamental solution $\Gamma^D$ when $D$ is replaced by the upper half plane $\{x_3>0\}$.

We can take advantage of the fact that $\Gamma^+$ is explicitly known, in fact its expression,
although complicated, was calculated by Rongved \cite{Ron55} in 1955. With the aid of Rongved's formulas we have been able to estimate the blowup rate of $(\Gamma^{+} - \Gamma)_{ii}(y,w)$, $i=1,2,3$, as $y, w\rightarrow 0$ vertically along the line $\{x_1=x_2=0\}$ for suitable choices of $y$, $w$. A notable fact is that we are obliged to pick very specific choices of $y$, $w$, $w\neq y$ (see Proposition \ref{prop:62.1}). In fact we have found explicit examples of moduli $(\lambda,\mu)\neq(\lambda^I,\mu^I)$ for which $(\Gamma^{+} - \Gamma)_{ii}(y,y)=0$. We emphasize
that such a precise analysis has been possible on the grounds of the explicit (algebraic)
character of $\Gamma^{+}$ (see Section \ref{Rongved}).

The organization of the paper is as follows. In Section
\ref{SecNot-and-info} we introduce some notation and the a-priori
information needed for our stability result. The main result of stability, Theorem
\ref{theo:4.1}, is stated in Section \ref{Formulation}. Section \ref{topo} is
devoted to some technical details of topological-metric character related to the evaluation
of the distance between the inclusions and to estimates of propagation of smallness.
The main properties of the fundamental solution of the Lam\'{e} system with
discontinuous coefficients are presented in Section \ref{funda}.
In Section \ref{proof_main} we state two auxiliary estimates, Theorem \ref{theo:1.23} and
Theorem \ref{theo:55.1}, and in their basis we prove the main
Theorem \ref{theo:4.1}. Theorem
\ref{theo:1.23} is proven in the following Section \ref{proof_upperbound}.
Section \ref{asymptotic} contains evaluations of the asymptotic behaviour of the fundamental
solution $\Gamma^D$, in preparation of the proof of Theorem \ref{theo:55.1}, which is completed in
Section \ref{proof_lowerbound}.
Section
\ref{Rongved} is devoted to the analysis of Rongved's
fundamental solution $\Gamma^{+}$. We also investigate the peculiar behaviour of $\Gamma^+ -\Gamma$ (which shows remarkable differences with the scalar case of the conductivity equations) by exploring explicit examples of material parameters $\mu$, $\lambda$ and $\mu^I$, $\lambda^I$.
Finally, Section \ref{technical} contains the proof of our main topological-metric lemma, Lemma 4.2.

\section{Notation and a-priori information}
\label{SecNot-and-info}

\subsection{Notation and definitions}

Let us denote by $\R^{3}_+= \{x\in \R^3 \ | \ x_3>0 \}$ and
$\R^3_-= \{x\in \R^3 \ | \ x_3<0 \}$. Given $x\in{\R}^3$, we
shall denote $x=(x',x_3)$, where $x'=(x_1,x_2)\in{\R}^{2}$,
$x_3\in{\R}$. Given $x \in \R^3$, $r>0$, we shall use the
following notation for balls and cylinders.
\begin{equation*}
    B_r(x)=\{y \in \R^3 \ \mid \ |y-x|<r\}, \quad B_r=B_r(0),
\end{equation*}
\begin{equation*}
    B'_r(x')=\{y' \in \R^2 \ \mid \ |y'-x'|<r\}, \quad B'_r=B'_r(0),
\end{equation*}
\begin{equation*}
    Q_{a,b}(x) = \{ (y', y_3) | \ |y'-x'|<a, \ |y_3-x_3|<b \},
    \quad Q_{a,b}=Q_{a,b}(0),
\end{equation*}
\begin{equation*}
    Q_{a,b}(x)^+ = \{ (y', y_3) | \ |y'-x'|<a, \ 0<y_3-x_3<b \},
    \quad Q_{a,b}^+=Q_{a,b}^+(0).
\end{equation*}

\begin{definition}
  \label{def:2.1} (${C}^{k,\alpha}$ regularity)
Let $E$ be a domain in ${\R}^{3}$. Given $k$,
$\alpha$, $k\in \N$, $0<\alpha\leq 1$, we say that $E$
is of \textit{class ${C}^{k,\alpha}$ with
constants $\rho_{0}$, $M_{0}>0$}, if, for any $P \in \partial E$, there
exists a rigid transformation of coordinates under which we have
$P=0$ and
\begin{equation*}
  E \cap B_{\rho_{0}}(0)=\{x \in B_{\rho_{0}}(0)\quad | \quad
x_{3}>\varphi(x')
  \},
\end{equation*}
where $\varphi$ is a ${C}^{k,\alpha}$ function on $B'_{\rho_{0}}$
satisfying
\begin{equation*}
\varphi(0)=0,
\end{equation*}
\begin{equation*}
\nabla \varphi (0)=0, \quad \hbox{when } k \geq 1,
\end{equation*}
\begin{equation*}
\|\varphi\|_{{C}^{k,\alpha}(B'_{\rho_{0}}(0))} \leq M_{0}\rho_{0}.
\end{equation*}

\medskip
\noindent When $k=0$, $\alpha=1$, we also say that $E$ is of
\textit{Lipschitz class with constants $\rho_{0}$, $M_{0}$}.
\end{definition}

\begin{rem}
  \label{rem:2.1}
  We use the convention to normalize all norms in such a way that their
  terms are dimensionally homogeneous and coincide with the
  standard definition when the dimensional parameter equals one.
  For instance, the norm appearing above is meant as follows
\begin{equation*}
  \|\varphi\|_{{C}^{k,\alpha}(B'_{\rho_{0}}(0))} =
  \sum_{i=0}^k \rho_0^i
  \|\nabla^i\varphi\|_{{L}^{\infty}(B'_{\rho_{0}}(0))}+
  \rho_0^{k+\alpha}|\nabla^k\varphi|_{\alpha, B'_{\rho_0}(0)},
\end{equation*}
where
\begin{equation*}
|\nabla^k\varphi|_{\alpha, B'_{\rho_0}(0)}= \sup_
{\overset{\scriptstyle x', \ y'\in B'_{\rho_0}(0)}{\scriptstyle
x'\neq y'}} \frac{|\nabla^k\varphi(x')-\nabla^k\varphi(y')|}
{|x'-y'|^\alpha}.
\end{equation*}

Similarly, for a vector function $u: \Omega \subset \R^3
\rightarrow \R^3$, we set
\begin{equation*}
\|u\|_{H^1(\Omega, \R^3)}=\left(\int_\Omega u^2
+\rho_0^2\int_\Omega|\nabla u|^2\right)^{\frac{1}{2}},
\end{equation*}
and so on for boundary and trace norms such as
$\|\cdot\|_{H^{\frac{1}{2}}(\partial\Omega, \R^3)}$,
$\|\cdot\|_{H^{-\frac{1}{2}}(\partial\Omega, \R^3)}$.
\end{rem}

For any $U \subset \R^3$ and for any $r>0$, we denote
\begin{equation}
  \label{eq:2.int_env}
  U_{r}=\{x \in U \mid \textrm{dist}(x,\partial U)>r
  \},
\end{equation}
\begin{equation}
  \label{eq:2.ext_env}
  U^{r}=\{x \in \R^3 \mid \textrm{dist}(x, U )<r
  \}.
\end{equation}

\medskip
We denote by $\M^{m\times n}$ the space of $m \times n$ real
valued matrices and by ${\cal L} (X, Y)$ the space of bounded
linear operators between Banach spaces $X$ and $Y$. When $m=n$, we
shall also denote $\M^n=\M^{n\times n}$.

For every pair of real $n$-vectors $a$ and $b$, we denote by $a
\otimes b$ the $n \times n$ matrix with entries
\begin{equation}
  \label{eq:diade}
  (a \otimes b)_{ij} = a_i b_j, \quad i,j=1,...,n.
\end{equation}

For every $3 \times 3$ matrices $A$, $B$ and for every $\C\in{\cal
L} ({\M}^{3}, {\M}^{3})$, we use the following notation:
\begin{equation}
  \label{eq:2.notation_1}
  ({\C}A)_{ij} = \sum_{k,l=1}^{3} C_{ijkl}A_{kl},
\end{equation}
\begin{equation}
  \label{eq:2.notation_2}
  A \cdot B = \sum_{i,j=1}^{3} A_{ij}B_{ij},
\end{equation}
\begin{equation}
  \label{eq:2notation_3}
  |A|= (A \cdot A)^{\frac {1} {2}},
\end{equation}
where $C_{ijkl}$, $A_{ij}$ and $B_{ij}$ are the entries of $\C$,
$A$ and $B$ respectively.

Given two bounded closed sets $A,B\subset \R^3$, let us recall that the Hausdorff distance $d_H(A,B)$ is defined as
\begin{equation*}
  \label{eq:Hausdorff}
  d_H(A,B)= \max\{\max_{x\in A}d(x,B), \max_{x\in B}d(x,A)\}
\end{equation*}

\subsection{A-priori information}
\label{SubsecInfo}

Throughout the paper, we use the following a-priori assumptions.

i) \textit{Domain}

Let $\Omega$ be a bounded domain in $\R^3$ such that
\begin{equation}
  \label{eq:1.0}
  \R^3\setminus \overline{\Omega}\ \hbox{is connected},
\end{equation}
\begin{equation}
  \label{eq:1.1}
  |\Omega| \leq M_1 \rho_0^3,
\end{equation}
\begin{equation}
  \label{eq:1.1bis}
  \Omega \ \hbox{is of class } C^{1,\alpha}, \ \hbox{with
  constants } \ \rho_0, \ M_0,
\end{equation}
where $\rho_0$, $M_0$, $M_1$ are given positive constants, and
$0<\alpha < 1$.

ii) \textit{Inclusion}

Let $D$ be a domain contained in $\Omega$ satisfying
\begin{equation}
  \label{eq:1.2}
  \R^3 \setminus \overline{D} \ \hbox{is connected},
\end{equation}
\begin{equation}
  \label{eq:1.4}
    D \ \hbox{is of class } C^{1,\alpha},\ \hbox{with
  constants } \ \rho_0, \ M_0,
\end{equation}
where $\rho_0$, $M_0$ are given positive constants, and $0<\alpha
< 1$.

iii) \textit{Material}

The body $\Omega$ is assumed to be made of linearly elastic,
isotropic and homogeneous material, with elastic tensor $\C$ of
components
\begin{equation}
  \label{eq:2.1}
    C_{ijkl}=\lambda \delta_{ij}\delta_{kl} + \mu
    (\delta_{ki}\delta_{lj}+\delta_{li}\delta_{kj}),
\end{equation}
where $\delta_{ij}$ is the Kronecker's delta. The \emph{constant}
Lam\'{e} moduli $\lambda$, $\mu$ satisfy the strong convexity
conditions
\begin{equation}
  \label{eq:2.2}
    \mu \geq \alpha_0, \quad 2\mu+3\lambda \geq \gamma_0,
\end{equation}
where $\alpha_0>0$, $\gamma_0 >0$ are given constants.
We shall also assume upper bounds on the Lam\'{e}  moduli
\begin{equation}
  \label{eq:2.2upper}
    \mu \leq \overline{\mu}, \quad \lambda \leq \overline{\lambda},
\end{equation}
where also $\overline{\mu}>0$, $\overline{\lambda}\in \R$ are known quantities.
In some
points of our analysis, we will express the constitutive equation
\eqref{eq:2.1} in terms of $\mu$ and of Poisson's ratio $\nu$,
instead of the Lam\'{e} moduli $\mu$, $\lambda$. Recalling that
\begin{equation}
  \label{eq:2.2BIS}
    \nu = \frac{\lambda}{2(\lambda + \mu)},
\end{equation}
by \eqref{eq:2.2}, \eqref{eq:2.2upper} we have
\begin{equation}
  \label{eq:2.2TER}
    -1<\nu_0\leq \nu \leq \nu_1< \frac{1}{2},
\end{equation}
where $\nu_0$, $\nu_1$ only depend on $\alpha_0$, $\gamma_0$,
$\overline{\mu}$, $\overline{\lambda}$. Let us notice that
\eqref{eq:2.1} trivially implies that
\begin{equation}
  \label{eq:2.2bis}
    C_{ijkl}=C_{klij}=C_{lkij}, \quad \hbox{ }{i,j,k,l=1,2,3}.
\end{equation}
We recall that the first equality in \eqref{eq:2.2bis} is usually named as the major symmetry of the tensor $\C$, whereas the second equality is called the minor symmetry.

Also we note that \eqref{eq:2.2} is equivalent to
\begin{equation}
  \label{eq:2.3}
    \C A \cdot A \geq \xi_0 |A|^2
\end{equation}
for every $3 \times 3$ symmetric matrix $A$, where $\xi_0=\min \{
2\alpha_0, \gamma_0\}$.

Similarly, the inclusion $D$ is made of isotropic homogeneous
material having elasticity tensor $\C^I$, with \emph{constant}
Lam\'{e} moduli $\lambda^I$, $\mu^I$ satisfying the conditions
\eqref{eq:2.2}, \eqref{eq:2.2upper} and such that
\begin{equation}
  \label{eq:3.1}
   (\lambda-\lambda^I)^2+ (\mu-\mu^I)^2\geq \eta_0^2>0,
\end{equation}
for a given constant $\eta_0>0$.

\medskip

In what follows we shall refer to the constants $M_0$, $\alpha$, $M_1$,
$\alpha_0$, $\gamma_0$, $\overline{\mu}$, $\overline{\lambda}$, $\eta_0$ as to the
\emph{a-priori data}.

\medskip

Observe that, in view of \eqref{eq:2.2BIS} and of the a-priori bounds on the Lam\'{e} moduli,
{}from \eqref{eq:3.1} it also follows
\begin{equation}
  \label{eq:3.1bis}
   (\nu-\nu^I)^2+ (\mu-\mu^I)^2\geq C\eta_0^2>0,
\end{equation}
where $C$ only depends on $\alpha_0$, $\gamma_0$, $\overline{\mu}$, $\overline{\lambda}$.

Finally, note that the jump condition \eqref{eq:3.1} does \emph{not} imply any kind of monotonicity relation between $\C$ and $\C^I$.

\section{Formulation of the inverse problem and stability result}
\label{Formulation}

For any $f \in H^{1/2}(\partial \Omega)$, let $u \in H^1(\Omega)$
be the weak solution to the Dirichlet problem
\begin{center}
\( {\displaystyle \left\{
\begin{array}{lr}
  \divrg((\C + (\C^I - \C)\chi_D)\nabla u)=0,
  & \hbox{in}\ \Omega,
    \vspace{0.25em}\\
  u=f, & \hbox{on}\ \partial \Omega
\end{array}
\right. } \) \vskip -4.4em
\begin{eqnarray}
& & \label{eq:3.2}\\
& & \label{eq:3.3}
\end{eqnarray}
\end{center}
where $\chi_D$ is the characteristic function of $D$.

Let us recall that the so-called Dirichlet-to-Neumann map
\begin{equation}
  \label{eq:3.4}
   \Lambda_D: H^{1/2}(\partial \Omega) \rightarrow
   H^{-1/2}(\partial \Omega),
\end{equation}
is defined in the weak form by
\begin{equation}
  \label{eq:3.4bis}
   <\Lambda_D f, v|_{\partial \Omega}>=\int_\Omega (\C + (\C^I - \C)\chi_D)
   \nabla u\cdot \nabla v,
\end{equation}
for every $v\in H^1(\Omega)$.

In what follows it will be convenient to write, with a slight, but customary, abuse of notation,
\begin{equation*}
   <\Lambda_D f, v|_{\partial \Omega}>=\int_{\partial\Omega} v \Lambda_D f.
\end{equation*}

The inverse problem we are interested in consists in recovering
the inclusion $D$ {}from the knowledge of the map $\Lambda_D$ and,
more precisely, we want to prove a stability estimate. Our main
result is the following.

\begin{theo}
   \label{theo:4.1}

Let $\Omega \subset \R^3$ satisfy
\eqref{eq:1.1}--\eqref{eq:1.1bis} and let $D_1$, $D_2$ be two
inclusions contained in $\Omega$ satisfying
\eqref{eq:1.2}--\eqref{eq:1.4}. Let $\C$ and $\C^I$ be the
constant elastic tensors of the material of $\Omega$ and of the
inclusions $D_i$, $i=1,2$, respectively, where $\C$ and $\C^I$
satisfy \eqref{eq:2.1}--\eqref{eq:2.2upper} and \eqref{eq:3.1}. If, for
some $\epsilon$, $0<\epsilon<1$,
\begin{equation}
  \label{eq:4.1}
   \| \Lambda_{D_1}-\Lambda_{D_2}\|_{\mathcal{L}(H^{1/2}(\partial \Omega), H^{-1/2}(\partial
   \Omega))}\leq \frac{\epsilon}{\rho_0},
\end{equation}
then
\begin{equation}
  \label{eq:4.2}
   d_H(\partial D_1, \partial D_2) \leq \rho_0\omega(\epsilon),
\end{equation}
where $\omega$ is an
increasing function on $[0,+\infty)$ satisfying
\begin{equation}
  \label{eq:4.3}
   \omega(t) \leq C |\log t |^{-\eta}, \ \hbox{for every } \
   0<t<1,
\end{equation}
where $C>0$ and $\eta$, $0<\eta\leq 1$, are constants only
depending on the a-priori data.
\end{theo}

\begin{rem}
  \label{rem:localmap}
  In the case when $D_1$, $D_2$ are at a prescribed positive distance {}from $\partial\Omega$, it is also  possible to obtain a result analogous to the above Theorem when the
  Dirichlet-to-Neumann maps $\Lambda_{D_1}$, $\Lambda_{D_2}$ are replaced with local maps.
  For instance, fixing $Q\in\partial\Omega$ and given $\rho_1>0$, denoting
  $\Sigma=\partial\Omega\cap B_{\rho_1}(Q)$, we introduce
  \begin{equation*}
  H_{co}^{1/2}(\Sigma)=\{g\in H^{1/2}(\partial\Omega)\ |\ \hbox{supp}\ g\subset\subset\Sigma\}
  \end{equation*}
   and define
   \begin{equation*}
   \Lambda^\Sigma_{D_i}: H_{co}^{1/2}(\Sigma)\rightarrow(H_{co}^{1/2}(\Sigma))^*\subset
  H^{-1/2}(\partial\Omega)
  \end{equation*}
  as the restriction of $\Lambda_{D_i}$ to $H_{co}^{1/2}(\Sigma)$.
  Thus, replacing the assumption \eqref{eq:4.1} with
  \begin{equation*}
  \| \Lambda_{D_1}^\Sigma-\Lambda_{D_2}^\Sigma\|_{\mathcal{L}\left(H^{1/2}_{co}(\Sigma),
  \left(H^{1/2}_{co}(\Sigma)\right)^*\right)}\leq \frac{\epsilon}{\rho_0},
  \end{equation*}
  we obtain \eqref{eq:4.2}--\eqref{eq:4.3} with constants only depending on the a-priori data
   and on $\rho_1$. Such a result is a nearly straightforward adaptation of the theory developed in \cite{AK12}.

\end{rem}
The proof of Theorem \ref{theo:4.1} will be given in Section \ref{proof_main}.
In the following two sections, we introduce some auxiliary results, concerning
the topological-metric aspects of the problem and the main properties of the fundamental solution
of the Lam\'e system with discontinuous coefficients.

\section{Metric lemmas}
\label{topo}

Let $\mathcal{G}$ be the connected component of $\R^3 \setminus
(\overline{D_1 \cup D_2})$ which contains $\R^3 \setminus
\overline{\Omega}$ and let us denote
\begin{equation}
  \label{eq:5.1}
   \Omega_D= \R^3 \setminus \overline{\mathcal{G}}.
\end{equation}
As we shall see later, one of the key ingredients of the stability
proof consists in propagating the smallness {}from the boundary
$\partial \Omega$ inside $\Omega$. Since the value $d_H(\partial
D_1, \partial D_2)$ may be attained at some point not belonging to
$\overline{\mathcal{G}}$ and, therefore, not reachable {}from the
exterior, it is necessary to introduce a \textit{modified
distance} following the ideas developed in \cite{ADiC05}.
Precisely, let us introduce the \textit{modified distance} between $D_1$ and $D_2$
\begin{equation}
  \label{eq:5.2}
   d_\mu (D_1,D_2) = \max \left \{
   \max_{x\in \partial D_1 \cap \partial \Omega_D}
   \hbox{dist}(x,D_2),  \ \max_{x\in \partial D_2 \cap \partial \Omega_D}
   \hbox{dist}(x,D_1) \right \}.
\end{equation}
We remark here that $d_\mu$ is not a metric and, in general, it
does not dominate the Hausdorff distance. However, under our a
priori assumptions on the inclusion, the following lemma holds
true.

\begin{lem}[Proposition 3.3 in \cite{ADiC05}]
   \label{lem:6.1}

Under the assumptions of Theorem \ref{theo:4.1}, there exists a constant $c_0 \geq 1$ only depending on $M_0$ and
$\alpha$ such that
\begin{equation}
  \label{eq:6.1}
   d_H (\partial D_1, \partial D_2) \leq c_0 d_\mu (D_1,D_2).
\end{equation}
\end{lem}

It is easy to verify that
\begin{equation*}
    \max_{x\in \partial D_1 \cap \partial \Omega_D}
   \hbox{dist}(x,D_2) =  \max_{x\in \partial D_1 \cap \partial \Omega_D}
   \hbox{dist}(x, \partial D_2)
\end{equation*}
\begin{equation*}
    \max_{x\in \partial D_2 \cap \partial \Omega_D}
   \hbox{dist}(x,D_1) = \max_{x\in \partial D_2 \cap \partial \Omega_D}
   \hbox{dist}(x, \partial D_1),
\end{equation*}
so that $d_\mu(D_1,D_2) \leq
d_H(\partial D_1, \partial D_2)$, and therefore, in view of Lemma \ref{lem:6.1}, these two
quantities are
comparable.

Another obstacle comes out {}from the fact that the propagation of
smallness arguments are based on an iterated application of the
three-spheres inequality for solutions to the Lam\'{e} system over
chains of balls contained in $\mathcal{G}$ and, in this step, it
is crucial to control {}from below the radii of these balls. In order to circumvent
the case in which points of $\partial \Omega_D$ are not reachable by such chains of
balls, we
found it convenient to adapt to our case ideas first presented in
\cite{AS12} in dealing with crack detection in electrical
conductors, which we summarize in the lemma below.
We note, incidentally, that this issue was somewhat underestimated in \cite{ADiC05}.
The procedure developed here  enables to fill the possible
gaps in the proofs in \cite{ADiC05} (and also in \cite{DiC07}, \cite{DiCV10}, \cite{DiCV11}).

Let us premise some notation. Given $O=(0,0,0)$, $v$ a unit
vector, $h>0$ and $\vartheta \in \left ( 0, \frac{\pi}{2} \right
)$, we denote
\begin{equation}
  \label{eq:cono}
   C (O, v, h, \vartheta )
   =\left \{
   x \in \R^3 |\ |x - (x \cdot v)v| \leq \sin \vartheta |x|, \ 0 \leq x\cdot v \leq h  \right \}
\end{equation}
the closed truncated cone with vertex at $O$, axis along the
direction $v$, height $h$ and aperture $2\vartheta$. Given $R$,
$d$, $0 < R < d$ and $Q=-de_3 $, let us consider the cone $C
\left (O,-e_3, \frac{d^2-R^2}{d}, \arcsin
\frac{R}{d}\right )$. We note that the lateral boundary of this cone
is tangent to the sphere $\partial B_R(Q)$ along the circumference of its base.

{}From now on, for simplicity, we assume that
\begin{equation}
  \label{eq:8.0}
   d_\mu (D_1,D_2) = \max_{x \in \partial D_1 \cap \partial
   \Omega_D} \hbox{dist}(x, \partial D_2)
\end{equation}
and we write $d_\mu = d_\mu (D_1,D_2)$.

Let us define
\begin{equation}
  \label{eq:8.2}
   S_{2\rho_0} = \left \{ x\in\R^3 \ | \rho_0 < \hbox{dist}(x, \overline{\Omega}) <2\rho_0 \right
   \}.
\end{equation}

We shall make use of paths connecting points in order that appropriate tubular neighbourhoods of such paths still remain within $\R^3\setminus \Omega_D$.

Let us pick a point $P\in\partial D_1\cap\partial\Omega_D$, let $\nu$ be the outer unit normal to $\partial D_1$ at $P$ and let $d>0$ be such that the segment $[(P+d\nu),P]$ is contained in $\R^3\setminus\Omega_D$. Given $P_0\in \R^3\setminus\Omega_D$, let $\gamma$ be a path
in $\R^3\setminus\Omega_D$ joining $P_0$ to $P+d\nu$. We consider the following neighbourhood of
$\gamma\cup [(P+d\nu),P]\setminus \{P\}$ formed by a tubular neighbourhood of $\gamma$ attached to a cone with vertex at $P$ and axis along $\nu$
\begin{equation}
  \label{eq:matita}
   V(\gamma) = \bigcup_{S \in \gamma} B_R(S) \cup
    C \left (P,\nu, \frac{d^2-R^2}{d}, \arcsin
   \frac{R}{d}\right ).
\end{equation}
Note that two significant parameters are associated to such a set, the radius $R$ of the tubular neighbourhood of $\gamma$, $\cup_{S\in\gamma}B_R(S)$, and the half-aperture $\arcsin{\frac{R}{d}}$
of the cone $C \left (P,\nu, \frac{d^2-R^2}{d}, \arcsin
\frac{R}{d}\right )$. In other terms, $V(\gamma)$ depends on $\gamma$ and also on the parameters
$R$ and $d$. At each of the following steps, such two parameters shall be appropriately chosen and
shall be accurately specified. For the sake of simplicity we convene to maintain the notation $V(\gamma)$ also when different values of $R$, $d$ are introduced.

Also we warn the reader that it will be convenient at various stages to use a reference
frame such that $P=O=(0,0,0)$ and $\nu=-e_3$.
\begin{lem}
   \label{lem:8.1}
Under the above notation, there exist positive constants $\overline{d}$,
$c_1$, where $ \frac{\overline{d}}{\rho_0}$ only depends on $M_0$ and
$\alpha$, and $c_1$ only depends on $M_0$, $\alpha$, $M_1$, and
there exists a point $P \in
\partial D_1$ satisfying
\begin{equation}
  \label{eq:8BIS.2}
   c_1 d_\mu \leq \hbox{dist}(P,D_2),
\end{equation}
and such that, giving any point $P_0 \in S_{2\rho_0}$, there
exists a path $\gamma \subset ( \overline{\Omega^{\rho_0}} \cup
S_{2\rho_0} ) \setminus \overline{\Omega_D} $ joining $P_0$ to
$P+\overline{d}\nu$, where $\nu$ is the unit outer normal to $D_1$ at $P$,
such that, choosing a coordinate system with origin $O$ at $P$ and
axis $e_3=-\nu$, the set $V(\gamma)$ introduced in \eqref{eq:matita}
satisfies
\begin{equation}
  \label{eq:8BIS.4}
   V(\gamma) \subset \R^3 \setminus \Omega_D,
\end{equation}
provided $R = \frac{\overline{d}}{ \sqrt{1+L_0^2}  }$, where $L_0$, $0<L_0 \leq M_0$, is
a constant only depending on $M_0$ and $\alpha$.
\end{lem}

The proof of Lemma \ref{lem:8.1} is given in Section \ref{technical}.

\section{Fundamental solution of the Lam\'e system with discontinuous coefficients}
\label{funda}

In this Section, $D$ is a domain of class $C^{1,\alpha}$ with constants $\rho_0$, $M_0$, $0<\alpha<1$, and $\C$, $\C^I$ satisfy \eqref{eq:2.1}--\eqref{eq:2.2upper}.

Given $y \in \R^3$ and a \textit{concentrated force} $l \delta(\cdot - y)$ applied at $y$,
$l \in \R^3$, $|l|=1$, let us consider the \textit{normalized
fundamental solution} $u^D \in L^1_{loc}(\R^3, \R^3)$ defined by
\begin{equation}
  \label{eq:13.2}
  \left\{ \begin{array}{ll}
  \divrg_x \left ( (\C + (\C^I - \C)\chi_D)\nabla_x u^D(x,y;l) \right ) =-l\delta(x-y),
  & \hbox{in}\ \R^3\setminus \{y\},\\
  &  \\
      \lim_{|x| \rightarrow \infty} u^D(x,y;l)=0,\\
  \end{array}\right.
\end{equation}

where $\delta(\cdot - y)$ is the Dirac distribution supported at
$y$, that is
\begin{equation}
  \label{eq:14.1}
   \int_{\R^3} (\C + (\C^I - \C)\chi_D)\nabla_x u^D(x,y;l) \cdot
   \nabla_x \varphi(x) = l \cdot \varphi (y), \quad \hbox{for every
   } \varphi \in C_c^\infty(\R^3, \R^3).
\end{equation}
It is well-known that
\begin{equation}
  \label{eq:14.2}
   u^D(x,y;l) = \Gamma^D(x,y)l,
\end{equation}
where $\Gamma^D=\Gamma^D(\cdot,y) \in L^1_{loc}(\R^3,
\mathcal{L}(\R^3,\R^3))$ is the \textit{normalized fundamental
matrix} for the operator $\divrg_x((\C + (\C^I -
\C)\chi_D)\nabla_x (\cdot))$. The existence of $\Gamma^D$ is
ensured by the following Proposition.
\begin{prop}
   \label{prop:14.1}
Under the above assumptions, there exists a unique fundamental
matrix $\Gamma^D(\cdot, y) \in C^0(\R^3\setminus \{y\})$.
Moreover, we have
\begin{equation}
  \label{eq:14.3}
   \Gamma^D(x,y) = (\Gamma^D(y,x))^T, \quad \hbox{for every } x\in
   \R^3, \ x \neq y,
\end{equation}
\begin{equation}
  \label{eq:14.4}
   |\Gamma^D(x,y)| \leq C |x-y|^{-1}, \quad \hbox{for every } x\in
   \R^3, \ x \neq y,
\end{equation}
\begin{equation}
  \label{eq:14.5}
   |\nabla_x \Gamma^D(x,y)| \leq C |x-y|^{-2}, \quad \hbox{for every } x\in
   \R^3, \ x \neq y,
\end{equation}
where the constant $C>0$ only depends on $M_0$, $\alpha$,
$\alpha_0$, $\gamma_0$, $\overline{\lambda}$, $\overline{\mu}$.
\end{prop}
Let us premise the following Lemma due to Li and Nirenberg
\cite{LN03}.
\begin{lem}
   \label{lem:14BIS.1}
Under the above hypotheses, let $u \in H^1(Q_{r,rM_0})$ be a solution
to
\begin{equation}
  \label{eq:14BIS.1}
   \divrg((\C + (\C^I - \C)\chi_D)\nabla u )=0, \quad \hbox{in }
   Q_{r,rM_0}.
\end{equation}
Then, $u \in C^0(Q_{r,rM_0})$ and, for every $x \in Q_{r,rM_0}$
such that $Q_{2\rho,2\rho M_0}(x) \subset Q_{r,rM_0}$, we have
\begin{multline}
  \label{eq:14BIS.2}
   \| \nabla u\|_{L^\infty ( Q_{\rho,\rho M_0}(x))}
   +
   \rho^{\beta} |\nabla u |_{\beta, Q_{\rho,\rho M_0}(x) \cap
   \overline{D}} +
    \rho^{\beta} | \nabla u |_{\beta, Q_{\rho,\rho M_0}(x)
    \setminus{D}}\leq
    \\
    \leq
    \frac{C}{\rho^{1+ \frac{3}{2}}}
    \left (
    \int_{Q_{2\rho,2\rho M_0}( {x})}
    |u|^2
    \right )^{ \frac{1}{2}},
\end{multline}
where $|\cdot|_\beta$ denotes the usual H\"{o}lder seminorm,
$\beta = \frac{\alpha}{2(1+\alpha)}$ and $C>0$ only depends on
$M_0$, $\alpha$, $\alpha_0$, $\gamma_0$, $\overline{\lambda}$,
$\overline{\mu}$.
\end{lem}
\begin{proof}[Proof of Proposition \ref{prop:14.1}]
In view of the results presented in \cite{HK07}, in order to ensure the existence
of $\Gamma^D$ and properties \eqref{eq:14.3}, \eqref{eq:14.4},
it is sufficient
to prove that there exist constants $\mu_0\in (0,1]$, $C>0$ such
that, for every $R>0$ and $\overline{x} \in \R^3$, all weak solutions $u \in H^1(B_{2R}(\overline{x}))$ of the equation
\begin{equation}
  \label{eq:15.1}
   \divrg (( \C + (\C^I -\C)\chi_D)\nabla u)=0
\end{equation}
satisfy
\begin{equation}
  \label{eq:15.2}
   |u|_{ \mu_0, B_R(\overline{x})}
    \leq \frac{C}{R^{\mu_0}}
   \left ( \frac{1}{| B_{2R}(\overline{x})   |}   \int_{B_{2R}(\overline{x})} |u|^2 \right ) ^{ \frac{1}{2}},
\end{equation}
see Lemma 2.3 in \cite{HK07}. In fact, we shall derive \eqref{eq:15.2} with $\mu_0=1$.

By Lemma \ref{lem:14BIS.1},
$u \in W^{1,\infty}(
B_{R}(\overline{x}))$ and consequently it is Lipschitz continuous.
By the results in \cite{LN03}, we have
\begin{equation}
  \label{eq:19.2}
   |u|_{ 1,B_R(\overline{x})}=\|\nabla u\|_{ L^\infty(B_R(\overline{x}))}
    \leq \frac{C}{R^{ \frac{5}{2}  }}
   \left ( \int_{B_{2R}(\overline{x})} |u|^2 \right ) ^{ \frac{1}{2}},
\end{equation}
where $C>0$ only depends on $\alpha$, $M_0$, $\alpha_0$,
$\gamma_0$, $\overline{\lambda}$, $\overline{\mu}$. By
\eqref{eq:19.2}, we finally obtain the desired
estimate
\begin{equation}
  \label{eq:19.3}
   |u|_{ 1,B_R(\overline{x})}
    \leq \frac{C}{R}
   \left ( \frac{1}{| B_{2R}(\overline{x})   |}   \int_{B_{2R}(\overline{x})} |u|^2 \right ) ^{ \frac{1}{2}},
\end{equation}
where $C>0$ only depends on $\alpha$, $M_0$, $\alpha_0$,
$\gamma_0$, $\overline{\lambda}$, $\overline{\mu}$.

It remains to prove estimate \eqref{eq:14.5}. By applying \eqref{eq:19.2} to $\Gamma^D(\cdot,y)$ in $B_s(x)$,
where $s = \frac{|x-y|}{4}$, we have
\begin{equation*}
     \| \nabla_x \Gamma^D(\cdot, y)\|_{ L^\infty (B_s(x))}
     \leq
     \frac{C}{s^{ \frac{5}{2}  }}
     \left (
     \int_{B_{2s}(x)} |\Gamma^D (\xi,y)|^2 d\xi
     \right )^{ \frac{1}{2}  }.
\end{equation*}
Since $|\xi-y|\geq 2s$, and by \eqref{eq:14.4},we have
\begin{equation}
      \label{eq:19BIS.1}
     \| \nabla_x \Gamma^D(\cdot, y)\|_{ L^\infty (B_s(x))}
     \leq
      \frac{C}{|x-y|^2},
\end{equation}
where $C>0$ only depends on $\alpha$, $M_0$, $\alpha_0$,
$\gamma_0$, $\overline{\lambda}$, $\overline{\mu}$.
\end{proof}

\section{Proof of the main theorem}
\label{proof_main}

We begin with the following identity, the prototype of which can be attributed to Alessandrini \cite{Al88} in connection with the inverse conductivity problem.

\begin{lem}
   \label{lem:13.1}
Under the assumptions of Theorem \ref{theo:4.1}, let $u_i \in H^1(\Omega)$, $i=1,2$, be solutions to
\eqref{eq:3.2} with $D=D_i$ respectively.
Then the following identity
holds
\begin{multline}
  \label{eq:13.1}
   \int_\Omega( \C +(\C^I-\C)\chi_{D_1}) \nabla u_1 \cdot \nabla u_2
   - \int_\Omega( \C +(\C^I-\C)\chi_{D_2}) \nabla u_1 \cdot \nabla
   u_2= \\
   = <(\Lambda_{D_1} -
   \Lambda_{D_2})u_2,u_1>.
\end{multline}
\end{lem}
\begin{proof}
Straightforward consequence of \eqref{eq:3.4bis} and of the symmetry properties of $\C$, $\C^I$.
\end{proof}

\medskip
Let us choose $y$, $w \in \R^3$, $y\neq w$, and $l$, $m \in
\R^3$ such that $|l|=|m|=1$. We define the functions
\begin{equation}
  \label{eq:20.1}
     S_{D_1}(y,w;l,m) = \int_{D_1} (\C^I -\C)\nabla_x (
     \Gamma^{D_1}(x,y)l) \cdot \nabla_x (
     \Gamma^{D_2}(x,w)m),
\end{equation}
\begin{equation}
  \label{eq:20.2}
     S_{D_2}(y,w;l,m) = \int_{D_2} (\C^I -\C)\nabla_x (
     \Gamma^{D_1}(x,y)l) \cdot \nabla_x (
     \Gamma^{D_2}(x,w)m),
\end{equation}
\begin{equation}
  \label{eq:20.3}
     f(y,w;l,m)=S_{D_1}(y,w;l,m)-S_{D_2}(y,w;l,m).
\end{equation}
The following Lemma takes its inspiration {}from a result due to Beretta, Francini and Vessella
\cite[Proposition 3.2]{BFV13}.
\begin{lem}
   \label{lem:21.0}
   For every $y,w\in\R^3$, $y\neq w$, we have
\begin{equation}
  \label{eq:21.0}
     f(y,w;l,m)=(\Gamma^{D_2}- \Gamma^{D_1}) (y,w)m\cdot l.
\end{equation}
\end{lem}
\begin{proof}
Let us denote $\Gamma_i=\Gamma^{D_i}$ and $\C_i =(\C+(\C^I-\C)\chi_{D_i})$,
$i=1,2$.
Let $R>0$ be large enough so that $\Omega\subset B_R(0)$ and $|y|, |w|<R$. By Green's formula we have
\begin{multline*}
     \int_{\partial B_R(0)}(\C_2\nabla_x\Gamma_2(x,w)m)\nu\cdot(\Gamma_1(x,y)l)-
     \int_{B_R(0)}\C_2\nabla_x(\Gamma_2(x,w)m)\cdot\nabla_x(\Gamma_1(x,y)l)=\\
     =-\Gamma_1(w,y)l\cdot m,
\end{multline*}
and also
\begin{multline*}
     \int_{\partial B_R(0)}(\C_1\nabla_x\Gamma_1(x,y)l)\nu\cdot(\Gamma_2(x,w)m)-
     \int_{B_R(0)}\C_1\nabla_x(\Gamma_1(x,y)l)\cdot\nabla_x(\Gamma_2(x,w)m)=\\
     =-\Gamma_2(y,w)m\cdot l,
\end{multline*}
By using the major symmetry of $\C_2$ and subtracting,
\begin{multline*}
     S_{D_1}(y,w;l,m)-S_{D_2}(y,w;l,m)+\\
     +\int_{\partial B_R(0)}(\C_2\nabla_x\Gamma_2(x,w)m)\nu\cdot(\Gamma_1(x,y)l)
     -\int_{\partial B_R(0)}(\C_1\nabla_x\Gamma_1(x,y)l)\nu\cdot(\Gamma_2(x,w)m)=\\
     =\Gamma_2(y,w)m\cdot l-\Gamma_1(w,y)l\cdot m.
\end{multline*}
By \eqref{eq:14.4}, \eqref{eq:14.5}, the boundary integrals are infinitesimal as $R\rightarrow \infty$
and, by \eqref{eq:14.3}
\begin{equation*}
     \Gamma_1(w,y)l\cdot m= \Gamma_1(y,w)m\cdot l .
\end{equation*}
\end{proof}

Let us fix $y= \overline{y}\in \R^3 \setminus \overline{\Omega_D}$ and $l\in
\R^3$, $|l|=1$. Let us define
\begin{equation}
  \label{eq:21.1}
     f_k(w;l)= f(\overline{y},w;l,e_k), \quad k=1,2,3,
\end{equation}
\begin{equation}
  \label{eq:21.2}
     f= (f_1,f_2,f_3).
\end{equation}
Similarly, let us fix $w=\overline{w} \in \R^3 \setminus \overline{\Omega_D}$
and $m \in \R^3$, $|m|=1$. We define
\begin{equation}
  \label{eq:21BIS.1}
     \widetilde{f}_j(y;m)= f(y,\overline{w};e_j,m), \quad j=1,2,3,
\end{equation}
\begin{equation}
  \label{eq:21BIS.2}
     \widetilde{f}= (\widetilde{f}_1,\widetilde{f}_2,\widetilde{f}_3).
\end{equation}
\begin{lem}
   \label{lem:21.1}
The vector-valued function $f=f(w;l)$ satisfies the Lam\'{e}
system
\begin{equation}
  \label{eq:21.3}
      \divrg_w(\C\nabla_w f)=0, \quad \hbox{in } \R^3\setminus \overline{\Omega_D},
\end{equation}
for every $l\in \R^3$, $|l|=1$.

The vector-valued function $\widetilde{f}=\widetilde{f}(y;m)$
satisfies the Lam\'{e} system
\begin{equation}
  \label{eq:21BIS.3}
     \divrg_y(\C\nabla_y \widetilde{f})=0, \quad \hbox{in } \R^3\setminus \overline{\Omega_D},
\end{equation}
for every $m\in \R^3$, $|m|=1$.
\end{lem}
\begin{proof} Since, by \eqref{eq:14.3}, $f=(\Gamma^{D_2}-\Gamma^{D_1})(w,\overline{y})l$ and
$\widetilde{f}=(\Gamma^{D_2}-\Gamma^{D_1})(y,\overline{w})m$, the thesis is a straightforward consequence of Lemma \ref{lem:21.0}.
\end{proof}
\begin{theo}[Upper bound on the function $f$]
   \label{theo:1.23}

   Under the notation of Lemma \ref{lem:8.1}, let
\begin{equation}
  \label{eq:23.1}
     y_h = P - h e_3,
\end{equation}
\begin{equation}
  \label{eq:23.2}
     w_h = P -\lambda_w h e_3, \quad 0<\lambda_w < 1,
\end{equation}
with
\begin{equation}
  \label{eq:23.3}
     0<h\leq \overline{d} \left ( 1 - \frac{\sin \widetilde{\vartheta}_0}{4} \right ),
\end{equation}
where $\widetilde{\vartheta}_0 = \arctan \frac{1}{L_0}$ and $\nu=-e_3$ is the
outer unit normal to $D_1$ at $P$. Then, for every $l$, $m\in\R^3$, $|l|=|m|=1$, we have
\begin{equation}
  \label{eq:23.4}
     |f(y_h, w_h; l,m)| \leq \frac{C}{\lambda_w h } \epsilon ^{ C_1 \left (
     \frac{h}{\rho_0}\right)^{C_2}},
\end{equation}
where $\epsilon$ is the error bound introduced in \eqref{eq:4.1} and the constant $C>0$ only depends on $M_0$, $\alpha$, $M_1$,
$\alpha_0$, $\gamma_0$, $\overline{\lambda}$, $\overline{\mu}$;
\begin{equation}
  \label{eq:23.5}
   C_1 = \gamma \delta^{ 2 + 2 \frac{|\log A|}{|\log
    \chi|}}, \quad  C_2 = 2\frac{| \log \delta| }{|\log \chi|}, \quad \quad A= \frac{\lambda_w  }{ \frac{\overline{d}}{\rho_0} (1
- \vartheta^* \frac{ \sin \widetilde{\vartheta}_0}{8}) }, \quad \chi = \frac{1
- \frac{\sin \widetilde{\vartheta}_0}{8} }{ 1+ \frac{\sin \widetilde{\vartheta}_0}{8} },
\end{equation}
where $\delta$, $0<\delta<1$, $\vartheta^*$, $0<\vartheta^*\leq
1$, only depend on $\alpha_0$, $\gamma_0$, $\overline{\lambda}$,
$\overline{\mu}$; $\gamma>0$ only depends on $M_0$, $\alpha$,
$M_1$, $\alpha_0$, $\gamma_0$, $\overline{\lambda}$,
$\overline{\mu}$.
\end{theo}
\begin{theo}[Lower bound on the function $f$]
   \label{theo:55.1}

Under the notation of Lemma \ref{lem:8.1}, let
\begin{equation}
  \label{eq:55.1}
     y_h = P -h e_3.
\end{equation}
For every $i=1,2,3$, there exists $\lambda_w \in \left\{ \frac{2}{3}, \frac{3}{4}, \frac{4}{5}
\right\}$ and there exists $\overline{h} \in \left (0, \frac{1}{2} \right )$ only depending
on $M_0$, $\alpha$, $\alpha_0$, $\gamma_0$, $\overline{\lambda}$,
$\overline{\mu}$, $\eta_0$, such that
\begin{equation}
  \label{eq:55.2}
     |f(y_h, w_h; e_i, e_i)| \geq \frac{C}{h}, \quad \hbox{for every } h, \
     0<h<\overline{h}\rho,
\end{equation}
where
\begin{equation}
  \label{eq:55.3}
     w_h = P -\lambda_w h e_3,
\end{equation}
\begin{equation}
  \label{eq:55.4}
     \rho =  \min \left \{
     dist(P,D_2), \ \frac{\rho_0}{12 \sqrt{1+M_0^2}} \cdot \min \{1,M_0\}  \right
     \},
\end{equation}
and $C>0$ only depends on $M_0$, $\alpha$, $\alpha_0$, $\gamma_0$,
$\overline{\lambda}$, $\overline{\mu}$, $\eta_0$.
\end{theo}

\medskip

\begin{proof} [Proof of Theorem \ref{theo:4.1}]
{}From the combination of the upper bound \eqref{eq:23.4}, with
$l=m=e_i$ for $i\in\{1,2,3\}$, and {}from the lower bound
\eqref{eq:55.2}, we have
\begin{equation}
  \label{eq:74.1}
   C \leq \epsilon ^{ C_1 \left (
     \frac{h}{\rho_0}\right)^{C_2}}, \quad \hbox{for every } h, \
     0<h\leq \overline{h}\rho,
\end{equation}
where $\rho$ is given in \eqref{eq:55.4}, the constants $C_1>0$,
$C_2>0$ are defined in \eqref{eq:23.5} and depend only on $M_0$, $\alpha$,
$M_1$, $\alpha_0$, $\gamma_0$, $\overline{\lambda}$,
$\overline{\mu}$, and the constants $C \in (0,1)$, $\overline{h}
\in \left ( 0, \frac{1}{2} \right )$ only depend on $M_0$,
$\alpha$, $\alpha_0$, $\gamma_0$, $\overline{\lambda}$,
$\overline{\mu}$, $\eta_0$.

Passing to the logarithm and recalling that $\epsilon \in (0,1)$,
we have
\begin{equation}
  \label{eq:74.2}
   h \leq C\rho_0 \left ( \frac{1}{ |\log \epsilon |  } \right
   )^{ \frac{1}{C_2} }, \quad \hbox{for every } h, \
     0<h\leq \overline{h}\rho,
\end{equation}
In particular, choosing $h = \overline{h}\rho$, we have
\begin{equation}
  \label{eq:75.1}
   \rho \leq C\rho_0 \left ( \frac{1}{ |\log \epsilon |  } \right
   )^{ \frac{1}{C_2} }.
\end{equation}
If $\rho = dist(P,D_2)$, by Lemma \ref{lem:6.1} and Lemma
\ref{lem:8.1}, the thesis follows. If, otherwise, $\rho=
\frac{\rho_0}{ 12 \sqrt{1+M_0^2}   } \min\{1,M_0\}$, the thesis
follows by noticing that $d_H(\partial D_1, \partial D_2) \leq
diam(\Omega)\leq C\rho_0$, with $C>0$ only depending on $M_0$,
$M_1$.
\end{proof}

\section{Proof of Theorem \ref{theo:1.23}}
\label{proof_upperbound}

The proof is divided into four steps.

\noindent \textit{\textbf{Step 1}. For any $y$, $w \in
S_{2\rho_0}$ and for any $l$, $m \in \R^3$, $|l|=|m|=1$, we have
\begin{equation}
  \label{eq:24.1}
   |f(y,w;l,m)| \leq C   \frac{\epsilon}{\rho_0},
\end{equation}
where $C>0$ only depends on $M_0$, $\alpha$, $M_1$, $\alpha_0$,
$\gamma_0$, $\overline{\lambda}$, $\overline{\mu}$.}

\textit{For any $y \in S_{2\rho_0}$, $w \in \overline{\Omega^{\rho_0}}
\setminus \overline{\Omega_D}$, and for every $l$, $m \in \R^3$, $|l|=|m|=1$,
we have
\begin{equation}
  \label{eq:25.1}
   |f(y,w;l,m)| \leq  \frac{C}{\rho_0},
\end{equation}
where $C>0$ only depends on $M_0$, $\alpha$, $M_1$, $\alpha_0$,
$\gamma_0$, $\overline{\lambda}$, $\overline{\mu}$.}

\medskip
\begin{proof} [Proof of Step 1]
When $y$, $w \in \R^3\setminus \overline{\Omega}$, we may apply
the identity \eqref{eq:13.1} with $u_1(\cdot) =
\Gamma^{D_1}(\cdot,y)l$, $u_2(\cdot) = \Gamma^{D_2}(\cdot,w)m$
obtaining
\begin{equation}
  \label{eq:20.4}
     f(y,w;l,m)= \int_{\partial \Omega} (\Gamma^{D_1}(x,y)l) \cdot
     (\Lambda_{D_1}- \Lambda_{D_2})(
     \Gamma^{D_2}(x,w)m).
\end{equation}

By \eqref{eq:20.4} and by \eqref{eq:4.1}, we have
\begin{equation*}
   |f(y,w;l,m)| \leq \frac{\epsilon}{\rho_0}\|\Gamma^{D_1}(\cdot,y)l\|_{H^{\frac{1}{2}}(\partial\Omega)}
   \|\Gamma^{D_2}(\cdot,w)m\|_{H^{\frac{1}{2}}(\partial\Omega)}.
\end{equation*}
By \eqref{eq:14.4} and \eqref{eq:14.5}, we have
\begin{multline*}
   \|\Gamma^{D_1}(\cdot,y)l\|_{H^{1/2}(\partial\Omega)}\leq
   \|\Gamma^{D_1}(\cdot,y)l\|_{H^1(\partial\Omega)}=\\
   =\left(\int_{\partial\Omega}|\Gamma^{D_1}(x,y)|^2
   +\rho_0^2|\nabla\Gamma^{D_1}(x,y)|^2\right)^\frac{1}{2}\leq\\
   \leq C\left(\int_{\partial\Omega}|x-y|^{-2}+\rho_0^2|x-y|^{-4}\right)^\frac{1}{2}
\end{multline*}
Noticing that, for any $x \in \partial \Omega$, $|x-y| \geq
\rho_0$, and estimating $|\partial \Omega|$
as
\begin{equation}
  \label{eq:24.2}
   |\partial \Omega| \leq C \rho_0^2,
\end{equation}
where $C>0$ only depends on $M_0$, $\alpha$, $M_1$, it follows
that
\begin{equation*}
\|\Gamma^{D_1}(\cdot,y)l\|_{H^{1/2}(\partial\Omega)}\leq C,
\end{equation*}
with $C>0$ only depending on $M_0$, $\alpha$, $M_1$, $\alpha_0$,
$\gamma_0$, $\overline{\lambda}$, $\overline{\mu}$. Since for any
$x \in \partial \Omega$, $|x-w| \geq \rho_0$, a similar estimate
holds for $\|\Gamma^{D_2}(\cdot,w)m\|_{H^{1/2}(\partial\Omega)}$,  and
\eqref{eq:24.1} follows.

Let $y \in S_{2\rho_0}$, $w \in \overline{\Omega^{\rho_0}}
\setminus \overline{\Omega_D}$. By \eqref{eq:14.5} we have
\begin{equation}
  \label{eq:25.2}
   |f(y,w;l,m)| \leq C \sum_{i=1}^2 \int_{D_i}
   |x-y|^{-2}|x-w|^{-2},
\end{equation}
where $C>0$ only depends on $M_0$, $\alpha$, $\alpha_0$,
$\gamma_0$, $\overline{\lambda}$, $\overline{\mu}$. Since
$|x-y|\geq \rho_0$, we have
\begin{equation}
  \label{eq:25.3}
   |f(y,w;l,m)| \leq C\rho_0^{-2} \sum_{i=1}^2 \int_{D_i}
   |x-w|^{-2} \equiv C\rho_0^{-2} (I_1 +I_2).
\end{equation}
Let $R=\hbox{diam}(\Omega) + \rho_0 \leq C \rho_0$, with $C>0$
only depending on $M_0$, $\alpha$, $M_1$. Then, $\Omega \subset
B_R(w)$ and
\begin{equation}
  \label{eq:25.4}
   I_i \leq \int_{B_R(w)} |x-w|^{-2} =2\pi^2 R \leq C\rho_0, \quad
   i=1,2,
\end{equation}
and \eqref{eq:25.1} follows.
\end{proof}

\medskip

\textit{\textbf{Step 2.} For any $\overline{y} \in S_{2\rho_0}$,
for every $l$, $m \in \R^3$, $|l|=|m|=1$, we have
\begin{equation}
  \label{eq:26BIS.1}
    |f(\overline{y},w_h;l,m)| \leq \frac{C}{\rho_0} \epsilon^\eta,
\end{equation}
where
\begin{equation}
  \label{eq:26BIS.2}
    \eta = \beta
    \delta^
    { \frac{ \left |   \log  \frac{\lambda_w h}{d_*}  \right | }{  |\log
    \chi|}+ 1 },
\end{equation}
and $\gamma_0$, $\overline{\lambda}$, $\overline{\mu}$; $\chi$,
$0<\chi<1$, only depends on $M_0$, $\alpha$; $d_*$, $0<d_*<\overline{d}$,
only depends on $M_0$, $\alpha$, $\alpha_0$, $\gamma_0$,
$\overline{\lambda}$, $\overline{\mu}$, where $\overline{d}$ has been
introduced in Lemma \ref{lem:8.1}; $\beta$, $0<\beta<1$, $C>0$
only depend on $M_0$, $\alpha$, $M_1$, $\alpha_0$, $\gamma_0$,
$\overline{\lambda}$, $\overline{\mu}$.}

\medskip
\begin{proof} [Proof of Step 2]

Let us fix $\overline{y}$, $\overline{w} \in S_{2\rho_0}$.

By Lemma \ref{lem:8.1}, there exists a path $\gamma \subset \left
( \overline{\Omega^{\rho_0}} \cup S_{2\rho_0} \right ) \setminus
\overline{\Omega_D}$ joining $\overline{w}$ to $Q=P-\overline{d}e_3$, such
that
\begin{equation}
  \label{eq:26.5}
   V(\gamma)  \subset \R^3 \setminus \Omega_D,
\end{equation}
when $R=\frac{\overline{d}}{\sqrt{1+L_0^2} }$. Note that $\arcsin
\frac{R}{\overline{d}}=\widetilde{\vartheta}_0$ as defined in the
statement of Theorem \ref{theo:1.23}.

Recalling Lemma \ref{lem:21.1}, we know that the vector-valued
function $f=(f_1,f_2,f_3)$, where
$f_k(\cdot;l)=f(\overline{y},\cdot;l,e_k)$, $k=1,2,3$, satisfies
the Lam\'{e} system with constant coefficients
\begin{equation}
  \label{eq:27.2}
  \divrg_w(\C\nabla_w f)=0, \quad \hbox{in } \R^3 \setminus \overline{\Omega_D}.
\end{equation}
At this stage, a basic tool is the following three spheres
inequality for solutions to the Lam\'{e} system \eqref{eq:27.2} in
$B_{\overline{r}}(x) \subset \R^3 \setminus \Omega_D$: there
exists $\vartheta^*$, $0<\vartheta^*\leq 1$, only depending on
$\alpha_0$, $\gamma_0$, $\overline{\lambda}$, $\overline{\mu}$,
such that for every $r_1$, $r_2$, $r_3$,
$0<r_1<r_2<r_3\leq \vartheta^* \overline{r}$, we have
\begin{equation}
  \label{eq:27.3}
   \|f\|_{L^\infty (B_{r_2}(x))} \leq C \|f\|_{L^\infty
   (B_{r_1}(x))}^\delta \cdot \|f\|_{L^\infty
   (B_{r_3}(x))}^{1-\delta},
\end{equation}
where $C>0$ and $\delta$, $0<\delta<1$, only depend on $\alpha_0$,
$\gamma_0$, $\overline{\lambda}$, $\overline{\mu}$, $ \frac{
r_2}{r_3}$,  $ \frac{ r_1}{r_3}$.

Let us choose $r_1= \frac{\vartheta^*\overline{d}}{4}$,
$r_2=3r_1$, $r_3=4r_1$. Let $x_1=\overline{w}$ and let us
define $\{x_i\}$, $i=1,...,s$, as follows: $x_1=\overline{w}$,
$x_{i+1}=\gamma(t_i)$, where $t_i = \max \left \{ t | \ |\gamma
(t)-x_i|=r_1  \right \}$ if $|x_i -Q| > 2r_1$; otherwise, let
$i=s$ and stop the process. By construction, the balls
$B_{r_1}(x_i)$ are pairwise disjoint, $|x_{i+1}-x_i| =2r_1$ for
$i=1,...,s-1$, $|x_s-Q| \leq 2r_1$. Hence, we have
\begin{equation}
  \label{eq:28.0}
   s \leq C \left ( \frac{\rho_0}{r_1} \right )^3\leq C',
\end{equation}
where $C'>0$ only depends on $M_0$, $\alpha$, $M_1$, $\alpha_0$,
$\gamma_0$, $\overline{\lambda}$, $\overline{\mu}$. An iterated
application of \eqref{eq:27.3} and estimates \eqref{eq:24.1},
\eqref{eq:25.1} give
\begin{equation}
  \label{eq:28.1}
    \|f(\cdot;l)\|_{L^\infty (B_{r_1}(Q))}
    \leq
    C\left ( \frac{1}{\rho_0}\right )^{(1-\delta^s)}\cdot
    \|f(\cdot;l)\|_{L^\infty
    (B_{r_1}(\overline{w}))}^{\delta^s}
    \leq \frac{C}{\rho_0} \epsilon^\beta,
\end{equation}
where the constant $C>0$ depends only on $M_0$, $\alpha$, $M_1$,
$\alpha_0$, $\gamma_0$, $\overline{\lambda}$, $\overline{\mu}$,
and the constant $\beta$, $0<\beta<1$, only depends on $M_0$,
$\alpha$, $M_1$, $\alpha_0$, $\gamma_0$, $\overline{\lambda}$,
$\overline{\mu}$.

Let us denote
\begin{equation}
  \label{eq:28.2}
   \lambda_1=\overline{d},
\end{equation}
\begin{equation}
  \label{eq:28.4}
   \vartheta_1= \arcsin \left ( \frac{\sin \widetilde{\vartheta}_0}{4} \right ),
\end{equation}
\begin{equation}
  \label{eq:28.5}
   w_1=Q=P-\lambda_1 e_3,
\end{equation}
\begin{equation}
  \label{eq:28.6}
   \rho_1 = \vartheta^* \lambda_1 \sin \vartheta_1.
\end{equation}
In order to approach $w_h$, we construct a sequence of balls
contained in the cone $C \left ( P, -e_3, \frac{\overline{d}^2-R^2}{\overline{d}},
\arcsin \frac{R}{\overline{d}} \right )$, with $R= \frac{\overline{d}}{ \sqrt{1+L_0^2}
}$, as follows. Let us define, for $k \geq 2$,
\begin{equation}
  \label{eq:29.1}
   w_k=P-\lambda_k e_3,
\end{equation}
\begin{equation}
  \label{eq:29.2}
   \lambda_k= \chi\lambda_{k-1},
\end{equation}
\begin{equation}
  \label{eq:29.3}
   \rho_k= \chi\rho_{k-1},
\end{equation}
with
\begin{equation}
  \label{eq:29.4}
   \chi= \frac{ 1-\sin \vartheta_1  }{1+ \sin \vartheta_1}.
\end{equation}
We have that
\begin{equation}
  \label{eq:29.5}
   \rho_k= \chi^{k-1}\rho_1,
\end{equation}
\begin{equation}
  \label{eq:29.6}
   \lambda_k= \chi^{k-1}\lambda_1,
\end{equation}
\begin{equation}
  \label{eq:29.7}
   B_{\rho_{k+1}}(w_{k+1}) \subset B_{3\rho_k} (w_k).
\end{equation}
Denoting
\begin{equation}
  \label{eq:29.8}
   d(k)=|w_k -P| -\rho_k,
\end{equation}
we have
\begin{equation}
  \label{eq:29.9}
   d(k)= \chi^{k-1}d_*,
\end{equation}
with
\begin{equation}
  \label{eq:29.10}
   d_*=\lambda_1 (1-\vartheta^*\sin \vartheta_1).
\end{equation}
For any $t$, $0<t<d_*$, let $k(t)$ the smallest positive integer
such that $d(k) \leq t$, that is
\begin{equation}
  \label{eq:29.11}
    \frac
    {  \left | \log   \frac{t}{d_*}  \right |  }
    {|\log \chi |}
    \leq k(t)-1
    \leq
    \frac
    {  \left | \log   \frac{t}{d_*}  \right |  }
    {|\log \chi |} +1.
\end{equation}
By applying the three spheres inequality \eqref{eq:27.3} over the balls
centered at $w_j$ with radii $\rho_j$, $3\rho_j$, $4\rho_j$, for
$j=1,...,k(t)-1$, we obtain
\begin{equation}
  \label{eq:30.1}
    \|f(\cdot;l)\|_{L^\infty (B_{\rho_{k(t)}}(w_{k(t)}))}
    \leq
    \frac{C}{\rho_0} \epsilon^{\beta \delta^{k(t)-1}},
\end{equation}
where the constant $C>0$ only depends on $M_0$, $\alpha$, $M_1$,
$\alpha_0$, $\gamma_0$, $\overline{\lambda}$, $\overline{\mu}$. In
particular, in view of \eqref{eq:23.3}, inequality \eqref{eq:30.1}
holds with $t=\lambda_wh$, and we have
\begin{equation}
  \label{eq:30.2}
    |f(w_h;l)| \leq \frac{C}{\rho_0} \epsilon^\eta,
\end{equation}
with $\eta$ given by \eqref{eq:26BIS.2}.

For any $m \in \R^3$, $|m|=1$, by linearity of
$f(\overline{y},w_h;l,m)$ with respect to the last variable, we
have
\begin{equation}
  \label{eq:30.3}
    |f(\overline{y},w_h;l,m)| =\left | \sum_{k=1}^3
    m_kf_k(w_h;l) \right | \leq |m| \cdot
    |f(w_h;l)|,
\end{equation}
and, by \eqref{eq:30.2}, the thesis follows.
\end{proof}

\medskip

At this stage, in order to estimate $f(y_h, w_h; l,m)$ when $y_h$,
$w_h$ are defined by \eqref{eq:23.1}, \eqref{eq:23.2}, we shall
propagate the smallness with respect to the first variable, by
iterating the three spheres inequality over suitable chains of
balls. As in Step i), we need a preliminary rough
estimate of $f(y,w_h;l,m)$ for any $y \in \R^3\setminus \Omega_D$.
However, since such an estimate degenerates when $y$ approaches
$\Omega_D$, we have to restrict our analysis to points $y$
sufficiently far {}from $\Omega_D$. Precisely, we consider the set $\widehat{V}(\gamma)$
obtained reducing the width
of the set $V(\gamma)$ appearing in \eqref{eq:26.5} by replacing in its definition
$R= \frac{\overline{d}}{\sqrt{1+L_0^2} }$ with $\widehat{R}= \frac{\overline{d}}{2\sqrt{1+L_0^2} }$.
Let us denote $\widehat{\vartheta}_0=\arcsin
   \frac{\widehat{R}}{\overline{d}}= \frac{1}{2\sqrt{1+L_0^2}}$.

In Step 4, we shall apply the three-spheres inequality
on a chain of balls contained in $\widehat{V}(\gamma)$ and centered
at points belonging either to the arc $\gamma$ or to the segment
joining $Q$ to $y_h=P- h e_3$. By a straightforward
computation, the distance {}from $\Omega_D$ of the points of all
these balls is, at least,
\begin{equation}
  \label{eq:32.1}
     h \sin \widehat{\vartheta}_0 =  \frac{1}{2}h
    \sin \widetilde{\vartheta}_0.
\end{equation}

\medskip

\textit{\textbf{Step 3.} For any $y\in \overline{\Omega^{\rho_0}}
\setminus \Omega_D^{ h \sin \widehat{\vartheta}_0}$ and for any
$l$, $m \in \R^3$, $|l|=|m|=1$, we have
\begin{equation}
  \label{eq:32.2}
    |f(y,w_h;l,m)| \leq \frac{C}{\rho_0 \lambda_w} \left ( \frac{\rho_0}{ h }
    \right ),
\end{equation}
where $C>0$ only depends on $M_0$, $\alpha$, $M_1$, $\alpha_0$,
$\gamma_0$, $\overline{\lambda}$, $\overline{\mu}$.}

\medskip

\medskip
\begin{proof} [Proof of Step 3]

By \eqref{eq:14.5} and \eqref{eq:20.3} we have
\begin{equation}
  \label{eq:32.3}
    |f(y,w_h;l,m)| \leq C \sum_{i=1}^2 \int_{D_i}
    |x-y|^{-2}|x-w_h|^{-2},
\end{equation}
where $C>0$ only depends on $M_0$, $\alpha$, $\alpha_0$,
$\gamma_0$, $\overline{\lambda}$, $\overline{\mu}$. By H\"{o}lder
inequality we have
\begin{equation}
  \label{eq:32.4}
    \int_{D_i} |x-y|^{-2}|x-w_h|^{-2} \leq \left ( \int_{D_i}
    |x-y|^{-4} \right )^{ \frac{1}{2}} \left ( \int_{D_i}
    |x-w_h|^{-4} \right )^{ \frac{1}{2}},
\end{equation}
$i=1,2$. Let $R=\hbox{diam}(\Omega)+\rho_0 \leq C \rho_0$, with
$C>0$ only depending on $M_0$, $\alpha$, $M_1$. Then, $\Omega
\subset B_R(y)$ and $\Omega \subset B_R(w_h)$. Since $|x-y|\geq \sin
\widehat{\vartheta}_0 h$ for every $x\in\Omega_D$, we have
\begin{equation}
  \label{eq:33.1}
    \int_{D_i}  |x-y|^{-4} \leq
    \int_{B_R(y) \setminus B_{ \sin
    \widehat{\vartheta}_0 h}(y)} |x-y|^{-4} \leq \frac{C}{
    h}, \quad i=1,2,
\end{equation}
where $C>0$ only depends on $M_0$, $\alpha$. Similarly, since $|x-w_h|\geq \lambda_w \sin
\widehat{{\vartheta}}_0 h$, we have
\begin{equation}
  \label{eq:33.2}
    \int_{D_i}  |x-w_h|^{-4} \leq \frac{C}{\lambda_w
    h}, \quad i=1,2,
\end{equation}
and \eqref{eq:32.2} follows.
\end{proof}

\medskip

\textit{\textbf{Step 4.} Conclusion.}

Let $\widetilde{y} \in S_{2\rho_0}$ such that
$\hbox{dist}(\widetilde{y}, \partial \Omega) = \frac{3}{2}\rho_0$,
so that $B_{ \frac{\rho_0}{2}}( \widetilde{y}) \subset
S_{2\rho_0}$ and, by \eqref{eq:26BIS.1} of Step 2,
\begin{equation}
  \label{eq:34.1}
    \|f(\cdot ,w_h;l,m)\|_
    {L^\infty ( B_{  \frac{ \rho_0}{2}}(\widetilde{y} ))}
    \leq
    \frac{C}{\rho_0} \epsilon^
    \eta,
\end{equation}
where
\begin{equation}
  \label{eq:34.1BIS}
    \eta =
    \beta
    \delta^{\frac{ \left | \log   \frac{\lambda_w h}{d_*} \right |    }{ |\log \chi|  } +1}.
\end{equation}
By Lemma \ref{lem:8.1}, there exists a path $\gamma \subset (
\overline{\Omega^{\rho_0}} \cup S_{2\rho_0} ) \setminus
\overline{\Omega}_D$ joining $\widetilde{y}$ to $Q=P-\overline{d}e_3$, such
that $V(\gamma) \subset \R^3 \setminus \Omega_D$, where $V(\gamma)$ is
defined by \eqref{eq:26.5}.

By Lemma \ref{lem:21.1}, the vector-valued function
$\widetilde{f}=(\widetilde{f}_1, \widetilde{f}_2,
\widetilde{f}_3)$, where $\widetilde{f}_k(\cdot,m)=f(\cdot,
w_h;e_k,m)$, $k=1,2,3$, satisfies the Lam\'{e} system
\eqref{eq:27.2}, with constant coefficients
$\lambda$, $\mu$. Then, we can repeat the propagation of smallness
arguments of Step 2 with the following modifications
\begin{equation}
  \label{eq:34.2}
   r_1= \frac{\vartheta^*\overline{d}}{8}, \quad \vartheta_1=\arcsin
    \left ( \frac{\sin \widetilde{\vartheta}_0}{8} \right ),
\end{equation}
ensuring that the
geometrical construction takes place inside the set $\widehat{V}(\gamma)$ as specified
in the previous Step 2. Therefore, estimate \eqref{eq:32.2} holds for
every point $y$ belonging to the chain of balls.

By repeating the arguments of Step 2, in view of the estimates \eqref{eq:32.2}
and \eqref{eq:34.1}, the analogous of
\eqref{eq:28.1} becomes
\begin{equation}
  \label{eq:35.1}
    \|\widetilde{f}(\cdot,w_h)\|_{L^\infty (B_{\rho}(Q))}
    \leq
    \frac{C}{ \rho_0 \lambda_w  }
    \left ( \frac{\rho_0}{h} \right )
    \epsilon^
    {\widetilde{\eta}
    },
\end{equation}
where
\begin{equation}
  \label{eq:35.1BIS}
    \widetilde{\eta}=
    \widetilde{\beta}
    \delta^
    {
    \frac{ \left |  \log   \frac{ \lambda_w h  }{d_*} \right |  }
    { |\log \chi|  }+1
    },
\end{equation}
and $C>0$ and $\widetilde{\beta}$, $0<\widetilde{\beta}<1$, only
depend on $M_0$, $\alpha$, $M_1$, $\alpha_0$, $\gamma_0$,
$\overline{\lambda}$, $\overline{\mu}$, and $\delta$, $\chi$,
$d_*$ are the quantities appearing in \eqref{eq:26BIS.1}.

Finally, by adapting the geometrical construction seen above to
the chain of balls joining $Q$ to $y_h$ inside $\widehat{V}(\gamma)$,
recalling \eqref{eq:32.1}
and noticing that the new values of $\chi$ and $d_*$ are bigger
than the previous ones, we have
\begin{equation}
  \label{eq:35.2}
    |\widetilde{f}(y_h,w_h)|
    \leq
    \frac{C}{\rho_0 \lambda_w }
    \left ( \frac{\rho_0}{h} \right )
    \epsilon^
    {
    \widehat{\eta}},
\end{equation}
where
\begin{equation}
  \label{eq:35.2BIS}
   \widehat{\eta} =
    \gamma
    \delta^
    {
    2 +\frac{2}{|\log \chi|}
    \left |  \log   \frac{ \lambda_w h  }{ \left ( 1-\vartheta^* \frac{\sin \widetilde{\vartheta}_0}{8} \right )\overline{d}   } \right |
     },
\end{equation}
and $\chi = \frac{ 1 - \frac{\sin \widetilde{\vartheta}_0}{8}
}{    1 + \frac{\sin \widetilde{\vartheta}_0}{8}     }$; $\gamma
>0$ only depends on $M_0$, $\alpha$, $M_1$, $\alpha_0$,
$\gamma_0$, $\overline{\lambda}$, $\overline{\mu}$; and $C>0$ only
depends on $M_0$, $\alpha$, $M_1$, $\alpha_0$, $\gamma_0$,
$\overline{\lambda}$, $\overline{\mu}$.

By linearity of $\widetilde{f}(y,w;l,m)$ with respect to the third
argument, the bound \eqref{eq:35.2} holds also for $f(y_h, w_h;l, m)$,
for every $l$, $m \in \R^3$, $|l|=|m|=1$.

Introducing
\begin{equation}
  \label{eq:36.1}
   A= \frac{ \lambda_w   }{ \frac{\overline{d}}{\rho_0} (1
- \vartheta^* \frac{ \sin \widetilde{\vartheta}_0}{8}) },\qquad B= \frac{2}{|\log \chi|},
\end{equation}
we may rewrite the second factor in the right hand side of
\eqref{eq:35.2} as
\begin{equation}
  \label{eq:36.3}
   \epsilon^
   {
   \gamma
   \delta^2
   \delta^
   {B \left| \log A\left ( \frac{h}{\rho_0}\right ) \right |}
   }
   \leq
   \epsilon^
   {\gamma
   \delta^{2 +B |\log A|}
   \left (
   \frac{h}{\rho_0}
   \right )^{B|\log \delta|}
   },
\end{equation}
which gives the desired estimate \eqref{eq:23.4}.

\section{Asymptotics of $\Gamma^D$}
\label{asymptotic}

Given a bounded domain $D$ with boundary $\partial D$ of class
$C^{1,\alpha}$, with constants $\rho_0$, $M_0$, $0<\alpha \leq 1$,
let $O \in \partial D$ and $\nu=\nu(O)$ the outer unit normal to
$D$ at $O$.

Let us choose a coordinate system with origin $O$ and axis $e_3 =
-\nu$, and let $\Gamma^+(x,y)=\Gamma^{\R^3_+}(x,y)$ the
normalized fundamental matrix associated to $D=\R^3_+$.
We recall that its explicit expression was found by Rongved \cite{Ron55}.
See also Section \ref{Rongved} where Rongved's formulas shall be used in more detail.

Recalling the notation $u^D(x,y)=\Gamma^D(x,y)l$ (see
\eqref{eq:14.2}) and defining similarly $u^+(x,y)=\Gamma^+(x,y)l$,
for any $l\in \R^3$, $|l|=1$, let us prove an asymptotic
approximation of $u^D$ in terms of $u^+$.

\begin{theo}
   \label{theo:1.38}
Let $y=(0,0,-h)$, $0<h<\frac{\rho_0 M_0}{8\sqrt{1+M_0^2}}$. Under the above assumptions and notation, we have
\begin{equation}
  \label{eq:38.1}
   |u^D(x,y)-u^+(x,y)| \leq \frac{C}{\rho_0}
   \left (
   \frac{|x-y|}{\rho_0}
   \right )^{-1+\alpha},
   \quad \forall x \in Q_{ \frac{\rho_0}{8 \sqrt{1+M_0^2}}, \frac{\rho_0 M_0}{8
   \sqrt{1+M_0^2}}
   }\cap D,
\end{equation}
\begin{equation}
  \label{eq:38.2}
   |\nabla_x u^D(x,y)-\nabla_x u^+(x,y)| \leq \frac{C}{\rho_0^2}
   \left (
   \frac{|x-y|}{\rho_0}
   \right )^{-2 +  \frac{\alpha^2}{3\alpha+2}  },\quad
   \forall x \in Q^+_{ \frac{\rho_0}{12 \sqrt{1+M_0^2}}, \frac{\rho_0 M_0}{12
   \sqrt{1+M_0^2}}
   }\cap D,
\end{equation}
where $C>0$ only depends on $M_0$, $\alpha$, $\alpha_0$,
$\gamma_0$, $\overline{\lambda}$, $\overline{\mu}$.
\end{theo}

\begin{proof}

Let us set
\begin{equation}
  \label{eq:AGG21.1}
   R(x,y)=u^D(x,y)-u^+(x,y).
\end{equation}
The estimate of $R$ is based on a local flattening of the boundary
$\partial D$, which is realized through the following
transformation $\Phi$ (see, for instance, \cite{ADiC05}).

Let us consider a cut-off function $\vartheta \in C^\infty(\R)$
such that $0 \leq \vartheta (t) \leq 1$ in $\R$, $\vartheta=1$ if
$|t| \leq 1$, $\vartheta (t)=0$ if $|t| \geq 2$, $|\vartheta' (t)|
\leq 2$ and $|\vartheta '' (t)| \leq 4$  for every $t\in\R$. Let
\begin{equation}
  \label{eq:39.0}
   \rho=
\frac{\rho_0}{ \sqrt{1+M_0^2}  }=\rho_0 \cos \vartheta_0,
\end{equation}
where
$\tan \vartheta_0 =M_0$. Let us introduce the
following transformation
\begin{equation}
  \label{eq:39.1}
   \Phi: \R^3 \rightarrow \R^3,
\end{equation}
\begin{center}
\( {\displaystyle \left\{
\begin{array}{lr}
  \xi_1=x_1,
    \vspace{0.25em}\\
  \xi_2=x_2,
    \vspace{0.25em}\\
  \xi_3=x_3 -\varphi(x_1,x_2)\vartheta \left ( \frac{x_3}{5\rho
  M_0} \right ) \vartheta \left ( \frac{|x'|}{\rho} \right ),
\end{array}
\right. } \) \vskip -2.4em
\begin{eqnarray}
& & \label{eq:39.2}
\end{eqnarray}
\end{center}
where $\varphi$ is the $C^{1,\alpha}$ function that represents
locally $\partial D$.

It is easy to prove that $\Phi$ is a $C^{1,\alpha}$-diffeomorphism
satisfying the following properties:

\begin{subequations}
\begin{eqnarray}
\label{eq:39.3} \Phi=Id \quad \hbox{in } \R^3 \setminus Q_{2\rho, 10\rho
   M_0}, \\[2mm]
\label{eq:39.4}  \Phi ( Q_{2\rho, 10\rho M_0}) = Q_{2\rho, 10\rho M_0};
\end{eqnarray}
\end{subequations}
\begin{subequations}
\begin{eqnarray}
\label{eq:39BIS.-1}  \Phi(x) =(x', x_3-\varphi(x')), \quad \hbox{in }
   Q_{\rho,\rho M_0},\\[2mm]
\label{eq:39BIS.0} \Phi( Q_{\rho,\rho M_0} \cap \partial D) =
    Q_{\rho,\rho M_0} \cap \{x_3=0\},\\[2mm]
    \label{eq:39BIS.1}
   \Phi(Q_{\rho,\rho M_0} \cap D ) \supset
   \{(\xi',\xi_3)| \ |\xi'|< \rho, \ 0<\xi_3<M_0(\rho-|\xi'|) \},\\[2mm]
   \label{eq:39BIS.2}
   \Phi(Q_{\rho,\rho M_0} \setminus \overline{D} ) \supset
   \{(\xi',\xi_3)| \ |\xi'|< \rho, \ -M_0(\rho -|\xi'|)<\xi_3<0 \};
\end{eqnarray}
\end{subequations}
\medskip
\begin{equation}
  \label{eq:39.6}
   c^{-1}|x-\widetilde{x}|\leq
   |\Phi(x)-\Phi(\widetilde{x})| \leq c|x-\widetilde{x}|, \quad
   \hbox{for every } x, \ \widetilde{x} \in \R^3;
\end{equation}
\medskip
\begin{equation}
  \label{eq:39.7}
   |\Phi(x) - x | \leq \frac{c}{\rho_0^\alpha}
   |x'|^{1+\alpha}, \quad  \hbox{for every } x\in \R^3;
\end{equation}

\begin{equation}
  \label{eq:40.0}
   |J(x)- Id| \leq \frac{c}{\rho_0^\alpha}
   |x'|^{\alpha} , \quad  \hbox{for every } x\in \R^3,
\end{equation}
where $J(x)=\nabla\Phi(x)$ and $c>0$ is a constant only depending on
$M_0$.

Denoting
\begin{equation}
  \label{eq:40.1}
   \xi=\Phi(x), \quad \eta=\Phi(y)=y,
\end{equation}
and defining
\begin{equation}
  \label{eq:40.2}
   \widetilde{\Gamma}^D(\xi,\eta)=\Gamma^D(x,y), \quad
   \widetilde{\chi}_D(\xi)=\chi_D(x), \quad
   \widetilde{J}(\xi)=J(x),
\end{equation}
we have that
\begin{multline}
  \label{eq:40.3}
   \divrg_\xi
   \left \{
   \left [
    (\C +(\C^D-\C)\widetilde{\chi}_D(\xi))
    (\nabla_\xi (\widetilde{\Gamma}^D(\xi,\eta)l)
    \widetilde{J}(\xi) )
    \right ]
    \frac{ \widetilde{J}^T(\xi)    }{ \det  \widetilde{J}(\xi) }
    \right \}
    =
    -l \delta(\xi-\eta), \\
    \hbox{for every } \xi\in \R^3 \ \hbox{and for every } l\in\R^3, \ |l|=1.
\end{multline}
By \eqref{eq:39BIS.1}, \eqref{eq:39BIS.2},
$\widetilde{\chi}_D(\xi)=\chi_+(\xi)$ for every $\xi \in
\{(\xi',\xi_3)| \ |\xi'|<\rho, \ |\xi_3|<M_0(\rho-|\xi'|)\}$ and,
therefore,
\begin{multline}
  \label{eq:41.1}
   \divrg_\xi
   \left \{
   \left [
    (\C +(\C^D-\C){\chi}_+(\xi))
    (\nabla_\xi (\widetilde{\Gamma}^D(\xi,\eta)l)
    \widetilde{J}(\xi) )
    \right ]
    \frac{ \widetilde{J}^T(\xi)    }{ \det  \widetilde{J}(\xi) }
    \right \}
    =
    -l \delta(\xi-\eta), \\
    \hbox{in } Q_{ \frac{\rho}{2}, \frac{\rho
    M_0}{2}} \ \hbox{and for every } l\in\R^3, |l|=1.
\end{multline}
Let us define the function
\begin{equation}
  \label{eq:41.2}
   \widetilde{R}(\xi, \eta)=
   \widetilde{u}^D(\xi,\eta)-u^+(\xi,\eta), \quad \hbox{in } Q_{ \frac{\rho}{2}, \frac{\rho
    M_0}{2}},
\end{equation}
where
\begin{equation}
  \label{eq:41.3}
   \widetilde{u}^D(\xi,\eta)= \widetilde{\Gamma}^D(\xi, \eta)l,
\end{equation}
\begin{equation}
  \label{eq:41.4}
   u^+(\xi,\eta)= \Gamma^+(\xi, \eta)l.
\end{equation}
The function $\widetilde{R}(\xi, \eta)$ satisfies the equation
\begin{equation}
  \label{eq:41.5}
   \divrg_\xi
   \left [
    (\C +(\C^D-\C){\chi}_+(\xi))
    \nabla_\xi \widetilde{R}(\xi,\eta)
    \right ]
    =
    \mathcal{F}(\xi,\eta), \quad \hbox{in } Q_{ \frac{\rho}{2}, \frac{\rho
    M_0}{2}},
\end{equation}
where
\begin{multline}
  \label{eq:41.6}
   \mathcal{F}(\xi,\eta)
   =
   \divrg_\xi
   \left \{
   \left [
    (\C +(\C^D-\C){\chi}_+(\xi))
    \nabla_\xi \widetilde{u}^D(\xi,\eta)
    \right ]
    \left (
    Id - \frac{ \widetilde{J}^T(\xi)    }{ \det  \widetilde{J}(\xi) }
    \right )
    \right \}
    - \\
    -  \divrg_\xi
   \left \{
   \left [
    (\C +(\C^D-\C){\chi}_+(\xi))
    (\nabla_\xi \widetilde{u}^D(\xi,\eta) (
    \widetilde{J}(\xi)-Id))
    \right ]
    \frac{ \widetilde{J}^T(\xi)    }{ \det  \widetilde{J}(\xi) }
    \right \}.
    \end{multline}
We want to estimate $\widetilde{R}(z, \eta)$ for $z \in Q_{
\frac{\rho}{4}, \frac{\rho M_0}{4}}^+$ and $\eta=(0,0,-h) \in Q_{
\frac{\rho}{4}, \frac{\rho M_0}{4}}^-$. Using Green's formulas one finds
\begin{equation}
  \label{eq:43.1}
    \widetilde{R}(z,\eta) \cdot m = I_1 +I_2 +I_3 +I_4 +I_5,
\end{equation}
where
\begin{equation}
  \label{eq:43.2}
    I_1 =
    - \int_{\partial Q_{ \frac{\rho}{2}, \frac{\rho M_0}{2}}}
    \Gamma^+(\xi,z)m
    \cdot
    \left \{
    \left [
    \widetilde{\C}(\xi) \nabla_\xi \widetilde{u}^D(\xi, \eta)
    \right ]
    \left (
    Id - \frac{\widetilde{J}^T(\xi)   }{\det \widetilde{J}(\xi) }
    \right )
    \right \}\nu_\xi \ d\sigma_\xi,
\end{equation}
\begin{equation}
  \label{eq:43.3}
    I_2 =
    \int_{\partial Q_{ \frac{\rho}{2}, \frac{\rho M_0}{2}}}
    \Gamma^+(\xi,z)m
    \cdot
    \left \{
    \left [
    \widetilde{\C}(\xi)
    \left (
    \nabla_\xi \widetilde{u}^D(\xi, \eta)
    (\widetilde{J}(\xi)  - Id)
    \right )
    \right ]
    \frac{\widetilde{J}^T(\xi)   }{\det \widetilde{J}(\xi) }
    \right \}\nu_\xi \ d\sigma_\xi,
\end{equation}
\begin{equation}
  \label{eq:43.4}
    I_3 =
    -\int_{\partial Q_{ \frac{\rho}{2}, \frac{\rho M_0}{2}}}
    \left [
    \widetilde{R}(\xi,\eta) \cdot
    ( \widetilde{\C}(\xi) \nabla_\xi (\Gamma^+(\xi, z)m))\nu_\xi
    -
    \Gamma^+(\xi,z)m \cdot
    ( \widetilde{\C}(\xi) \nabla_\xi \widetilde{R}(\xi,\eta))\nu_\xi
    \right ] d\sigma_\xi,
\end{equation}
\begin{equation}
  \label{eq:43.5}
    I_4 =
    \int_{Q_{ \frac{\rho}{2}, \frac{\rho M_0}{2}}}
    \nabla_\xi (\Gamma^+(\xi,z)m)
    \cdot
    \left \{
    \left [
    \widetilde{\C}(\xi) \nabla_\xi \widetilde{u}^D(\xi, \eta)
    \right ]
    \left (
    Id - \frac{\widetilde{J}^T(\xi)   }{\det \widetilde{J}(\xi) }
    \right )
    \right \}d\xi,
\end{equation}
\begin{equation}
  \label{eq:43.6}
    I_5 =
    - \int_{Q_{ \frac{\rho}{2}, \frac{\rho M_0}{2}}}
    \nabla_\xi (\Gamma^+(\xi,z)m)
    \cdot
    \left \{
    \left [
    \widetilde{\C}(\xi)
    \left (
    \nabla_\xi \widetilde{u}^D(\xi, \eta)
    (\widetilde{J}(\xi)  - Id)
    \right )
    \right ]
    \frac{\widetilde{J}^T(\xi)   }{\det \widetilde{J}(\xi) }
    \right \}d\xi,
\end{equation}
where $\widetilde{\C}=\C + (\C^D -\C)\chi_+$.

Let us estimate $I_1$. By Proposition \ref{prop:14.1} and by the
properties of the transformation $\Phi$ defined by
\eqref{eq:39.2}, we have
\begin{equation}
  \label{eq:44.1}
    |I_1| \leq
    C\rho_0^{-\alpha}
    \int_{\partial Q_{ \frac{\rho}{2}, \frac{\rho M_0}{2}}}
    |\xi-z|^{-1} |\xi-\eta|^{-2} |\xi|^{\alpha}
    d\sigma_\xi,
\end{equation}
where $C>0$ only depends on $M_0$, $\alpha$, $\alpha_0$,
$\gamma_0$, $\overline{\lambda}$, $\overline{\mu}$. Since $\xi \in
\partial Q_{ \frac{\rho}{2}, \frac{\rho M_0}{2}}$ and $z$, $\eta
\in Q_{ \frac{\rho}{4}, \frac{\rho M_0}{4}}$, then
\begin{equation}
  \label{eq:44.2}
    |\xi-\eta|\geq C\rho, \quad |\xi-z| \geq C\rho,
\end{equation}
with $C= \frac{1}{4} \min\{1,M_0\}$. Therefore,
\begin{equation}
  \label{eq:44.3}
    |I_1| \leq
    \frac{C}{\rho_0}\left(\frac{\rho}{\rho_0}\right)^{\alpha-1},
\end{equation}
where $C>0$ only depends on $M_0$, $\alpha$, $\alpha_0$,
$\gamma_0$, $\overline{\lambda}$, $\overline{\mu}$.

Similarly, one finds
\begin{equation}
  \label{eq:44.4}
    |I_2| \leq
    \frac{C}{\rho_0}\left(\frac{\rho}{\rho_0}\right)^{\alpha-1},
\end{equation}
\begin{equation}
  \label{eq:44.5}
    |I_3| \leq
    \frac{C}{\rho_0}\left(\frac{\rho}{\rho_0}\right)^{\alpha-1},
\end{equation}
where $C>0$ only depends on $M_0$, $\alpha$, $\alpha_0$,
$\gamma_0$, $\overline{\lambda}$, $\overline{\mu}$.

In order to estimate $I_4$, let us write
\begin{equation}
  \label{eq:44.6}
    I_4 = I'_4 + I''_4,
\end{equation}
where
\begin{equation}
  \label{eq:44.6bis}
    I_4' =
    \int_{Q_{ \frac{\rho}{2}, \frac{\rho M_0}{2}} \setminus Q_{ \frac{3\rho}{8}, \frac{3\rho M_0}{8}}     }
    \nabla_\xi (\Gamma^+(\xi,z)m)
    \cdot
    \left \{
    \left [
    \widetilde{\C}(\xi) \nabla_\xi \widetilde{u}^D(\xi, \eta)
    \right ]
    \left (
    Id - \frac{\widetilde{J}^T(\xi)   }{\det \widetilde{J}(\xi) }
    \right )
    \right \}d\xi,
\end{equation}
\begin{equation}
  \label{eq:44.6ter}
    I_4'' =
    \int_{ Q_{ \frac{3\rho}{8}, \frac{3\rho M_0}{8}}     }
    \nabla_\xi (\Gamma^+(\xi,z)m)
    \cdot
    \left \{
    \left [
    \widetilde{\C}(\xi) \nabla_\xi \widetilde{u}^D(\xi, \eta)
    \right ]
    \left (
    Id - \frac{\widetilde{J}^T(\xi)   }{\det \widetilde{J}(\xi) }
    \right )
    \right \}d\xi.
\end{equation}
Arguing similarly as in estimating $I_1$, we have
\begin{equation}
  \label{eq:45.1}
    |I_4'| \leq
    \frac{C}{\rho_0}\left(\frac{\rho}{\rho_0}\right)^{\alpha-1},
\end{equation}
where $C>0$ only depends on $M_0$, $\alpha$, $\alpha_0$,
$\gamma_0$, $\overline{\lambda}$, $\overline{\mu}$.
Note that in the above formulas the factor $\frac{\rho}{\rho_0}$
is a constant which only depends on $M_0$. We found it convenient for the following calculations to keep such a constant factor in evidence (see \eqref{eq:47.2} below).

For
simplicity, let us denote $Q=Q_{ \frac{3\rho}{8}, \frac{3\rho
M_0}{8}}$. We have
\begin{equation}
  \label{eq:45.2}
    I_4''=
    I_{4a}'' + I_{4b}'',
\end{equation}
where
\begin{equation}
  \label{eq:45.2bis}
    I_{4a}''=
    \int_{ Q \cap \{ |\xi|<4|\eta-z|\}}
    \nabla_\xi (\Gamma^+(\xi,z)m)
    \cdot
    \left \{
    \left [
    \widetilde{\C}(\xi) \nabla_\xi \widetilde{u}^D(\xi, \eta)
    \right ]
    \left (
    Id - \frac{\widetilde{J}^T(\xi)   }{\det \widetilde{J}(\xi) }
    \right )
    \right \}d\xi,
\end{equation}
\begin{equation}
  \label{eq:45.2ter}
    I_{4b}''=
    \int_{ Q \cap \{ |\xi|\geq4|\eta-z|\}}
    \nabla_\xi (\Gamma^+(\xi,z)m)
    \cdot
    \left \{
    \left [
    \widetilde{\C}(\xi) \nabla_\xi \widetilde{u}^D(\xi, \eta)
    \right ]
    \left (
    Id - \frac{\widetilde{J}^T(\xi)   }{\det \widetilde{J}(\xi) }
    \right )
    \right \}d\xi.
\end{equation}
By Proposition \ref{prop:14.1} and introducing the change of
variables
\begin{equation}
  \label{eq:45.3}
    \xi=|z-\eta|w,
\end{equation}
we have
\begin{equation}
  \label{eq:45.4}
    |I_{4a}''|
    \leq
    C \rho_0^{-\alpha} |z-\eta|^{\alpha-1}
    \int_{|w|<4}
    \left |
    w - \frac{z}{|z-\eta|}
    \right |^{-2}
    \left |
    w - \frac{\eta}{|z-\eta|}
    \right |^{-2} dw.
\end{equation}
Since the integral on the right hand side is bounded by an
absolute constant, see, for instance, \cite[Chapter 2, Section
11]{Mir70} , we have
\begin{equation}
  \label{eq:45.5}
    |I_{4a}''| \leq
    \frac{C}{\rho_0}\left(\frac{|z-\eta|}{\rho_0}\right)^{\alpha-1},
\end{equation}
where $C>0$ only depends on $M_0$, $\alpha$, $\alpha_0$,
$\gamma_0$, $\overline{\lambda}$, $\overline{\mu}$.

By Proposition \ref{prop:14.1}, and noticing that $Q \subset \{
|\xi| < \frac{3}{8}\rho \sqrt{1+M_0^2} < 2\rho \sqrt{1+M_0^2}\}$
and that $4|z-\eta| < 2\rho \sqrt{1+M_0^2}$, we have
\begin{equation}
  \label{eq:46.1}
    |I_{4b}''| \leq
    C\rho_0^{-\alpha}
    \int_{ 4|\eta-z| \leq |\xi| \leq 2\rho \sqrt{1+M_0^2}   }
    |\xi-z|^{-2} |\xi-\eta|^{-2} |\xi|^{\alpha}
    d\xi,
\end{equation}
where $C>0$ only depends on $M_0$, $\alpha$, $\alpha_0$,
$\gamma_0$, $\overline{\lambda}$, $\overline{\mu}$. Since,
trivially, $|z-\eta|\geq |z|$, $|z-\eta|\geq |\eta|$, we have
\begin{equation}
  \label{eq:46.2}
    |\xi| \leq |\xi-\eta|+|\eta|\leq |\xi - \eta|+ |z-\eta| \leq |\xi -
    \eta| + \frac{|\xi|}{4},
\end{equation}
so that
\begin{equation}
  \label{eq:46.3}
    |\xi| \leq \frac{4}{3}  |\xi-\eta|.
\end{equation}
Similarly,
\begin{equation}
  \label{eq:46.4}
    |\xi| \leq \frac{4}{3}  |\xi-z|.
\end{equation}
By inserting \eqref{eq:46.3}, \eqref{eq:46.4} in \eqref{eq:46.1},
we have
\begin{equation}
  \label{eq:46.5}
    |I_{4b}''| \leq
    C\rho_0^{-\alpha}
    \int_{ 4|\eta-z| \leq |\xi| \leq 2\rho \sqrt{1+M_0^2}   }
    |\xi|^{\alpha -4}
    d\xi
    \leq
    \frac{C}{\rho_0}
    \left (
    \frac{|z-\eta|}{\rho_0}
    \right )^{\alpha -1},
\end{equation}
where $C>0$ only depends on $M_0$, $\alpha$, $\alpha_0$,
$\gamma_0$, $\overline{\lambda}$, $\overline{\mu}$.

By \eqref{eq:45.1}, \eqref{eq:45.5}, \eqref{eq:46.5}
and taking into account that
\begin{equation}
  \label{eq:47.2}
    |z-\eta|< \frac{\rho \sqrt{1+M_0^2}}{2}=\frac{\rho_0}{2},
\end{equation}
we have
\begin{equation}
  \label{eq:47.3}
    |I_4| \leq \frac{C}{\rho_0}
    \left(\frac{|z-\eta|}{\rho_0}\right)^{\alpha -1},
\end{equation}
where $C>0$ only depends on $M_0$, $\alpha$, $\alpha_0$,
$\gamma_0$, $\overline{\lambda}$, $\overline{\mu}$.

An estimate analogous to \eqref{eq:47.3} holds for $I_5$.

By the above estimates and choosing in \eqref{eq:43.1} $m = \frac{
\widetilde{R} (z,\eta)}{|\widetilde{R} (z,\eta)|}$, for every $z
\in  Q_{\frac{\rho}{4}, \frac{\rho M_0}{4}}^+$ and for every
$\eta=(0,0,-h) \in Q_{\frac{\rho}{4}, \frac{\rho M_0}{4}}^-$ we
have
\begin{equation}
  \label{eq:47.5}
    |\widetilde{R} (z,\eta)| \leq   \frac{C}{\rho_0}
    \left (
    \frac{|z-\eta|}{\rho_0}
    \right )^{\alpha -1},
\end{equation}
where $C>0$ only depends on $M_0$, $\alpha$, $\alpha_0$,
$\gamma_0$, $\overline{\lambda}$, $\overline{\mu}$. Notice that,
by \eqref{eq:47.2}, $2\frac{|z-\eta|}{\rho_0}<1$.

Let us get back to the original variables in order to obtain the
estimate for $R(x,y)$. Since
\begin{equation}
  \label{eq:48.0}
    \Phi^{-1}(\xi) = (\xi', \xi_3
+\varphi(\xi')), \quad \hbox{ in } \Phi(Q_{\rho,\rho M_0}),
\end{equation}
it easy to see that
\begin{equation}
  \label{eq:48.1}
    \Phi^{-1}(Q_{\frac{\rho}{4}, \frac{\rho M_0}{4}} \cap
    \{x_3=0\}) = Q_{\frac{\rho}{4},\frac{\rho M_0}{4} } \cap \partial D,
\end{equation}
\begin{equation}
  \label{eq:48.2}
    \Phi^{-1}(Q_{\frac{\rho}{4}, \frac{\rho M_0}{4}}^+) \subset D,
\end{equation}
\begin{equation}
  \label{eq:48.3}
    \Phi^{-1}(Q_{\frac{\rho}{4}, \frac{\rho M_0}{4}}^-) \cap D =\emptyset,
\end{equation}
\begin{equation}
  \label{eq:48.4}
    Q_{\frac{\rho}{8}, \frac{\rho M_0}{8}} \subset \Phi^{-1}
    (Q_{\frac{\rho}{4}, \frac{\rho M_0}{4}}),
\end{equation}
\begin{equation}
  \label{eq:48.5}
    Q_{\frac{\rho}{8}, \frac{\rho M_0}{8}} \cap D \subset \Phi^{-1}
    (Q_{\frac{\rho}{4}, \frac{\rho M_0}{4}}^+),
\end{equation}
\begin{equation}
  \label{eq:48.6}
    Q_{\frac{\rho}{8}, \frac{\rho M_0}{8}} \setminus \overline{D} \subset \Phi^{-1}
    (Q_{\frac{\rho}{4}, \frac{\rho M_0}{4}}^-).
\end{equation}
Therefore, for every $x \in Q_{\frac{\rho}{8}, \frac{\rho M_0}{8}}
\cap D$ and for every $y \in Q_{\frac{\rho}{8}, \frac{\rho
M_0}{8}} \setminus \overline{D}$, $y=(0,0,-h)$, with $h \in \left
(0, \frac{\rho M_0}{8} \right )$, we have
\begin{equation}
  \label{eq:48.7}
    |\widetilde{R} (\Phi(x),\Phi(y))| \leq   \frac{C}{\rho_0}
    \left (
    \frac{|\Phi(x)-\Phi(y)|}{\rho_0}
    \right )^{\alpha -1},
\end{equation}
where $C>0$ only depends on $M_0$, $\alpha$, $\alpha_0$,
$\gamma_0$, $\overline{\lambda}$, $\overline{\mu}$.

Recalling that $\Phi(y)=y$, we have
\begin{multline}
  \label{eq:49.1}
    R(x,y)=
    u^D(x,y)-u^+(\Phi(x),\Phi(y))+u^+(\Phi(x),\Phi(y))-u^+(x,y) =
    \\
    =
    \widetilde{R}(\Phi(x),\Phi(y)) + \left (
    u^+(\Phi(x),\Phi(y))-u^+(x,y) \right ).
\end{multline}
In order to estimate the second addend in \eqref{eq:49.1}, let us
distinguish two cases.

\textit{Case i)}: $x_3 \geq 0$, that is $x \in
\overline{\R^{3+}}$.

By the results in \cite{LN03}, $u^+ (\cdot,
y) \in C^1( \overline{\R^{3+}} )$ and, since the segment $S=[x,
\Phi(x)]$ is contained in $\overline{\R^{3+}}$, by
\eqref{eq:14.5}, \eqref{eq:39.7} and noticing that $|x'| \leq
|x-y|$, we have
\begin{multline}
  \label{eq:49.2}
    |u^+(\Phi(x),\Phi(y))-u^+(x,y)| \leq \| \nabla u^+(\cdot,
    y)\|_{L^\infty (S)}|\Phi(x)-x| =
    \\
    = |\nabla u^+ (w,y)|\cdot |\Phi(x)-x| \leq
    C\rho_0^{-\alpha}|w-y|^{-2}|x-y|^{1+\alpha},
\end{multline}
where $w \in [x,\Phi(x)]$ and $C>0$ only depends on $M_0$,
$\alpha$, $\alpha_0$, $\gamma_0$, $\overline{\lambda}$,
$\overline{\mu}$.

Since either $|w-y| \geq |x-y|$ or $|w-y| \geq |\Phi(x)-y|$
and, by \eqref{eq:39.6}, $|\Phi(x)-y|=|\Phi(x)-\Phi(y)|\geq
C|x-y|$, we have that
\begin{equation}
  \label{eq:50.1}
    |u^+(\Phi(x),\Phi(y))-u^+(x,y)| \leq \frac{C}{\rho_0}
    \left (
    \frac{|x-y|}{\rho_0}
    \right )^{-1+\alpha},
\end{equation}
where $C>0$ only depends on $M_0$, $\alpha$, $\alpha_0$,
$\gamma_0$, $\overline{\lambda}$, $\overline{\mu}$.

\medskip

\textit{Case ii)}: $x_3 < 0$, that is $x \in \overline{\R^{3-}}$.

Let $\widetilde{x}=(x',\varphi(x'))\in \partial D$ and
$\overline{x}=\Phi(\widetilde{x}) = (x',0)$. By using the
$C^1$-regularity of $u^+$ separately in $\overline{\R^{3+}}$ and
in $\overline{\R^{3-}}$, we have
\begin{multline}
  \label{eq:50.2}
    |u^+(\Phi(x),y)-u^+(x,y)| \leq |
    u^+(\Phi(x),y)-u^+(\overline{x},y)|+
    |u^+(\overline{x},y)-u^+(x,y)| \leq
    \\
    \leq |\nabla u^+ (w^+,y)| \cdot |\Phi(x)-\overline{x}| +
   |\nabla u^+ (w^-,y)| \cdot |\overline{x}-x|,
\end{multline}
where $w^+ \in S^+=[\Phi(x), \overline{x}]$,  $w^- \in
S^-=[\overline{x}, x]$. Let $\widetilde{w}^+ = \Phi^{-1}(w^+)$.
Since $w^+ \in [\overline{x}, \Phi (x)]$, by \eqref{eq:48.0} we
have that $\widetilde{w}^+ \in [ \Phi^{-1}( \overline{x} ),
\Phi^{-1} (\Phi(x))]= [\widetilde{x},x]$. Therefore, $x \in
[\widetilde{w}^+, w^+]$, so that either $|w^+ -y| \geq |x-y|$ or
$|\widetilde{w}^+ -y| \geq |x-y|$. Noticing that $|w^+ -y| = |
\Phi( \widetilde{w}^+) -\Phi(y)| \geq C |\widetilde{w}^+ -y|$, in
both cases, by \eqref{eq:14.5}, we have
\begin{equation}
  \label{eq:50.3}
    |\nabla u^+(w^+,y)| \leq C |x-y|^{-2},
\end{equation}
where $C>0$ only depends on $M_0$, $\alpha$, $\alpha_0$,
$\gamma_0$, $\overline{\lambda}$, $\overline{\mu}$.

Now, let $\widetilde{w}^- = \Phi^{-1}(w^-)$. Recalling
\eqref{eq:48.0} and considering the third components, {}from
$w^-_3 < [\Phi(x)]_3$ it follows that $w^-_3 = [\Phi^{-1}(w^-)]_3
< [\Phi^{-1}(\Phi(x))]_3 =x_3$. On the other hand, $x_3 < w^-_3$,
so that $x \in [w^-, \widetilde{w}^-]$. Therefore, either $|w^- -
y| \geq |x-y|$ or $|\widetilde{w}^- - y| \geq |x-y|$. Noticing
that $|w^- -y| = | \Phi(\widetilde{w}^-)-\Phi(y)| \geq
C|\widetilde{w}^- -y|$, in both cases, by \eqref{eq:14.5}, we have
\begin{equation}
  \label{eq:51.1}
    |\nabla u^+(w^-,y)| \leq C |x-y|^{-2},
\end{equation}
where $C>0$ only depends on $M_0$, $\alpha$, $\alpha_0$,
$\gamma_0$, $\overline{\lambda}$, $\overline{\mu}$.

By \eqref{eq:50.2}--\eqref{eq:51.1} and
\eqref{eq:39.7}, we have
\begin{multline}
  \label{eq:51.2}
    |u^+(\Phi(x),y)-u^+(x,y)| \leq
    C |x-y|^{-2} ( |\Phi(x) - \overline{x}| + |\overline{x} -x| )
    =
    \\
    = C |x-y|^{-2} |\Phi(x) -x| \leq \frac{C}{\rho_0} \left (
    \frac{|x-y|}{\rho_0}\right )^{-1+\alpha},
\end{multline}
where $C>0$ only depends on $M_0$, $\alpha$, $\alpha_0$,
$\gamma_0$, $\overline{\lambda}$, $\overline{\mu}$.

By \eqref{eq:48.7}--\eqref{eq:49.1}, \eqref{eq:51.2} and
\eqref{eq:39.7}, estimate \eqref{eq:38.1} follows.

\medskip

In order to obtain \eqref{eq:38.2}, let $x,y\in Q_{\frac{\rho}{12}, \frac{\rho M_0}{12}}$, $x\in D\cap\{x_3\geq 0\}$, $y=(0,0,-h)$, $h>0$, and let
\begin{equation}
  \label{eq:53.1}
    \Delta(x)=\frac{|x-y|}{8\sqrt{1+M_0^2}}.
\end{equation}
Since $|x-y|\leq \frac{\rho}{6}\sqrt{1+M_0^2}$, for every $x,y\in Q_{\frac{\rho}{12}, \frac{\rho M_0}{12}}$,
it follows that

\begin{equation}
  \label{eq:53.1,5}
    \Delta(x)\leq \frac{\rho}{48},
\end{equation}
so that
\begin{equation}
  \label{eq:53.2}
    Q_{2\Delta(x), 2\Delta(x)M_0}(x)\subset Q_{\frac{\rho}{8}, \frac{\rho M_0}{8}}, \forall x\in
    Q_{\frac{\rho}{12}, \frac{\rho M_0}{12}}.
\end{equation}

Let us define $Q_D^+(x)=Q_{\Delta(x),\Delta(x)M_0}^+(x)\cap D$.
Let $\beta= \frac{\alpha}{2(1+\alpha)}$.

By a standard interpolation inequality, we have
\begin{multline}
  \label{eq:53.3}
    \|\nabla R(\cdot,y)\|_{L^\infty(Q_D^+(x))}\leq\\
    \leq C
     \left(\| R(\cdot,y)\|_{L^\infty(Q_D^+(x))}^{\frac{\beta}{1+\beta}}
      | \nabla R(\cdot,y)|_{\beta,Q_D^+(x)}^{\frac{1}{1+\beta}}
      +\frac{1}{\Delta(x)}\| R(\cdot,y)\|_{L^\infty(Q_D^+(x))}\right),
\end{multline}
where $C>0$ only depends
on $M_0$ and $\alpha$.

By Lemma \ref{lem:14BIS.1}, and noticing that $\overline{Q_D^+(x)}\subset\overline{\R^{3+}}\cap \overline{D}$,
\begin{multline}
  \label{eq:54.1}
      | \nabla R(\cdot,y)|_{\beta,Q_D^+(x)}\leq
      |  \nabla u^D(\cdot,y)|_{\beta,Q_D^+(x)}+
      |  \nabla u^+(\cdot,y)|_{\beta,Q_D^+(x)}
      \leq\\
      \leq
      \frac{C}{\Delta(x)^{1+\beta}}\left(\|  u^D(\cdot,y)\|_{L^{\infty}(Q_{2\Delta(x), 2\Delta(x)M_0}(x))}+
      \|  u^+(\cdot,y)\|_{L^{\infty}(Q_{2\Delta(x), 2\Delta(x)M_0}(x))}
      \right)
      ,
\end{multline}
where $C>0$ only depends on $M_0$, $\alpha$, $\alpha_0$,
$\gamma_0$, $\overline{\lambda}$, $\overline{\mu}$. Noticing that
for any $w\in Q_{2\Delta(x), 2\Delta(x)M_0}(x)$,
\begin{equation}
  \label{eq:54.2}
    |w-y|\geq|x-y|-|w-x|\geq |x-y|-2\Delta(x)\sqrt{1+M_0^2}=\frac{3}{4}|x-y|,
\end{equation}
and by applying \eqref{eq:14.4} to both $u^D$ and $u^+$, we have
\begin{equation}
  \label{eq:54.3}
     \|  u^D(\cdot,y)\|_{L^{\infty}(Q_{2\Delta(x), 2\Delta(x)M_0}(x))}\leq C|x-y|^{-1},
\end{equation}
\begin{equation}
  \label{eq:54.4}
     \|  u^+(\cdot,y)\|_{L^{\infty}(Q_{2\Delta(x), 2\Delta(x)M_0}(x))}\leq C|x-y|^{-1},
\end{equation}
where $C>0$ only depends on $M_0$, $\alpha$, $\alpha_0$,
$\gamma_0$, $\overline{\lambda}$, $\overline{\mu}$.

{}From \eqref{eq:54.1},\eqref{eq:54.3}, \eqref{eq:54.4}, we have
\begin{equation}
  \label{eq:54.5}
      | \nabla R(\cdot,y)|_{\beta,Q_D^+(x)}\leq
      C|x-y|^{-2-\beta},
\end{equation}
with $C>0$ only depending on $M_0$, $\alpha$, $\alpha_0$,
$\gamma_0$, $\overline{\lambda}$, $\overline{\mu}$.

By \eqref{eq:53.3}, \eqref{eq:54.5}, \eqref{eq:38.1} and \eqref{eq:54.2}, we obtain
\begin{multline}
  \label{eq:54.6}
   |\nabla_x u^D(x,y)-\nabla_x u^+(x,y)| \leq
   \|\nabla R(\cdot,y)\|_{L^\infty(Q_D^+(x))}\leq\\
   \leq\frac{C}{\rho_0^2}
   \left [
   \left (
   \frac{|x-y|}{\rho_0}
   \right )^{-2 +  \frac{\alpha\beta}{\beta+1}  }+
   \left (
   \frac{|x-y|}{\rho_0}
   \right )^{-2 +  \alpha}
   \right],
\end{multline}
where $C>0$ only depends on $M_0$, $\alpha$, $\alpha_0$,
$\gamma_0$, $\overline{\lambda}$, $\overline{\mu}$. Taking into account that
$0<\frac{\beta}{\beta+1}<1$, we obtain
\eqref{eq:38.2}.

\end{proof}

\section{Proof of Theorem \ref{theo:55.1}}
\label{proof_lowerbound}

Let $\rho = \min \left
\{ dist(O,D_2), \ \frac{\rho_0}{12 \sqrt{1+M_0^2}} \cdot \min
\{1,M_0\}  \right  \}$, and  $h\leq \overline{h}\rho$, with $\overline{h}\in(0,\frac{1}{2})$ to be chosen later, where $O \equiv P$ denotes the origin of
the cartesian coordinate system, with $e_3 = -\nu$, $\nu$ being
the outer unit normal to $D_1$ at $O$. This choice ensures that
estimates \eqref{eq:38.1}, \eqref{eq:38.2} hold for $D=D_1$ in
$B_\rho(O) \cap D_1$ and $B_\rho^+(O) \cap D_1$, respectively. For
simplicity, in the following we shall denote $B_\rho=B_\rho(O)$,
$B_\rho^+=B_\rho^+(O)$.

By \eqref{eq:20.3} we have
\begin{equation}
  \label{eq:56.1}
     |f(y_h,w_h;l,m)|\geq |S_{D_1}(y_h,w_h;l,m)|-|S_{D_2}(y_h,w_h;l,m)|.
\end{equation}
In order to estimate $|S_{D_1}|$ {}from below, we write
\begin{multline}
  \label{eq:57.1}
     S_{D_1}(y_h,w_h;l,m)=
     \int_{D_1 \cap B_\rho}
     (\C^I -\C) \nabla_x ( \Gamma^+(x,y_h)l)
     \cdot
     \nabla_x(\Gamma(x,w_h)m) +
     \\
     +
     \int_{D_1 \cap B_\rho}
     (\C^I -\C) \nabla_x ((\Gamma^{D_1}(x,y_h)-\Gamma^+(x,y_h))l)
     \cdot
     \nabla_x ((\Gamma^{D_2}(x,w_h)-\Gamma(x,w_h))m)+
     \\
     +
     \int_{D_1 \cap B_\rho}
     (\C^I -\C) \nabla_x ((\Gamma^{D_1}(x,y_h)-\Gamma^+(x,y_h))l)
     \cdot
     \nabla_x (\Gamma(x,w_h)m) +
     \\
     +
     \int_{D_1 \cap B_\rho}
     (\C^I -\C) \nabla_x (\Gamma^+(x,y_h)l)
     \cdot
     \nabla_x ((\Gamma^{D_2}(x,w_h)-\Gamma(x,w_h))m)+
     \\
     +
     \int_{D_1 \setminus B_\rho}
     (\C^I -\C) \nabla_x ( \Gamma^{D_1}(x,y_h)l)
     \cdot
     \nabla_x(\Gamma^{D_2}(x,w_h)m),
\end{multline}
where $\Gamma(x,y)$ is the \textit{Kelvin fundamental solution} in
$\R^3$ of the Lam\'{e} operator with constant coefficients
$\lambda$, $\mu$. It is well known (see
\cite{Gu72}) that
\begin{equation}
  \label{eq:57.2}
     \Gamma(x,y)= \frac{1}{16\pi\mu(1-\nu)}
     \cdot
     \frac{1}{|x-y|}
     \left (
     \frac{(x-y)\otimes (x-y)   }{|x-y|^2} + (3-4\nu)Id
     \right ).
\end{equation}

Since the leading term of $S_{D_1}$, as $h \rightarrow 0$, is the
first integral in the right hand side of \eqref{eq:57.1}, it is
convenient to represent the domain of integration as follows
\begin{equation}
  \label{eq:57.3}
     D_1 \cap B_\rho = B_\rho^+ \cup (D_1 \cap B_\rho^-) \setminus
     (B_\rho^+ \setminus D_1)
\end{equation}
and rewrite \eqref{eq:57.1} as
\begin{equation}
  \label{eq:58.1}
     S_{D_1}(y_h,w_h;l,m)= I_1 +R_1 +R_2 +R_3,
\end{equation}
where
\begin{equation}
  \label{eq:58.2}
    I_1 = \int_{B_\rho^+}
     (\C^I -\C) \nabla_x ( \Gamma^+(x,y_h)l)
     \cdot
     \nabla_x(\Gamma(x,w_h)m),
\end{equation}
\begin{multline}
  \label{eq:58.3}
    R_1 = \int_{D_1 \cap B_\rho^-}
     (\C^I -\C) \nabla_x ( \Gamma^+(x,y_h)l)
     \cdot
     \nabla_x(\Gamma(x,w_h)m) -
     \\
     - \int_{B_\rho^+ \setminus D_1}
     (\C^I -\C) \nabla_x ( \Gamma^+(x,y_h)l)
     \cdot
     \nabla_x(\Gamma(x,w_h)m),
\end{multline}
\begin{equation}
  \label{eq:58.4}
    R_2= \int_{D_1 \setminus B_\rho}
     (\C^I -\C) \nabla_x ( \Gamma^{D_1}(x,y_h)l)
     \cdot
     \nabla_x(\Gamma^{D_2}(x,w_h)m),
\end{equation}
\begin{multline}
  \label{eq:58.5}
    R_3 =   \int_{D_1 \cap B_\rho}
     (\C^I -\C) \nabla_x ((\Gamma^{D_1}(x,y_h)-\Gamma^+(x,y_h))l)
     \cdot
     \nabla_x (\Gamma(x,w_h)m) +
     \\
     + \int_{D_1 \cap B_\rho}
     (\C^I -\C) \nabla_x (\Gamma^{D_1}(x,y_h)l)
     \cdot
     \nabla_x ((\Gamma^{D_2}(x,w_h)-\Gamma(x,w_h))m).
\end{multline}
\begin{lem}
    \label{lem:59.1}
\begin{multline}
  \label{eq:59.1}
    I_1 = \frac{1}{h}\int_{R^3_+}
     (\C^I -\C) \nabla_x ( \Gamma^+(x,-e_3)l)
     \cdot
     \nabla_x(\Gamma(x,-\lambda_w e_3)m) -
     \\
     -  \frac{1}{h}\int_{R^3_+ \setminus B_{ \frac{\rho}{h}}^+}
     (\C^I -\C) \nabla_x ( \Gamma^+(x,-e_3)l)
     \cdot
     \nabla_x(\Gamma(x,-\lambda_w e_3)m).
\end{multline}
\end{lem}
\begin{proof}
Let us start with proving the following identities, which hold for
any $h>0$, $\xi$, $y_0\in \R^3$, $\xi \neq y_0$:
\begin{equation}
  \label{eq:59.2}
    \Gamma(\xi, y_0) = h \Gamma(h \xi, h y_0), \quad
    \Gamma^+(\xi,y_0) = h \Gamma^+(h\xi,hy_0).
\end{equation}
Let us prove the first identity. By the definition of
$\Gamma(\cdot,y)$, we have
\begin{equation}
  \label{eq:59.3}
    \int_{\R^3} \C \nabla_x (\Gamma(x,y)l) \cdot
    \nabla_x\varphi(x) = l \cdot \varphi (y), \quad \hbox{for
    every } \varphi \in C_0^\infty(\R^3).
\end{equation}
By choosing $y=hy_0$ and performing the change of variables $\xi=
\frac{x}{h}$, we have
\begin{equation}
  \label{eq:59.4}
    \int_{\R^3} \C \nabla_\xi (h\Gamma(h\xi,h y_0)l) \cdot
    \nabla_\xi \psi(\xi) = l \cdot \psi (y_0), \quad \hbox{for
    every } \psi \in C_0^\infty(\R^3),
\end{equation}
where $\psi(\xi)=\phi(h\xi)$, that is the first identity in
\eqref{eq:59.2} holds. The second identity in \eqref{eq:59.2} can
be derived similarly, taking into account that
$\chi^+(hx)=\chi^+(x)$ for every $h >0$.

By applying the change of variables $\xi = \frac{x}{h}$ to $I_1$,
recalling \eqref{eq:55.1}, \eqref{eq:55.2} and using
\eqref{eq:59.2}, we obtain the identity \eqref{eq:59.1}.
\end{proof}
Let us set
\begin{equation}
  \label{eq:60.0}
    I_1 = I_1' - I_1'',
\end{equation}
where
\begin{equation}
  \label{eq:60.0bis}
    I_1' = \frac{1}{h}\int_{R^3_+}
     (\C^I -\C) \nabla_x ( \Gamma^+(x,-e_3)l)
     \cdot
     \nabla_x(\Gamma(x,-\lambda_w e_3)m),
\end{equation}
\begin{equation}
  \label{eq:60.0ter}
    I_1'' = \frac{1}{h}\int_{R^3_+ \setminus B_{ \frac{\rho}{h}}^+}
     (\C^I -\C) \nabla_x ( \Gamma^+(x,-e_3)l)
     \cdot
     \nabla_x(\Gamma(x,-\lambda_w e_3)m).
\end{equation}

Let us first estimate {}from above $I_1''$. By recalling
\eqref{eq:14.5} and observing that $|\xi + e_3| \geq
|\xi|$, $|\xi +\lambda_w e_3| \geq |\xi|$, we have
\begin{equation}
  \label{eq:60.1}
    |I_1''| \leq \frac{C}{\rho},
\end{equation}
where $C>0$ only depends on $\alpha_0$, $\gamma_0$,
$\overline{\lambda}$, $\overline{\mu}$.

In order to evaluate $I_1'$, let us premise the following
identity.
\begin{lem}
    \label{lem:60.1}
\begin{multline}
  \label{eq:60.2}
     \int_{\R^3_+}
     (\C^I -\C) \nabla_x ( \Gamma^+(x,y_0)l)
     \cdot
     \nabla_x(\Gamma(x,w_0)m) =
     (\Gamma(y_0,w_0) -\Gamma^+(y_0,w_0))m \cdot l,
     \\
     \hbox{for every } y_0, w_0 \in \R^{3}, \ y_0 \neq w_0.
\end{multline}
\end{lem}
\begin{proof}
This is a special case of \cite[Proposition 3.2]{BFV13}, the proof is analogous to the one of Lemma \ref{lem:21.0}.
\end{proof}

\begin{prop}
    \label{prop:62.1}
Let $y_0=(0,0,-1)$,
$w_0=(0,0,-\lambda_w)$. For every $i=1,2,3$, there exists $\lambda_w \in \left\{
\frac{2}{3}, \frac{3}{4}, \frac{4}{5} \right\}$ such that
\begin{equation}
  \label{eq:62.2}
     \left | (\Gamma^+(y_0,w_0)- \Gamma(y_0,w_0))e_i \cdot e_i \right |
     \geq
     \mathcal{C},
\end{equation}
where $\mathcal{C} > 0$ only depends on $\alpha_0$, $\gamma_0$,
$\overline{\lambda}$, $\overline{\mu}$, $\eta_0$.
\end{prop}

The proof is postponed to the next Section
\ref{Rongved}.

\begin{rem}
  \label{rem:finitely}
We note that for the present purposes it would suffice to prove that \eqref{eq:62.2} holds
true for at least one $i=1,2,3$. We believe that such a result, slightly stronger than necessary, may be instructive because we shall show by examples (Section \ref{Rongved}) that there may be values of $\lambda_w$ for which $(\Gamma^+(y_0,w_0)-\Gamma(y_0,w_0))e_i\cdot e_i$ may indeed equal zero.

On the other hand, taking advantage of the explicit character of $\Gamma^+$, we shall prove that
$(\Gamma^+(y_0,w_0)-\Gamma(y_0,w_0))e_i\cdot e_i$ may vanish only for finitely many values of $\lambda_w$.

Finally, let us note that we restrict the choice of $\lambda_w$
to three specific values just for the sake of definiteness and also with the purpose of having
\eqref{eq:62.2} in a constructive form.
\end{rem}

\medskip

Choosing $l=m=e_i$, $i=1,2,3$, and
taking into account \eqref{eq:59.1}, \eqref{eq:60.1},
\eqref{eq:60.2} and \eqref{eq:62.2}, we have
\begin{equation}
  \label{eq:64.2}
     I_1 \geq \frac{\mathcal{C}}{h} + I_1'', \quad \hbox{with }
     |I_1''| \leq \frac{C}{\rho},
\end{equation}
where $\mathcal{C} > 0$ and $C>0$ only depend on $\alpha_0$,
$\gamma_0$, $\overline{\lambda}$, $\overline{\mu}$, $\eta_0$.

Let us estimate $|R_1|$ {}from above. By recalling
\eqref{eq:58.3}, by using \eqref{eq:14.5}, \eqref{eq:57.2}, and by
the change of variables $y= \frac{x}{h}$, we have
\begin{multline}
  \label{eq:64.3}
     |R_1| \leq
     C
     \int_{\R^2}
     \left (
     \int_{ - \frac{M_0}{\rho_0^\alpha}|x'|^{1+\alpha}}^{ \frac{M_0}{\rho_0^\alpha}|x'|^{1+\alpha}}
     |x-y_h|^{-2} |x-w_h|^{-2} dx_3
     \right )
     dx_1 dx_2 = \\
     = C
      \int_{\R^2}
     \left (
     \int_{ - \frac{M_0}{\rho_0^\alpha}|x'|^{1+\alpha}}^{ \frac{M_0}{\rho_0^\alpha}|x'|^{1+\alpha}}
     \frac{1}{ ( |x'|^2 + (x_3+h)^2 )
      ( |x'|^2 + (x_3+\lambda_w h)^2 ) } dx_3
     \right )
     dx_1 dx_2 =
     \\
     = \frac{C}{h}
      \int_{\R^2}
     \left (
     \int_{ - \frac{M_0}{\rho_0^\alpha}h^\alpha |y'|^{1+\alpha}}^{ \frac{M_0}{\rho_0^\alpha}h^\alpha |y'|^{1+\alpha}}
     \frac{1}{ ( |y'|^2 + (y_3+1)^2 )
      ( |y'|^2 + (y_3+\lambda_w )^2 ) } dy_3
     \right )
     dy_1 dy_2,
\end{multline}
where $C>0$ only depends on $M_0$, $\alpha$, $\alpha_0$,
$\gamma_0$, $\overline{\lambda}$, $\overline{\mu}$.

It is convenient to split $\R^2$ as the union of the sets $A=\{ y'
\in \R^2 \ | \ |y'| \geq \left (   \frac{2M_0}{\rho_0^\alpha}
\right )^{- \frac{1}{1+\alpha}} h^{- \frac{\alpha}{1+\alpha}} \}$,
$B=\{ y' \in \R^2 \ | \ |y'| < \left (
\frac{2M_0}{\rho_0^\alpha} \right )^{- \frac{1}{1+\alpha}} h^{-
\frac{\alpha}{1+\alpha} }\}$ and to estimate the integral in the right hand side of
\eqref{eq:64.3} separately in $A$ and $B$. We obviously have
\begin{multline}
  \label{eq:64.3a}
     \frac{C}{h}
      \int_{A}
     \left (
     \int_{ - \frac{M_0}{\rho_0^\alpha}h^\alpha |y'|^{1+\alpha}}^{ \frac{M_0}{\rho_0^\alpha}h^\alpha |y'|^{1+\alpha}}
     \frac{1}{ ( |y'|^2 + (y_3+1)^2 )
      ( |y'|^2 + (y_3+\lambda_w )^2 ) } dy_3
     \right )
     dy_1 dy_2 \leq \\
     \leq
     \frac{C}{\rho_0^\alpha} h^{\alpha - 1}
     \int_{A}
     \frac{|y'|^{1+\alpha}   }{  |y'|^4   }
     dy_1 dy_2
     =
      \frac{C}{\rho_0} \left (  \frac{h}{\rho_0} \right ) ^{\alpha - 1}
     \int_{  \left (  \frac{2M_0}{\rho_0^\alpha} \right )  ^{- \frac{1}{1+\alpha}} h^{-
    \frac{\alpha}{1+\alpha}}  }^{+\infty} r^{\alpha -2 } dr
    =
    \frac{C}{\rho_0} \left (  \frac{h}{\rho_0} \right )  ^{ \frac{\alpha - 1 }{\alpha + 1}  },
\end{multline}
where $C>0$ only depends on $M_0$, $\alpha$, $\alpha_0$,
$\gamma_0$, $\overline{\lambda}$, $\overline{\mu}$.

In $B$ we have that $|y_3| < \frac{1}{2}$, so that $|y_3 + 1 | >
\frac{1}{2}$ and $|y_3 + \lambda_w| \geq \lambda_w -
\frac{1}{2}\geq \frac{2}{3} - \frac{1}{2} = \frac{1}{6}$.
Therefore
\begin{multline}
  \label{eq:64.3b}
     \frac{C}{h}
      \int_{B}
     \left (
     \int_{ - \frac{M_0}{\rho_0^\alpha}h^\alpha |y'|^{1+\alpha}}^{ \frac{M_0}{\rho_0^\alpha}h^\alpha |y'|^{1+\alpha}}
     \frac{1}{ ( |y'|^2 + (y_3+1)^2 )
      ( |y'|^2 + (y_3+\lambda_w )^2 ) } dy_3
     \right )
     dy_1 dy_2 \leq \\
     \leq
     \frac{C}{\rho_0^\alpha} h^{\alpha - 1}
     \int_{B}
     \frac{|y'|^{1+\alpha}   }{ ( |y'|^2  + \frac{1}{36} )^2  }
     \leq
     \frac{C}{\rho_0} \left (  \frac{h}{\rho_0} \right ) ^{\alpha - 1}
     \int_{\R^2}
     \frac{|y'|^{1+\alpha}   }{ ( |y'|^2  + \frac{1}{36} )^2  } =
     \frac{C}{\rho_0} \left (  \frac{h}{\rho_0} \right ) ^{\alpha - 1},
\end{multline}
where $C>0$ only depends on $M_0$, $\alpha$, $\alpha_0$,
$\gamma_0$, $\overline{\lambda}$, $\overline{\mu}$. By
\eqref{eq:64.3}--\eqref{eq:64.3b}, and recalling that $
\frac{h}{\rho_0} \leq 1$,  we have
\begin{equation}
  \label{eq:64.3c}
    |R_1| \leq
    \frac{C}{\rho_0} \left (  \frac{h}{\rho_0} \right ) ^{\alpha - 1},
\end{equation}
where $C>0$ only depends on $M_0$, $\alpha$, $\alpha_0$,
$\gamma_0$, $\overline{\lambda}$, $\overline{\mu}$.

\medskip

Let us estimate $R_2$ {}from above. By \eqref{eq:14.5}, we have
\begin{equation}
  \label{eq:68.1}
    |R_2| \leq C \int_{\R^3 \setminus B_\rho}
    |x-y_h|^{-2}|x-w_h|^{-2},
\end{equation}
where $C>0$ only depends on $M_0$, $\alpha$, $\alpha_0$,
$\gamma_0$, $\overline{\lambda}$, $\overline{\mu}$. Since $ h \leq
\frac{\rho}{2}$ and $\lambda_w \in (0,1)$, we have $B_h(y_h)
\subset B_\rho$, $B_h(w_h) \subset B_\rho$, so that $|x-y_h| \geq
h$, $|x-w_h| \geq h$ for every $x \in \R^3 \setminus B_\rho$.
Moreover, $|x-y_h| \geq |x|-h$, $|x-w_h| \geq |x|-\lambda_wh \geq
|x|- h$. Passing to spherical coordinates, denoting by $r$ the radial coordinate and taking into account
that, since $h \leq \frac{\rho}{2} \leq \frac{r}{2}$, we have $r-h
\geq \frac{r}{2}$, it follows that
\begin{equation}
  \label{eq:68.2}
    |R_2| \leq C \int_\rho^\infty
    \frac{r^2}{(r-h)^4}dr \leq  C \int_\rho^\infty \frac{dr}{r^2}
    = \frac{C}{\rho},
\end{equation}
where $C>0$ only depends on $M_0$, $\alpha$, $\alpha_0$,
$\gamma_0$, $\overline{\lambda}$, $\overline{\mu}$.

Let us estimate $R_3$ {}from above. To this aim, let us set

\begin{equation}
  \label{eq:68.3}
    R_3 = R_3' + R_3'',
\end{equation}
where
\begin{equation}
  \label{eq:68.3bis}
    R_3' = \int_{D_1 \cap B_\rho}
     (\C^I -\C) \nabla_x ((\Gamma^{D_1}(x,y_h)-\Gamma^+(x,y_h))l)
     \cdot
     \nabla_x (\Gamma(x,w_h)m),
\end{equation}
\begin{equation}
  \label{eq:68.3ter}
    R_3'' = \int_{D_1 \cap B_\rho}
     (\C^I -\C) \nabla_x (\Gamma^{D_1}(x,y_h)l)
     \cdot
     \nabla_x ((\Gamma^{D_2}(x,w_h)-\Gamma(x,w_h))m).
\end{equation}

Noticing that $D_1 \cap B_{\rho} \subset \R^3_+ \cup \left ( \R^3
\cap \left \{ - \frac{M_0}{\rho_0^\alpha}|x'|^{1+\alpha} \leq x_3 \leq 0 \right \}
\right )$, by \eqref{eq:38.2}, \eqref{eq:57.2}, \eqref{eq:14.5} we
have
\begin{multline}
  \label{eq:69.1}
    R_3' \leq \frac{C}{\rho_0^\gamma} \int_{\R^3_+}
    |x-y_h|^{-2+\gamma}|x-w_h|^{-2} + C \int_{\R^3
\cap \left \{ - \frac{M_0}{\rho_0^\alpha}|x'|^{1+\alpha} \leq x_3 \leq 0 \right \}
} |x-y_h|^{-2}|x-w_h|^{-2},
\end{multline}
where $\gamma = \frac{\alpha^2}{3\alpha+2}<\frac{1}{2}$ and $C>0$
only depends on $M_0$, $\alpha$, $\alpha_0$, $\gamma_0$,
$\overline{\lambda}$, $\overline{\mu}$.

By passing to cylindrical coordinates and applying H\"{o}lder
inequality twice, we have
\begin{multline}
  \label{eq:69.2}
    \frac{1}{\rho_0^\gamma} \int_{\R^3_+}
    |x-y_h|^{-2+\gamma}|x-w_h|^{-2}\leq\\
    \leq \frac{C}{\rho_0^\gamma} \int_0^\infty \left (
    \int_0^\infty r(r^2 + (x_3 +h)^2)^{-1 + \frac{\gamma}{2}} (r^2 + (x_3 +\lambda_w h)^2)^{-1} dr
    \right ) dx_3 \leq
    \\
    \leq
    \frac{C}{\rho_0^\gamma} \int_0^\infty
    (x_3 + h)^{\gamma -1} (x_3 + \lambda_w h)^{ -1}
    dx_3 = \frac{C}{\rho_0} \left ( \frac{h}{\rho_0} \right
    )^{\gamma -1},
\end{multline}
where $C>0$ only depends on $M_0$, $\alpha$, $\alpha_0$,
$\gamma_0$, $\overline{\lambda}$, $\overline{\mu}$.

On the other hand, by \eqref{eq:64.3}, \eqref{eq:64.3c} we have
\begin{multline}
  \label{eq:69.3}
     \int_{\R^3
\cap \left \{ - \frac{M_0}{\rho_0^\alpha}|x'|^{1+\alpha} \leq x_3 \leq 0 \right \}
} |x-y_h|^{-2}|x-w_h|^{-2}\leq\\
     \leq
     \int_{\R^2}
     \left (
     \int_{ - \frac{M_0}{\rho_0^\alpha}|x'|^{1+\alpha}}^{ \frac{M_0}{\rho_0^\alpha}|x'|^{1+\alpha}}
     |x-y_h|^{-2} |x-w_h|^{-2} dx_3
     \right )
     dx_1 dx_2\leq
    \frac{C}{\rho_0} \left (  \frac{h}{\rho_0} \right ) ^{\alpha - 1},
\end{multline}
where $C>0$ only depends on $M_0$, $\alpha$, $\alpha_0$,
$\gamma_0$, $\overline{\lambda}$, $\overline{\mu}$.

In order to estimate $R_3''$, let us notice that, by
\eqref{eq:55.4}, $D_2 \cap B_\rho = \emptyset$, so that, by $\lambda_w h\leq h\leq\frac{\rho}{2}$, $D_2 \cap B_{ \frac{\rho}{2}} (w_h) =
\emptyset$. Therefore, the function
\begin{equation}
  \label{eq:70.1}
    v(x,w_h) = ( \Gamma^{D_2}(x,w_h) - \Gamma(x,w_h))m
\end{equation}
satisfies the Lam\'{e} system with constant coefficients
$\lambda$, $\mu$
\begin{equation}
  \label{eq:70.2}
    \divrg_x (\C \nabla_x v(x,w_h))=0, \quad \hbox{in } B_{ \frac{\rho}{2}}
    (w_h).
\end{equation}
By standard regularity estimates, we have
\begin{equation}
  \label{eq:70.3}
    \sup_{ B_{ \frac{\rho}{4}}(w_h)} | \nabla_x v(x, w_h)| \leq \frac{C}{\rho^{ \frac{5}{2}  }}
    \left ( \int_{B_{ \frac{\rho}{2}}(w_h)} |v(x,w_h)|^2 \right )^{
    \frac{1}{2}},
\end{equation}
where $C>0$ only depends on $\alpha_0$, $\gamma_0$,
$\overline{\lambda}$, $\overline{\mu}$.

At this stage, we apply the Maximum Modulus Theorem by Fichera
\cite{Fi61}, which asserts that
\begin{equation}
  \label{eq:70.4}
    \sup_{ B_{ \frac{\rho}{2}}(w_h)} | v(x, w_h)| \leq C \sup_{ \partial B_{ \frac{\rho}{2}}(w_h)} | v(x, w_h)|,
\end{equation}
where $C>0$ only depends on $\alpha_0$, $\gamma_0$,
$\overline{\lambda}$, $\overline{\mu}$.

By \eqref{eq:70.3}, \eqref{eq:70.4}, \eqref{eq:14.5},
\eqref{eq:57.2}, we have
\begin{equation}
  \label{eq:70.5}
    \sup_{ B_{ \frac{\rho}{4}}(w_h)} | \nabla_x v(x, w_h)| \leq \frac{C}{\rho^{2}},
\end{equation}
where $C>0$ only depends on $M_0$, $\alpha_0$, $\gamma_0$,
$\overline{\lambda}$, $\overline{\mu}$.

It is convenient to split the integral $R_3''$ as follows
\begin{equation}
  \label{eq:71.5}
  R_3'' = R_{3a}'' + R_{3b}'' ,
\end{equation}
where
\begin{multline}
  \label{eq:71.5bis}
   R_{3a}''= \int_{ B_{ \frac{\rho}{4}}(w_h) \cap D_1}
     (\C^I -\C) \nabla_x (\Gamma^{D_1}(x,y_h)l)
     \cdot
     \nabla_x ((\Gamma^{D_2}(x,w_h)-\Gamma(x,w_h))m),
\end{multline}
\begin{multline}
  \label{eq:71.5ter}
   R_{3b}''=
     \int_{ (D_1 \cap B_\rho) \setminus B_{ \frac{\rho}{4}}(w_h)}
     (\C^I -\C) \nabla_x (\Gamma^{D_1}(x,y_h)l)
     \cdot
     \nabla_x ((\Gamma^{D_2}(x,w_h)-\Gamma(x,w_h))m).
\end{multline}
By \eqref{eq:14.5}, \eqref{eq:70.5} and noticing that, by $h\leq \frac{\rho}{2}$, $B_{\frac{\rho}{4}}(w_h)\subset
B_{\frac{3\rho}{4}}(y_h)$,
we have
\begin{equation}
  \label{eq:71.6}
    R_{3a}'' \leq \frac{C}{\rho^2} \int_{ B_{ \frac{\rho}{4}}(w_h)}
    |x-y_h|^{-2} \leq
    \frac{C}{\rho^2} \int_{ B_{ \frac{3\rho}{4}}(y_h)}
    |x-y_h|^{-2} \leq
    \frac{C}{\rho},
\end{equation}
where $C>0$ only depends on $M_0$, $\alpha_0$, $\gamma_0$,
$\overline{\lambda}$, $\overline{\mu}$.

By \eqref{eq:14.5}, by H\"{o}lder inequality and requiring that $\overline{h}\leq \frac{1}{8}$, so that $h\leq \frac{\rho}{8}$ and $B_{\frac{\rho}{8}}(y_h)\subset
B_{\frac{\rho}{4}}(w_h)$,
we have
\begin{multline}
  \label{eq:71.7}
    R_{3b}'' \leq C \int_{ \R^3 \setminus B_{ \frac{\rho}{4}}(w_h)}
    |x-y_h|^{-2}|x-w_h|^{-2} \leq \\
    \leq
    C \left(\int_{ \R^3 \setminus B_{ \frac{\rho}{4}}(w_h)}
    |x-w_h|^{-4} \right)^{\frac{1}{2}}
    \left(\int_{ \R^3 \setminus B_{ \frac{\rho}{8}}(y_h)}
    |x-y_h|^{-4} \right)^{\frac{1}{2}}
    \leq \frac{C}{\rho},
\end{multline}
where $C>0$ only depends on $M_0$, $\alpha_0$, $\gamma_0$,
$\overline{\lambda}$, $\overline{\mu}$.

By \eqref{eq:58.5}, \eqref{eq:69.2},
\eqref{eq:71.5}, \eqref{eq:71.6}, \eqref{eq:71.7}, we have
\begin{equation}
  \label{eq:71BIS.1}
    |R_{3}| \leq \frac{C}{\rho_0}
    \left (
    \left (
    \frac{h}{\rho_0}
    \right )^{\gamma-1}
    +\left ( \frac{h}{\rho_0} \right )^{\alpha-1} \right )
    +
    \frac{C}{\rho},
\end{equation}
where $C>0$ only depends on $M_0$, $\alpha_0$, $\gamma_0$,
$\overline{\lambda}$, $\overline{\mu}$.

Let us estimate {}from above $|S_{D_2} (y_h, w_h)|$. By
\eqref{eq:14.5} and recalling that $D_2 \cap B_{\rho}= \emptyset$,
we have
\begin{equation}
  \label{eq:72.1}
    |S_{D_2}(y_h, w_h)| \leq C \int_{\R^3 \setminus B_{\rho}}
    |x-y_h|^{-2} |x-w_h|^{-2},
\end{equation}
where $C>0$ only depends on $M_0$, $\alpha_0$, $\gamma_0$,
$\overline{\lambda}$, $\overline{\mu}$. Noticing that
\begin{equation}
  \label{eq:72.2}
    |x-y_h| \geq |x| -h, \quad |x-w_h| \geq |x|-h,
\end{equation}
and passing to spherical coordinates, we have
\begin{equation}
  \label{eq:72.3}
    |S_{D_2}(y_h, w_h)| \leq C \int_{\rho}^\infty
    \frac{r^2}{(r-h)^4}dr,
\end{equation}
where $C>0$ only depends on $M_0$, $\alpha_0$, $\gamma_0$,
$\overline{\lambda}$, $\overline{\mu}$. By $h \leq \frac{\rho}{2}
\leq \frac{r}{2}$, we have $r-h \geq \frac{r}{2}$, so that
\begin{equation}
  \label{eq:72.4}
    |S_{D_2}(y_h, w_h)| \leq C \int_{\rho}^\infty
    \frac{dr}{r^2} = \frac{C}{\rho},
\end{equation}
where $C>0$ only depends on $M_0$, $\alpha_0$, $\gamma_0$,
$\overline{\lambda}$, $\overline{\mu}$.

Finally, by \eqref{eq:56.1}, \eqref{eq:58.1}, \eqref{eq:64.2}, \eqref{eq:64.3c},
\eqref{eq:68.2}, \eqref{eq:71BIS.1},
\eqref{eq:72.4}, we have
\begin{equation}
  \label{eq:73.1}
   |f(y_h,w_h;e_i, e_i)| \geq
   \frac{\mathcal{C}}{h}
   \left (
   1 - C_1 \left ( \frac{h}{\rho_0} \right
   )^\alpha - C_2 \left ( \frac{h}{\rho_0} \right )^\gamma- C_3 \frac{h}{\rho}
   \right ), \quad i=1,2,
\end{equation}
where $\gamma=  \frac{\alpha^2}{3\alpha + 2}$ and the constants
$C_i>0$, $i=1,2,3$, only depend on $M_0$, $\alpha$, $\alpha_0$,
$\gamma_0$, $\overline{\lambda}$, $\overline{\mu}$. Therefore,
there exists $ \overline{h}>0$, only depending on  $M_0$,
$\alpha$, $\alpha_0$, $\gamma_0$, $\overline{\lambda}$,
$\overline{\mu}$, $\eta_0$, such that, for any $h$, $0<h< \overline{h}\rho$,
estimate \eqref{eq:55.2} follows.

\section{Rongved's fundamental solution and proof of Proposition \ref{prop:62.1}}
\label{Rongved}

In this section, in order to prove Proposition \ref{prop:62.1}, we
investigate whether there exist directions $l$, $m$ for which
there exists $w_0=(0,0,-c)$, $0<c<1$, such that
\begin{equation}
  \label{eq:81.1}
     (\Gamma^+(y_0, w_0)-\Gamma(y_0,w_0))m \cdot l \neq 0
\end{equation}
for any couple of Lam\'{e} materials with moduli $(\mu,\nu)$,
$(\mu^I, \nu^I)$, where $y_0=(0,0,-1)$. To this aim, we introduce
the closed-form expression of $\Gamma^+$ derived by Rongved
\cite{Ron55}.

Let us choose a coordinate system $(0, e_x, e_y, e_z)$. Consider
the two half-spaces $R^{3+}=\{(x,y,z) \ | \ z>0\}$,
$R^{3-}=\{(x,y,z) \ | \ z<0\}$ made by homogeneous Lam\'{e}
materials with moduli $(\mu, \nu)$, $(\mu^I, \nu^I)$,
respectively. The two half-spaces are glued together on the
interface $z=0$, that is the traction and the displacement both
are continuous across the interface $z=0$. The problem of
determining the displacement field
\begin{equation}
  \label{eq:76.1}
   u^+(\cdot, P) = \Gamma^+(\cdot, P)l
\end{equation}
in $\R^3$ caused by a force $l\in \R^3$, $|l|=1$, acting at the
point $P \equiv (0,0,c)$, $0<c<1$, is described by the following
boundary value problem
\begin{equation}
  \label{eq:76.2}
  \left\{ \begin{array}{ll}
  \mu \Delta u^+ + \frac{\mu}{1-2\nu} \nabla ( \divrg u^+)= -l
  \delta(P), &
  \mathrm{in}\ z>0 ,\\
  &  \\
   \mu^I \Delta u^+ + \frac{\mu^I}{1-2\nu^I} \nabla ( \divrg u^+)= 0,
    &
  \mathrm{in}\ z<0 ,\\
  &  \\
   \left (
    \mu(\nabla u^+ + (\nabla u^+)^T) + \frac{2\mu\nu}{1-2\nu} (\divrg u^+) Id
    \right )e_z |_{(x,y,0^+)} =\\
    =\left (
    \mu^I(\nabla u^+ + (\nabla u^+)^T) +  \frac{2\mu^I\nu^I}{1-2\nu^I} (\divrg u^+) Id
    \right )e_z |_{(x,y,0^-)},
    & \mathrm{on }\ z=0,\\
    &  \\
    u^+ |_{(x,y,0^+)} = u^+ |_{(x,y,0^-)}, & \mathrm{on }\ z=0,\\
    &  \\
     \lim_{|(x,y,z)| \rightarrow \infty} u^+(x,y,z)=0.\\
  \end{array}\right.
\end{equation}

\bigskip

\textit{\textbf{Case 1.} Force $l=e_z$ normal to the interface:}

\medskip

\begin{equation}
  \label{eq:77.1}
   \Gamma^+(\cdot, P) e_z \cdot e_x = \Gamma_{xz}^+(\cdot, P) = - \frac{1}{4(1-\nu)}
   \left (
   \frac{\partial \beta}{\partial x} + z \frac{\partial B_z}{\partial x}
   \right ),
\end{equation}
\begin{equation}
  \label{eq:77.2}
   \Gamma^+(\cdot, P) e_z \cdot e_y = \Gamma_{yz}^+(\cdot, P) = - \frac{1}{4(1-\nu)}
   \left (
   \frac{\partial \beta}{\partial y} + z \frac{\partial B_z}{\partial y}
   \right ),
\end{equation}
\begin{equation}
  \label{eq:77.3}
   \Gamma^+(\cdot, P) e_z \cdot e_z = \Gamma_{zz}^+(\cdot, P) = \frac{3-4\nu}{4(1-\nu)}B_z -
   \frac{1}{4(1-\nu)}\left (
   \frac{\partial \beta}{\partial z} + z \frac{\partial B_z}{\partial
   z},
   \right ),
\end{equation}
where

\medskip

\textit{for $z>0$}:
\begin{equation}
  \label{eq:77.4}
   B_z = \frac{1}{4\pi \mu }
   \left \{
    \frac{1}{R_1} + \frac{\mu-\mu^I}{\mu+\mu^I(3-4\nu)}
    \left (
    \frac{3-4\nu}{R_2} + \frac{2c(z+c)}{R_2^3}
    \right )
    \right \},
\end{equation}
\begin{multline}
  \label{eq:77.5}
   \beta = - \frac{1}{4\pi \mu }
   \left \{
   \frac{c}{R_1} +
   \frac{\mu-\mu^I}{\mu+\mu^I(3-4\nu)}
   \left [
   \frac{c(3-4\nu)}{R_2} -
    \right.
    \right.
\\
   \left.
   \left.
   -\frac{4\mu(1-\nu)}{\mu-\mu^I}
   \left (
   \frac{\mu(1-2\nu)(3-4\nu^I)-\mu^I(1-2\nu^I)(3-4\nu)}{\mu^I+\mu(3-4\nu^I)}
   \right )
   \log (R_2 +z+c)
   \right ]
   \right \},
\end{multline}
with
\begin{equation}
  \label{eq:78.1}
   R_1=(x^2+y^2+(z-c)^2)^{\frac{1}{2}}, \quad
   R_2=(x^2+y^2+(z+c)^2)^{\frac{1}{2}};
\end{equation}
\medskip

\textit{for $z<0$}:
\begin{equation}
  \label{eq:78.2}
   B_z = \frac{1-\nu^I}{\pi R_1 (\mu^I +\mu(3-4\nu^I))},
\end{equation}
\begin{equation}
  \label{eq:78.3}
   \beta = \frac{1-\nu^I}{1-\nu} \left ( - \frac{b_1}{R_1} + b_2
   \log(R_1 -z+c) \right ),
\end{equation}
where
\begin{equation}
  \label{eq:78.4}
   b_1 = \frac{c(1-\nu)}{\pi(\mu + \mu^I(3-4\nu))},
\end{equation}
\begin{equation}
  \label{eq:78.5}
   b_2 = \frac{1-\nu}{\pi(\mu + \mu^I(3-4\nu))} \cdot
   \frac{\mu(1-2\nu)(3-4\nu^I)-\mu^I(1-2\nu^I)(3-4\nu)}{\mu^I+\mu(3-4\nu^I)}.
\end{equation}
\medskip

\textit{\textbf{Case 2.} Force $l=e_x$ parallel to the interface:}

\medskip

\begin{equation}
  \label{eq:78.6}
   \Gamma^+(\cdot, P) e_x \cdot e_x=\Gamma_{xx}^+(\cdot, P) = \frac{3-4\nu}{4(1-2\nu)}B_x -
   \frac{1}{4(1-\nu)}
   \left (
   \frac{\partial \beta}{\partial x}
   + x \frac{\partial B_x}{\partial x} +
   z \frac{\partial B_z}{\partial x}
   \right ),
\end{equation}
\begin{equation}
  \label{eq:78.7}
   \Gamma^+(\cdot, P) e_x \cdot e_y=\Gamma_{yx}^+(\cdot, P) = - \frac{1}{4(1-\nu)}
   \left (
   \frac{\partial \beta}{\partial y} + x \frac{\partial B_x}{\partial
   y}+ z \frac{\partial B_z}{\partial y}
   \right ),
\end{equation}
\begin{equation}
  \label{eq:78.8}
   \Gamma^+(\cdot, P) e_x \cdot e_z= \Gamma_{zx}^+(\cdot, P) = \frac{3-4\nu}{4(1-\nu)}B_z -
   \frac{1}{4(1-\nu)}\left (
   \frac{\partial \beta}{\partial z} + x \frac{\partial B_x}{\partial
   z} + z \frac{\partial B_z}{\partial z}
   \right ),
\end{equation}
where
\medskip

\textit{for $z>0$}:

\begin{equation}
  \label{eq:79.1}
   B_x = \frac{1}{4\pi \mu}
   \left (
   \frac{1}{R_1} + \frac{1 - \frac{\mu^I}{\mu}  }{ 1+    \frac{\mu^I}{\mu}
   } \frac{1}{R_2}
   \right ),
\end{equation}
\begin{equation}
  \label{eq:79.2}
   B_z = \frac{\mu-\mu^I}{2\pi(\mu+\mu^I(3-4\nu))}
   \left (
   - \frac{cx}{\mu R_2^3} +
   \frac{(1-2\nu)x}{ (\mu+\mu^I)R_2(R_2+z+c)  }
   \right ),
\end{equation}
\begin{equation}
  \label{eq:79.3}
   \beta =
   \frac{1}{ 2\pi(\mu+\mu^I)(\mu+\mu^I(3-4\nu))  }
   \left (
    \frac{(1-2\nu)(\mu-\mu^I)cx}{R_2(R_2+z+c)}
    +
    A^* \frac{x}{R_2+z+c}
    \right ),
\end{equation}
with
\begin{multline}
  \label{eq:79.4}
     A^* =   \left \{
    (\mu-\mu^I)(1-2\nu)
    \left [
    \mu^I (3-4\nu)(1-2\nu^I) - \mu(3-4\nu^I)(1-2\nu)
    \right ] - \right.
    \\
    \left. - 2\mu^I (\nu - \nu^I) (\mu + \mu^I(3-4\nu))
    \right \} \cdot \frac{1}{ \mu^I + \mu(3-4\nu^I)   } \ ;
\end{multline}
\medskip

\textit{for $z<0$}:
\begin{equation}
  \label{eq:79.5}
     B_x = \frac{1}{2\pi\mu \left ( 1 +  \frac{\mu^I}{\mu}\right )
     } \cdot \frac{1}{R_1},
\end{equation}
\begin{equation}
  \label{eq:79.6}
     B_z =
     \frac{ (1-2\nu^I)(\mu-\mu^I)     }{ 2\pi(\mu+\mu^I)(\mu^I+\mu(3-4\nu^I))   }
     \cdot \frac{x}{R_1(R_1 -z +c)},
\end{equation}
\begin{multline}
  \label{eq:80.1}
     \beta = \frac{1-\nu^I}{ 2\pi(1-\nu)(\mu+\mu^I)(\mu+\mu^I(3-4\nu))
     }\cdot
     \\
     \cdot
     \left \{
     \left [
     (1-2\nu)(\mu-\mu^I) + \frac{(\nu-\nu^I)(\mu+\mu^I(3-4\nu))}{1-\nu^I}
     \right ]
     \frac{cx}{R_1(R_1-z+c)}
     \right. + \\
     \left.
     +\left [
     A^* +
     \frac{(\nu-\nu^I)(\mu+\mu^I(3-4\nu))}{1-\nu^I}
     \right ] \cdot \frac{x}{R_1-z+c}
     \right \}.
\end{multline}
In order to adapt these results to our notation, we find
convenient to introduce the following change of the coordinate
system:
\begin{center}
\( {\displaystyle \left\{
\begin{array}{lr}
  e_1 = e_y,
    \vspace{0.25em}\\
  e_2 = e_x,
    \vspace{0.25em}\\
  e_3 = -e_z,
\end{array}
\right. } \) \vskip -2.4em
\begin{eqnarray}
& & \label{eq:80.2}
\end{eqnarray}
\end{center}
associated to the rotation
\begin{equation}
    \label{eq:80.3}
    R=
    \begin{bmatrix}
    0       &       1 &     0 \vspace{2mm} \\
    1       &       0 &     0 \vspace{2mm} \\
    0       &       0 &    -1
    \end{bmatrix}.
\end{equation}
Then, we have
\begin{equation}
    \label{eq:80.4}
    \begin{bmatrix}
    \Gamma_{11}^+       &       \Gamma_{12}^+ &     \Gamma_{13}^+ \vspace{2mm} \\
    \Gamma_{21}^+       &       \Gamma_{22}^+ &     \Gamma_{23}^+ \vspace{2mm} \\
    \Gamma_{31}^+       &       \Gamma_{32}^+ &     \Gamma_{33}^+
    \end{bmatrix}
    =
    \begin{bmatrix}
    \Gamma_{yy}^+       &       \Gamma_{yx}^+ &     -\Gamma_{yz}^+ \vspace{2mm} \\
    \Gamma_{xy}^+       &       \Gamma_{xx}^+ &     -\Gamma_{xz}^+ \vspace{2mm} \\
    -\Gamma_{zy}^+       &       -\Gamma_{zx}^+ &    \Gamma_{zz}^+
    \end{bmatrix}.
\end{equation}
and a relationship analogous to \eqref{eq:80.4} holds for the
Kelvin fundamental matrix $\Gamma$.

Let us analyze the main cases.

\medskip

i) $m=e_3$, $l=e_3$.

In this case, by \eqref{eq:77.3}--\eqref{eq:77.5}, \eqref{eq:57.2}, and denoting  $Q=(0,0,1)=-y_0$,
we have
\begin{equation}
  \label{eq:82.3}
    \Gamma_{33}^+ (y_0, w_0) - \Gamma_{33}(y_0, w_0) =
\Gamma_{zz}^+(Q,P)-\Gamma_{zz}(Q,P)) = \frac{1}{4\pi\mu(1-\nu)}
\cdot \frac{{\mathcal{P}}(t)}{t^3(\mu+\mu^I(3-4\nu))},
\end{equation}
where
\begin{equation}
  \label{eq:82.4}
    \mathcal{P}(t) = (1-\nu) \left [
    (\mu-\mu^I)(3-4\nu)-\gamma \mu
    \right ]t^2 + (\mu^I-\mu)(t-1),
\end{equation}
with $t=1+c$, $1<t<2$, and
\begin{equation}
  \label{eq:82.5}
    \gamma =
    \frac{\mu(1-2\nu)(3-4\nu^I)-\mu^I(1-2\nu^I)(3-4\nu)}{\mu^I+\mu(3-4\nu^I)}.
\end{equation}
We note that $\mathcal{P}$ is a second degree polynomial

\begin{equation}
  \label{eq:82.6}
    \mathcal{P}(t)= \alpha t^2+\beta t + \gamma,
\end{equation}
whose coefficients can be estimated as follows
\begin{equation}
  \label{eq:82.7}
    \alpha^2+\beta^2 + \gamma^2= (A\delta\mu + B\delta\nu)^2 + C^2(\delta\mu)^2,
\end{equation}
where we have denoted
\begin{equation}
  \label{eq:82.8}
     \delta\mu =\mu-\mu^I, \qquad  \delta\nu =\nu-\nu^I,
\end{equation}
and the quantities $A$, $B$, $C$ satisfy
\begin{equation}
  \label{eq:82.9}
     C^2,B^2\geq \frac{1}{K},\qquad A^2\leq K,
\end{equation}
where $K>0$ only depends on the a-priori data.

Furthermore we observe that, being the space of real second degree polynomials a 3-dimensional
linear space, for any three distinct values $t_1$, $t_2$, $t_3\in \R$ we have
\begin{equation}
  \label{eq:82.10}
     (\mathcal{P}(t_1))^2+(\mathcal{P}(t_2))^2+(\mathcal{P}(t_3))^2\geq
     C(\alpha^2+\beta^2 + \gamma^2),
\end{equation}
where $C>0$ is a computable quantity only depending on $t_1$, $t_2$, $t_3$.

Thus, in view of \eqref{eq:3.1bis}, we obtain
\begin{equation}
  \label{eq:82.11}
     (\mathcal{P}(t_1))^2+(\mathcal{P}(t_2))^2+(\mathcal{P}(t_3))^2\geq
     Q^2>0,
\end{equation}
where $Q>0$ only depends on $t_1$, $t_2$, $t_3$ and on the a-priori data.

In conclusion, picking any three distinct values $c_1$, $c_2$,
$c_3\in [\frac{2}{3},\frac{4}{5}]$ (for the sake of concreteness
we may choose $c_1=\frac{2}{3}$, $c_2=\frac{3}{4}$,
$c_3=\frac{4}{5}$), we obtain that there exists $i\in\{1,2,3\}$
such that $w_0=(0,0,-c_i)$ satisfies
\begin{equation}
  \label{eq:82.12}
    |\Gamma_{33}^+ (y_0, w_0) - \Gamma_{33}(y_0, w_0)|\geq C>0,
\end{equation}
where $C$ only depends on the a-priori data.

\medskip

ii) $m=l=e_2$.

By \eqref{eq:80.4} and the analogous for $\Gamma$, by
\eqref{eq:78.6}, \eqref{eq:79.1}--\eqref{eq:79.4} we have
\begin{equation}
  \label{eq:83.5}
    \Gamma_{22}^+ (y_0, w_0) - \Gamma_{22}(y_0, w_0) =
\Gamma_{xx}^+(Q,P)-\Gamma_{xx}(Q,P)) = \frac{1}{4(1-\nu)} \cdot
\frac{{\mathcal{P}}(t)}{t^3},
\end{equation}
where
\begin{multline}
  \label{eq:84.1}
     \mathcal{P}(t) = \frac{3-4\nu}{4\pi \mu} \cdot
     \frac{\mu-\mu^D}{\mu+\mu^D}\cdot t^2 -
     \frac{(1-2\nu)(\mu-\mu^D) }{ 4\pi(\mu+\mu^D)(\mu+\mu^D(3-4\nu))
     }\cdot t^2 -\\
     - \frac{A^*}{4\pi(\mu+\mu^D)(\mu+\mu^D(3-4\nu))}\cdot t^2 +
     \frac{\mu-\mu^D}{2\pi \mu(\mu + \mu^D(3-4\nu))  }\cdot (t-1)
\end{multline}
and
\begin{multline}
  \label{eq:84.2}
     A^* =   \left \{
    (\mu-\mu^D)(1-2\nu)
    \left [
    \mu^D (3-4\nu)(1-2\nu^D) - \mu(3-4\nu^D)(1-2\nu)
    \right ] - \right.
    \\
    \left. - 2\mu^D (\nu - \nu^D) (\mu + \mu^D(3-4\nu))
    \right \} \cdot \frac{1}{ \mu^D + \mu(3-4\nu^D)   } \ .
\end{multline}
An inspection of the polynomial $\mathcal{P}$ analogous to the one performed above leads again
to the conclusion that, picking $c_1=\frac{2}{3}$, $c_2=\frac{3}{4}$, $c_3=\frac{4}{5}$,
there exists $i\in\{1,2,3\}$ such that $w_0=(0,0,-c_i)$ satisfies
\begin{equation}
  \label{eq:84.2bis}
    |\Gamma_{22}^+ (y_0, w_0) - \Gamma_{22}(y_0, w_0)|\geq C>0,
\end{equation}
where $C$ only depends on the a-priori data.

A similar result holds when $m=l=e_1$, namely
\begin{equation}
  \label{eq:84.2ter}
    |\Gamma_{11}^+ (y_0, w_0) - \Gamma_{11}(y_0, w_0)|\geq C>0,
\end{equation}
where $C$ only depends on the a-priori data.

\begin{proof}[Proof of Proposition \ref{prop:62.1}]
This is an immediate consequence of \eqref{eq:82.12}, \eqref{eq:84.2bis} and \eqref{eq:84.2ter}.
\end{proof}
\medskip

\begin{rem}
Observe that
\begin{equation}
  \label{eq:84.5}
    \Gamma_{ij}^+(y_0,w_0)=\Gamma_{ij}(y_0,w_0)=0, \quad i \neq j,
    \ i,j=1,2,3.
\end{equation}
Hence only diagonal terms of $\Gamma^+-\Gamma$ appear to be relevant.
\end{rem}

\subsection{Examples}

Let us conclude the present section with a remark emphasizing an
interesting difference with respect to the electrostatic case. For
the analogous inverse problem in electrostatics, which involves
the Laplace operator, {}from the explicit expression of $\Gamma^+$
given, for instance, by (3.11) in \cite{ADiC05}, it easy to see
that
\begin{equation}
  \label{eq:85.1}
    (\Gamma^+ - \Gamma)(y_0,y_0) \neq 0, \quad \hbox{for } \
    y_0=(0,0,-1),
\end{equation}
for any choice of different constant values of the conductivity
within the inclusion $D$ and in $\Omega \setminus D$. On the
contrary, an analogous result does not hold in the isotropic
elastic case. Since for current materials, the Poisson coefficient
takes positive value, let us restrict our analysis to the cases in
which $0<\nu< \frac{1}{2}$ and $0<\nu^I< \frac{1}{2}$.

In case $m=l=e_3$,
taking the limit in \eqref{eq:82.3} as $w_0$ tends to $y_0$, that
is choosing $t=2$, one finds
\begin{equation}
  \label{eq:85.9}
    (\Gamma_{33}^+  - \Gamma_{33})(y_0, y_0) = \frac{Q(\mu,\nu,\mu^I,\nu^I)}{32\pi\mu(1-\nu)(\mu+\mu^I(3-4\nu))(\mu^I+\mu(3-4\nu^I))},
\end{equation}
\begin{multline}
  \label{eq:85.10}
    Q(\mu,\nu,\mu^I,\nu^I)= 32  \mu^{2} \nu^{2} \nu^I - 32  \mu \mu^I \nu^{2} \nu^I - 24  \mu^{2} \nu^{2} - 64
 \mu^{2} \nu \nu^I + \\
 +16  \mu \mu^I \nu^{2} + 56  \mu \mu^I \nu \nu^I + 16  (\mu^I)^{2} \nu^{2} + 48
 \mu^{2} \nu + 28  \mu^{2} \nu^I - 28  \mu \mu^I \nu -\\
 - 20  \mu \mu^I \nu^I - 28  (\mu^I)^{2} \nu -
21  \mu^{2} + 10  \mu \mu^I  + 11  (\mu^I)^{2}.
\end{multline}
The polynomial $Q$ is homogeneous of degree $2$ in $\mu$ and
$\mu^I$, and of degree $1$ in $\nu^I$. Setting $s =
\frac{\mu}{\mu^I}$ and dividing by $(\mu^I)^2$, we obtain
\begin{equation}
  \label{eq:85.11}
    \frac{Q(\mu,\nu,\mu^I,\nu^I)}{(\mu^I)^2}
    = Q_2(\nu,\nu^I,s),
\end{equation}
where
\begin{multline}
  \label{eq:85.12}
    Q_2(\nu,\nu^I,s)= 32  \nu^{2} \nu^I s^{2} - 32  \nu^{2} \nu^I s - 24  \nu^{2} s^{2} - 64  \nu \nu^I
s^{2} + 16  \nu^{2} s + 56  \nu \nu^I s + \\
+ 48  \nu s^{2} + 28  \nu^I s^{2} + 16
 \nu^{2} - 28  \nu s - 20  \nu^I s - 21  s^{2} - 28  \nu + 10  s + 11.
\end{multline}
Solving $Q_2=0$ with respect to $\nu^I$, we have
\begin{equation}
  \label{eq:85.13}
    \nu^I=\frac{3  {\left(8  \nu^{2} - 16  \nu + 7\right)} s^{2} - 2  {\left(8
 \nu^{2} - 14  \nu + 5\right)} s - 16  \nu^{2} + 28  \nu - 11}{4
{\left[{\left(8  \nu^{2} - 16  \nu + 7\right)} s^{2} - {\left(8  \nu^{2}
- 14  \nu + 5\right)} s\right]}}
\end{equation}
{}From \eqref{eq:85.13} it is possible to determine triples of values $(\nu,\nu^I, s)$
satisfying \eqref{eq:85.13} and such that $0< \nu< \frac{1}{2}$, $0< \nu^I< \frac{1}{2}$,
$0<s\neq 1$, for instance $(\frac{1}{8}, \frac{17}{36}, \frac{6}{5})$, $(\frac{1}{4}, \frac{661}{1628}, \frac{11}{10})$, $(\frac{3}{8}, \frac{17}{36}, \frac{11}{10})$. Therefore there exist infinitely many
pairs of materials $\{\mu,\nu\}$, $\{\mu^I,\nu^I\}$ such that
$(\Gamma_{33}^+-\Gamma_{33})(y_0,y_0) = 0$.

\begin{figure}
\begin{center}
\resizebox{6cm}{!}{
\includegraphics{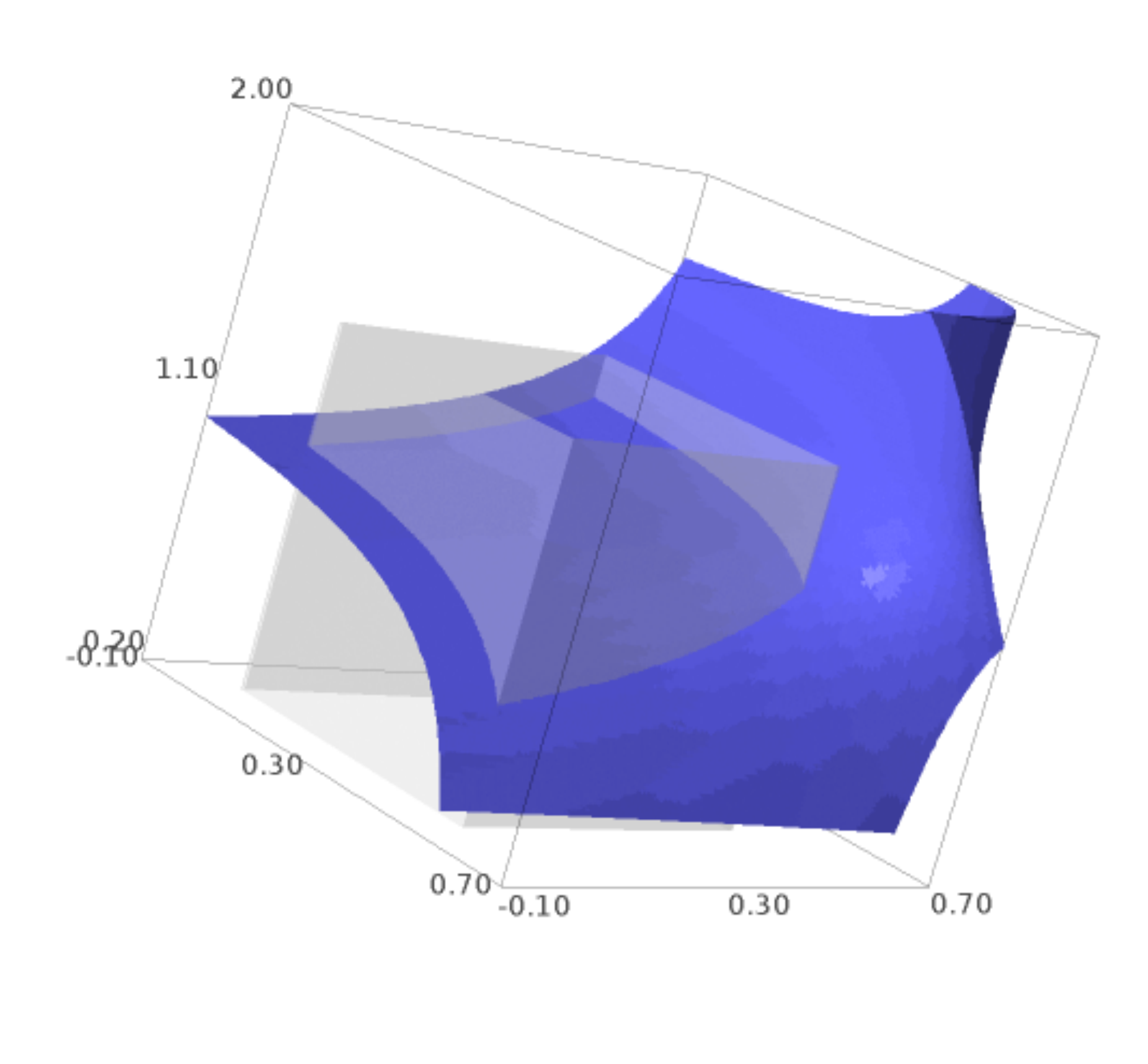}
}
\caption{Case $m=l=e_3$: Intersection of the surface $Q_2=0$ with
the set  $\{(\nu,\nu^I, s)\ |\ 0< \nu< \frac{1}{2}, 0< \nu^I< \frac{1}{2},
0<s<2\}$\label{figura3}}
\end{center}
\end{figure}

Figure \ref{figura3} shows the intersection of the surface $Q_2=0$ with
the set  $\{(\nu,\nu^I, s)\ |\ 0< \nu< \frac{1}{2}, 0< \nu^I< \frac{1}{2},
0<s<2\}$. It is evident {}from this graph that for each couple $(\nu,\nu^I)$ of Poisson
coefficients such that $0< \nu< \frac{1}{2}, 0< \nu^I< \frac{1}{2}$, there exists a positive value of
$s$ such that $Q_2(\nu,\nu^I, s)=0$.

Moreover, substituting $s=1$ in the expression of $Q_2$, one finds
\begin{equation}
\label{eq:85.14}
    Q_2(\nu,\nu^I, 1)= 8(\nu-1)(\nu-\nu^I),
\end{equation}
which has no zero when $\nu\neq \nu^I$. This implies that for each couple $(\nu,\nu^I)$ of Poisson
coefficients such that $0< \nu< \frac{1}{2}, 0< \nu^I< \frac{1}{2}$, $\nu\neq \nu^I$, there exists $\overline{s}$,
$0<\overline{s}\neq 1$ such that $Q_2(\nu,\nu^I, \overline{s})=0$, that is $Q_2(\nu,\nu^I, \overline{s}\mu^I, \mu^I)=0$ for any $\mu^I>0$.
Moreover, {}from \eqref{eq:85.14} it follows that if $\mu=\mu^I$, then for any choice of the Poisson
coefficients, such that $\nu\neq\nu^I$, then $Q_2\neq 0$ and therefore $(\Gamma_{33}^+-\Gamma_{33})(y_0,y_0) \neq 0$.

\begin{figure}
\begin{center}
\resizebox{6cm}{!}{
\includegraphics{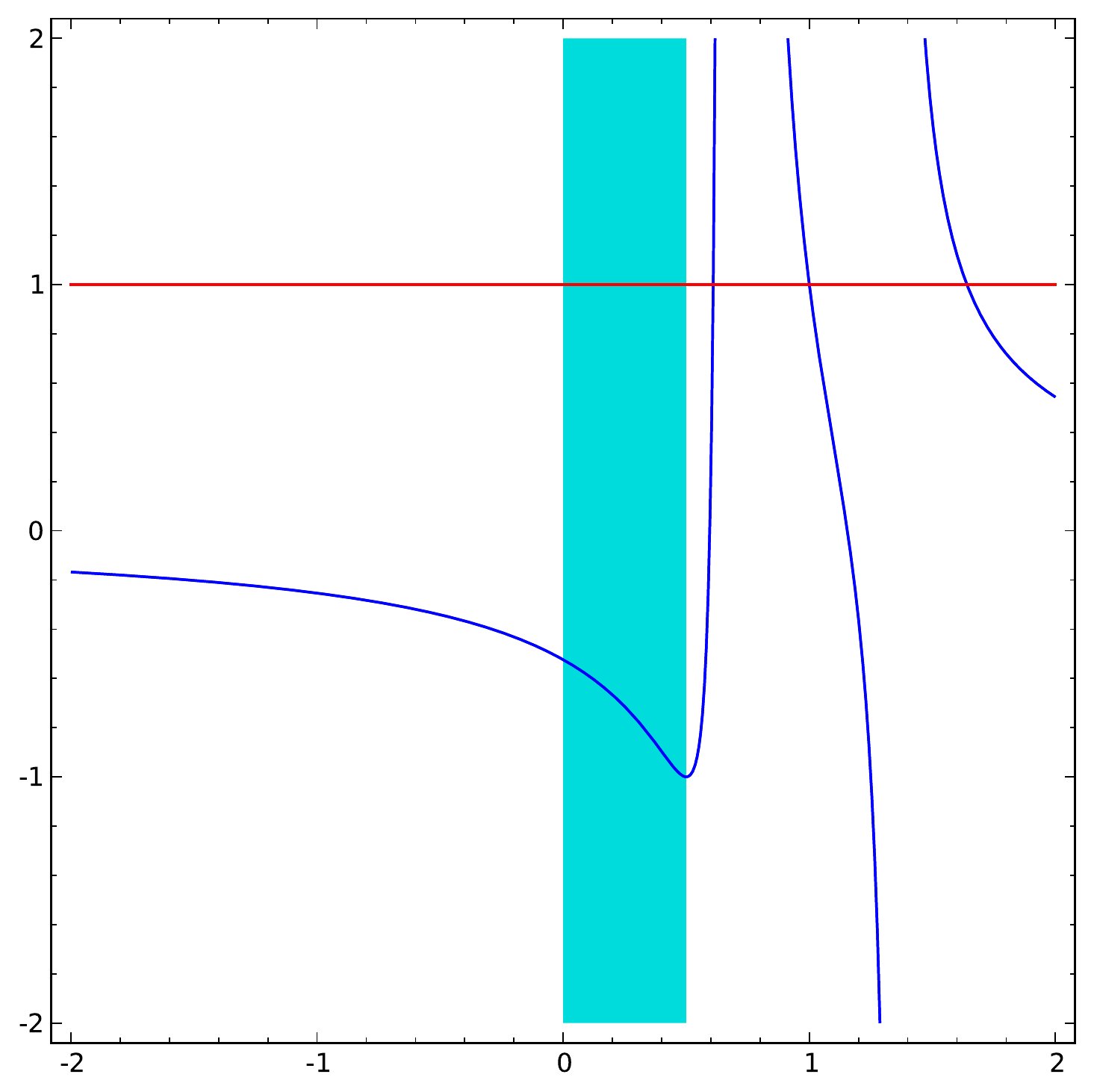}
} \caption{Case $m=l=e_3$: Representation of the curve $Q_2(\nu,\nu, s)=0$ for
$\nu\in [-2,2]$, $s\in[-2,2]$\label{figura4}}
\end{center}
\end{figure}

Next, putting $\nu^I=\nu$ in the expression of $Q_2$, one finds
\begin{multline}
\label{eq:85.15}
    Q_2(\nu,\nu, s)= (s-1)\cdot
    [(32 \nu^{3} -88 \nu^{2} +76\nu - 21)s +(-16 \nu^{2} + 28 \nu -11)],
\end{multline}
that is the intersection of the surface $Q_2=0$ with the plane
$\nu=\nu^I$, when represented in the plane $(\nu^I,s)$, splits in
the line $s=1$ and in an algebraic curve of degree 3. Figure
\ref{figura4}, which contains the graph of this curve and of the
line $s=1$, shows that in our set of interest,
$0<\nu<\frac{1}{2}$, the only solution is $s=1$, that is if the
Poisson coefficients coincide, but $\mu\neq \mu^I$ (that is $s\neq
1$), then $(\Gamma_{33}^+-\Gamma_{33})(y_0,y_0) \neq 0$.

\medskip

In case $m=l=e_2$ and
taking the limit in the complete expression of \eqref{eq:83.5} as $w_0$ tends to $y_0$, that
is choosing $t=2$, one finds
\begin{equation}
  \label{eq:85.2}
    (\Gamma_{22}^+  - \Gamma_{22})(y_0, y_0) = \frac{1}{16(1-\nu)} \cdot
    \frac{Q(\mu,\nu,\mu^I,\nu^I)}{R(\mu,\nu,\mu^I,\nu^I)},
\end{equation}
where $R(\mu,\nu,\mu^I,\nu^I)\neq 0$ for $\mu>0$, $\mu^I>0$, and
\begin{multline}
  \label{eq:85.3}
    Q(\mu,\nu,\mu^I,\nu^I)= 32 \mu^3 \nu^2 \nu^I + 64 \mu^2 \mu^I \nu^2
    \nu^I- 96 \mu (\mu^I)^2 \nu^2 \nu^I - 24 \mu^3 \nu^2 -48 \mu^3
    \nu \nu^I - \\
    - 56 \mu^2 \mu^I \nu^2 -104 \mu^2 \mu^I \nu \nu^I + 64 \mu
    (\mu^I)^2 \nu^2 + 136 \mu (\mu^I)^2 \nu \nu^I + 32 (\mu^I)^3
    \nu^2 + 36 \mu^3 \nu + \\
    + 28 \mu^3 \nu^I + 88 \mu^2\mu^I \nu + 40 \mu^2 \mu^I \nu^I
    - 92 \mu (\mu^I)^2\nu - 52 \mu (\mu^I)^2 \nu^I - \\
    - 48 (\mu^I)^3 \nu - 21 \mu^3 -35 \mu^2 \mu^I + 37 \mu
    (\mu^I)^2 + 19 (\mu^I)^3.
\end{multline}
The polynomial $Q$ is homogeneous of degree $3$ in $\mu$ and
$\mu^I$, and of degree $1$ in $\nu^I$. Setting $s =
\frac{\mu}{\mu^I}$ and dividing by $(\mu^I)^3$, we obtain
\begin{equation}
  \label{eq:85.4}
    \frac{Q(\mu,\nu,\mu^I,\nu^I)}{(\mu^I)^3}
    = Q_2(\nu,\nu^I,s),
\end{equation}
where
\begin{multline}
  \label{eq:85.5}
    Q_2(\nu,\nu^I,s)= 32\nu^{2}\nu^Is^{3} + 64\nu^{2}\nu^I s^{2} - 24 \nu^{2} s^{3} - 48 \nu
\nu^I s^{3} - 96 \nu^{2} \nu^I s - 56\nu^{2} s^{2} -\\
- 104 \nu \nu^I s^{2}
+ 36 \nu s^{3} + 28 \nu^I s^{3} + 64 \nu^{2} s
+ 136 \nu \nu^I s + 88 \nu s^{2} +\\
+40 \nu^I s^{2} + 32 \nu^{2} - 21 s^{3} - 92 \nu s
- 52 \nu^I s - 35s^{2} - 48 \nu + 37 s + 19.
\end{multline}
Solving $Q_2=0$ with respect to $\nu^I$, we have
\begin{equation}
  \label{eq:85.6}
    \nu^I=
    \frac{N^D}{D^D},
\end{equation}
where
\begin{multline}
  \label{eq:85.6BIS}
    N^D=
   3 {\left(8 \nu^{2} - 12 \nu + 7\right)} s^{3} + {\left(56
\nu^{2} - 88 \nu + 35\right)} s^{2} -\\
- {\left(64 \nu^{2} - 92 \nu +
37\right)} s
- 32 \nu^{2} + 48 \nu - 19,
\end{multline}
\begin{equation}
  \label{eq:85.6TER}
    D^D=
    4 \left({\left(8 \nu^{2}
- 12 \nu + 7\right)} s^{3} + 2{\left(8 \nu^{2} - 13 \nu +
5\right)} s^{2} - {\left(24 \nu^{2} - 34 \nu + 13\right)} s\right).
\end{equation}

{}From this expression of $\nu^I$ it is possible to determine triples of values $(\nu,\nu^I, s)$
satisfying \eqref{eq:85.6} and such that $0< \nu< \frac{1}{2}$, $0< \nu^I< \frac{1}{2}$,
$0<s\neq 1$, for instance $(\frac{1}{5},\frac{331}{663},\frac{17}{15})$, $(\frac{1}{4},\frac{1951}{47348},\frac{19}{20})$, $(\frac{7}{20},\frac{317}{1596},\frac{19}{20})$. Therefore there exist infinitely many
pairs of materials $\{\mu,\nu\}$, $\{\mu^I,\nu^I\}$ such that
$(\Gamma_{22}^+-\Gamma_{22})(y_0,y_0) = 0$.

\begin{figure}
\begin{center}
\resizebox{6cm}{!}{
\includegraphics{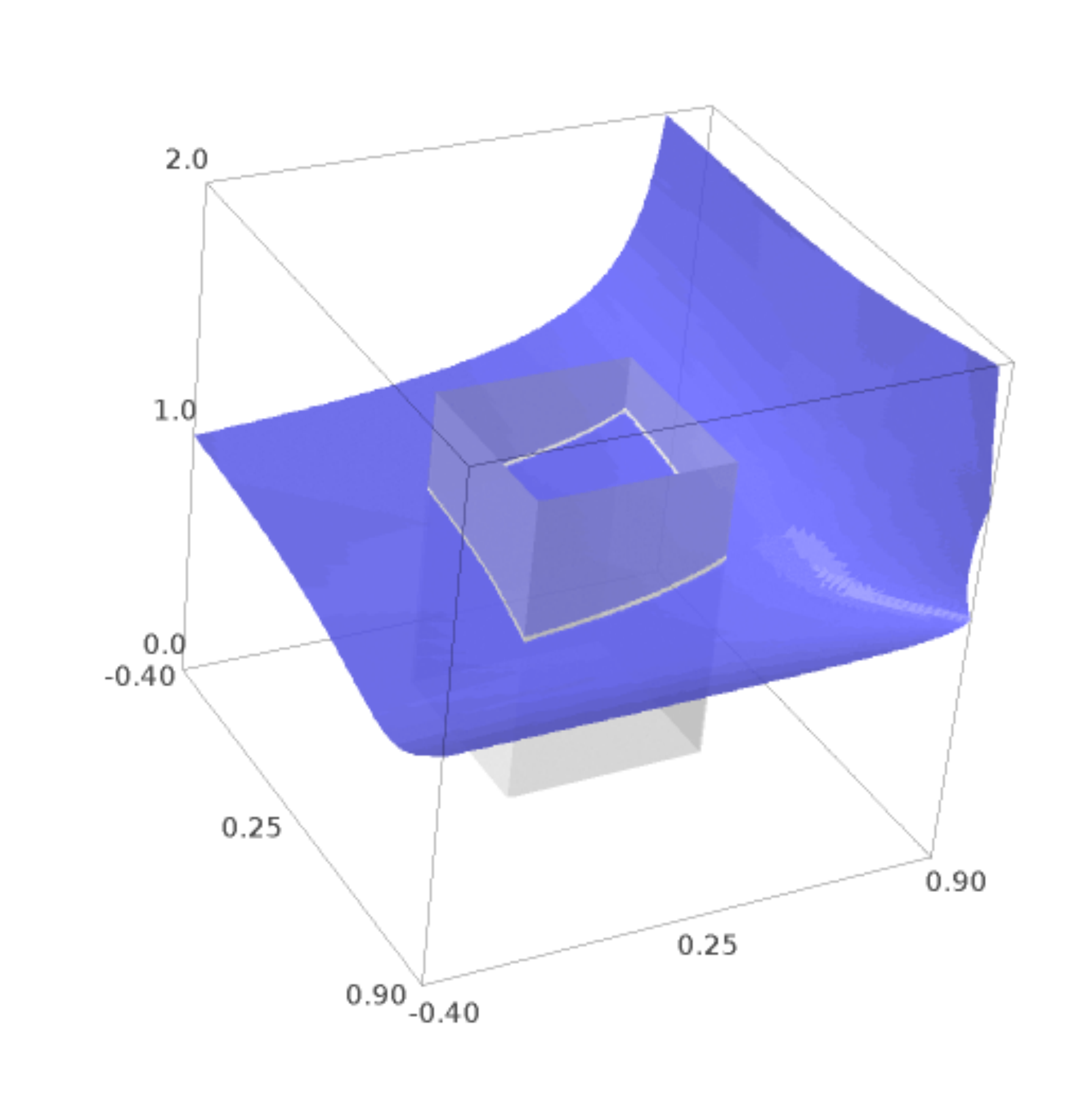}
}
\caption{Case $m=l=e_2$: Intersection of the surface $Q_2=0$ with
the set  $\{(\nu,\nu^I, s)\ |\ 0< \nu< \frac{1}{2}, 0< \nu^I< \frac{1}{2},
0<s<2\}$\label{figura1}}
\end{center}
\end{figure}

Figure \ref{figura1} shows the intersection of the surface $Q_2=0$ with
the set  $\{(\nu,\nu^I, s)\ |\ 0< \nu< \frac{1}{2}, 0< \nu^I< \frac{1}{2},
0<s<2\}$. It is evident {}from this graph that for each couple $(\nu,\nu^I)$ of Poisson
coefficients such that $0< \nu< \frac{1}{2}, 0< \nu^I< \frac{1}{2}$, there exists a positive value of
$s$ such that $Q_2(\nu,\nu^I, s)=0$.

Moreover, substituting $s=1$ in the expression of $Q_2$, one finds
\begin{equation}
\label{eq:85.7}
    Q_2(\nu,\nu^I, 1)= 16(\nu-1)(\nu-\nu^I),
\end{equation}
which has no zero when $\nu\neq \nu^I$. This implies that for each couple $(\nu,\nu^I)$ of Poisson
coefficients such that $0< \nu< \frac{1}{2}, 0< \nu^I< \frac{1}{2}$, $\nu\neq \nu^I$, there exists $\overline{s}$,
$0<\overline{s}\neq 1$ such that $Q_2(\nu,\nu^I, \overline{s})=0$, that is $Q_2(\nu,\nu^I, \overline{s}\mu^I, \mu^I)=0$ for any $\mu^I>0$.
Moreover, {}from \eqref{eq:85.7} it follows that if $\mu=\mu^I$, then for any choice of the Poisson
coefficients, such that $\nu\neq\nu^I$, then $Q_2\neq 0$ and therefore $(\Gamma_{22}^+-\Gamma_{22})(y_0,y_0) \neq 0$.

\begin{figure}
\begin{center}
\resizebox{6cm}{!}{
\includegraphics{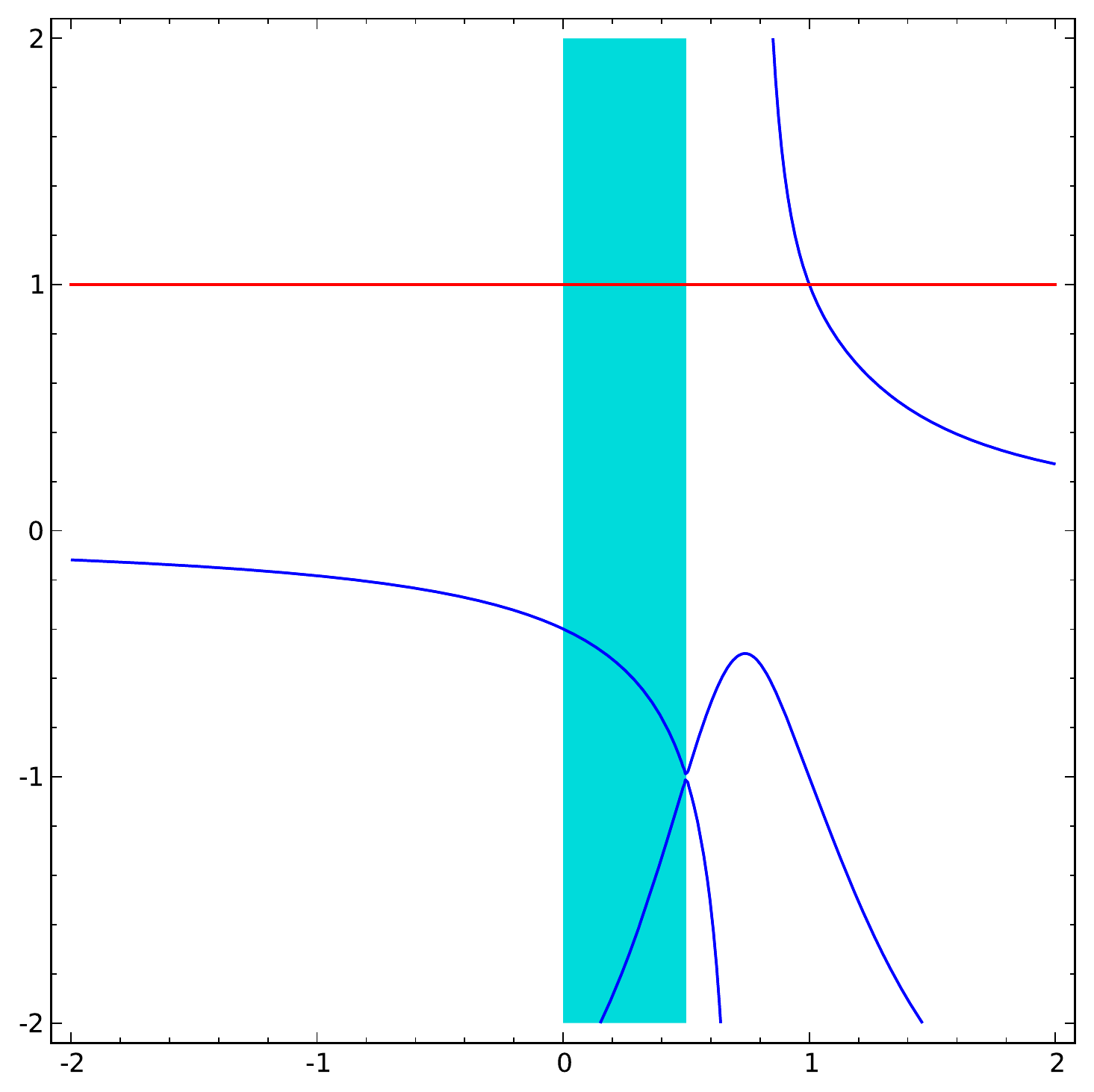}
} \caption{Case $m=l=e_2$: Representation of the curve $Q_2(\nu,\nu, s)=0$ for
$\nu\in [-2,2]$, $s\in[-2,2]$\label{figura2}}
\end{center}
\end{figure}

Next, putting $\nu^I=\nu$ in the expression of $Q_2$, one finds
\begin{multline}
\label{eq:85.8}
    Q_2(\nu,\nu, s)= (s-1)\cdot
    [32 \nu^{3} s^{2} + 96 \nu^{3} s - 72 \nu^{2} s^{2}-\\
    -232 \nu^{2} s + 64 \nu s^{2} - 32 \nu^{2} + 192 \nu s
- 21 s^{2} + 48 \nu - 56 s - 19],
\end{multline}
that is the intersection of the surface $Q_2=0$ with the plane
$\nu=\nu^I$, when represented in the plane $(\nu,s)$, splits in
the line $s=1$ and in an algebraic curve of degree 5. Figure
\ref{figura2}, which contains the graph of this curve and of the
line $s=1$, shows that in our set of interest,
$0<\nu<\frac{1}{2}$, the only solution is $s=1$, that is if the
Poisson coefficients coincide, but $\mu\neq \mu^I$ (that is $s\neq
1$), then $(\Gamma_{22}^+-\Gamma_{22})(y_0,y_0) \neq 0$.

Analogous considerations hold, for symmetry evidence, when studying  $(\Gamma_{11}^+-\Gamma_{11})(y_0,y_0)$.

\section{Metric lemmas, proofs}
\label{technical}

In order to prove Lemma \ref{lem:8.1}, we shall use the following
results.
\begin{lem}
   \label{lem:9.1} [Lemma 5.5 in \cite{ARRV09}]
Let $U$ be a Lipschitz domain in $\R^3$ with constants $\rho_0$,
$M_0$. There exists $h_0$, $0<h_0 <1$, only depending on $M_0$,
such that
\begin{equation}
  \label{eq:9.1}
   U_{h\rho_0} \hbox{ is connected for every } h, \
   0<h\leq h_0.
\end{equation}
\end{lem}
\begin{theo}
   \label{theo:9.2} [Theorem 3.6 in \cite{ABRV00}]
There exist positive constants $d_0$, $r_0$, $L_0$, $L_0 \leq
M_0$, with $ \frac{d_0}{\rho_0}$, $ \frac{r_0}{\rho_0}$ only
depending on $M_0$ and $L_0$ only depending on $\alpha$ and $M_0$,
such that if
\begin{equation}
  \label{eq:9.2}
   d_H(\partial D_1, \partial D_2) \leq d_0,
\end{equation}
then $\partial \Omega_D$ is Lipschitz with constants $r_0$ and
$L_0$. Moreover, for every $P \in \partial \Omega_D \cap \partial
D_1$, up to a rigid transformation of coordinates which maps $P$ into
the origin and $e_3=-\nu$, where $\nu$ is the outer unit normal to
$D_1$ at $P$, we have
\begin{equation}
  \label{eq:9BIS.1}
   D_i \cap B_{r_0}(P) = \left \{ x \in B_{r_0}(0) | \ x_3>
   \varphi_i(x')\right \}, \ \ i=1,2,
\end{equation}
\begin{equation}
  \label{eq:9BIS.2}
   \varphi_1(0)=0, \quad \nabla \varphi_1(0)=0,
\end{equation}
\begin{equation}
  \label{eq:9BIS.3}
   \| \varphi_i \|_{C^{0,1}(B_{r_0}'(0))} \leq L_0 r_0, \ \ i=1,2.
\end{equation}
An analogous representation holds for every $P \in \partial
\Omega_D \cap \partial D_2$.
\end{theo}
\begin{proof} [Proof of Lemma \ref{lem:8.1}] Let
\begin{equation}
  \label{eq:10.1}
   d_1 = \frac{d_0}{c_0},
\end{equation}
where $c_0$ is the constant introduced in Lemma \ref{lem:6.1}, and
let
\begin{equation}
  \label{eq:10.2}
   d_2 = \min \{ d_1, h_0 \rho_0 \},
\end{equation}
where $h_0$, $0<h_0<1$, only depending on $M_0$, has been
introduced in Lemma \ref{lem:9.1}. We shall distinguish two cases.

\textit{Case i) Let $d_\mu \leq d_1$.}

Then, by Lemma
\ref{lem:6.1} we have $d_H(\partial D_1, \partial D_2)\leq d_0$.
Therefore, by Theorem \ref{theo:9.2}, $\partial \Omega_D$ is
Lipschitz with constants $r_0$, $L_0$, where $ \frac{r_0}{\rho_0}$
only depends on $M_0$, and $L_0$ only depends on $M_0$ and
$\alpha$. We may apply Lemma \ref{lem:9.1} to $\R^3 \setminus
\Omega_D$ obtaining that there exists $ \widetilde{h}_0$,
$0<\widetilde{h}_0<1$, only depending on $\alpha$ and $M_0$, such
that $(\R^3 \setminus \Omega_D)_{hr_0}$ is connected for every $h
\leq\widetilde{h}_0$.

Let $P \in \partial D_1 \cap \partial \Omega_D$ be such that
\begin{equation}
  \label{eq:10.3}
   d_\mu(D_1, D_2) =\hbox{dist}(P,
   D_2).
\end{equation}
Under the coordinate system introduced in Theorem \ref{theo:9.2},
let us consider the point $Q=P - \frac{\widetilde{h}_0
r_0}{2}e_3$. We have that
\begin{equation}
  \label{eq:11.1}
   \hbox{dist}(Q, \Omega_D) \geq \frac{\widetilde{h}_0
r_0}{2 \sqrt{1+L_0^2}  }.
\end{equation}
Let us denote $h_1 = \frac{\widetilde{h}_0}{2 \sqrt{1+L_0^2}  }$. Since $h_1 <
\widetilde{h}_0$, the set $\overline{(\R^3 \setminus
\Omega_D)_{h_1 r_0}}$ is
connected and contains $Q$. Therefore, there exists a path $\gamma
\subset \overline{(\R^3 \setminus
\Omega_D)_{h_1 r_0}}$ joining any point $P_0 \in S_{2\rho_0}$ with
$Q$. Therefore, in the above coordinate system, the set $V(\gamma)$ satisfies
\begin{equation}
  \label{eq:11.2}
   V(\gamma) \subset \R^3 \setminus \Omega_D,
\end{equation}
provided
\begin{equation}
  \label{eq:11.3}
   d= \frac{\widetilde{h}_0 r_0}{2}, \quad R=\frac{d}{\sqrt{1+L_0^2}
   }.
\end{equation}

\textit{Case ii) Let $d_\mu \geq d_1$.}

Then, trivially, $d_\mu
\geq d_2$. Let $ \widetilde{P} \in \partial D_1 \cap \partial
\Omega_D$ be such that
\begin{equation}
  \label{eq:12.1}
   d_\mu (D_1, D_2) = \hbox{dist}(\widetilde{P}, D_2).
\end{equation}
Since $d_2 \leq h_0 \rho_0$, by Lemma \ref{lem:9.1}, $(\R^3
\setminus D_2)_{d_2}$ is connected. Therefore, given any point
$P_0 \in S_{2\rho_0}$, there exists a path $\gamma$, $\gamma:[0,1]
\rightarrow (\R^3 \setminus D_2)_{d_2}$ such that $\gamma(0) \in
S_{2\rho_0}$ and $\gamma(1) = \widetilde{P}$. Let $ \overline{t} =
\inf_{t\in [0,1]} \left \{ t| \ {dist}(\gamma(t), \partial D_1)
> \frac{d_2}{2} \right \}$. By definition, $dist(\gamma(
\overline{t}), \partial D_1) = \frac{d_2}{2}$, so that there
exists $P \in \partial D_1$ satisfying $|P- \gamma (
\overline{t})|= \frac{d_2}{2}$. We have that
\begin{equation}
  \label{eq:12.2}
   \hbox{dist}(P,D_2) \geq \hbox{dist}( \gamma ( \overline{t}),
   D_2) - |\gamma ( \overline{t})-P| \geq d_2 - \frac{d_2}{2} =
   \frac{d_2}{2}.
\end{equation}
Let $ \overline{\gamma} = \gamma|_{[0, \overline{t}]}$ and let us
choose a cartesian coordinate system  with origin $O$ at $P$, and
$e_3=-\nu$, where $\nu$ is the outer unit normal to $D_1$ at $P$.
We have that
\begin{equation}
  \label{eq:12.3}
   V(\overline{\gamma}) \subset \R^3 \setminus \Omega_D,
\end{equation}
assuming
\begin{equation}
  \label{eq:12.4}
   d= \frac{d_2}{2}, \quad R=\frac{d}{\sqrt{1+M_0^2}
   }.
\end{equation}
Let
\begin{equation}
  \label{eq:12BIS.1}
   \overline{d}=\min \left \{ \frac{\widetilde{h}_0 r_0}{2}, \frac{d_0}{2c_0}, \frac{h_0
   \rho_0}{2}\right \},
\end{equation}
and let us notice that $ \frac{\overline{d}}{\rho_0}$ only depends on $M_0$,
$\alpha$.
Observing that $L_0 \leq M_0$, formula \eqref{eq:8BIS.4}
follows with $\overline{d}$ given in \eqref{eq:12BIS.1}.
Since there exists a positive constant $C$ only
depending on $M_0$, $M_1$ such that $diam(\Omega) \leq C\rho_0$,
we have that
\begin{equation}
  \label{eq:12BIS.2}
   d_\mu \leq \left ( \frac{\hbox{diam}(\Omega)    }{ \frac{d_2}{2}
   }\right ) \frac{d_2}{2} \leq \widetilde{c}_1 \frac{d_2}{2},
\end{equation}
with $\widetilde{c}_1$ only depending on $M_0$, $\alpha$ and
$M_1$. Letting $ c_1 = \min \left \{1, \frac{1}{\widetilde{c}_1}
\right \}$, inequality \eqref{eq:8BIS.2} follows.

\end{proof}
\bigskip

\textit{Acknowledgements}.
The collaboration of
Professor Alessandro Logar in preparing the numerical simulations
of the last section by means of the open source software package
Sage is gratefully acknowledged.

The second and the third author began to work on this topic during a visit at the Department of Mathematics of
Hokkaido University. They wish to thank Professor Gen Nakamura for
supporting their visit and for the warm hospitality in Sapporo.

\end{document}